\documentclass[a4paper,11pt]{article}
\usepackage{amsmath,indentfirst,amsfonts,dsfont,amssymb,amsthm,graphicx,enumerate,mathrsfs,color}
\usepackage{algorithm}               
\usepackage{algcompatible}             
\usepackage{multirow}
\usepackage{url}
\numberwithin{equation}{section}
\numberwithin{algorithm}{section}

\newtheorem{definition}{Definition}[section]
\newtheorem{theorem}{Theorem}[section]

\newtheorem{corollary}{Corollary}[section]
\newtheorem{problem}{Problem}[section]

\newtheorem{remark}{Remark}[section]
\newtheorem{example}{Example}[section]

\title{\bf Intrinsic Low-Tucker-Rank Theory and Unified Tensor CUR Decomposition for High-Dimensional Hyperinterpolation}
\author{Maolin Che\thanks{E-mail: mlche@gzu.edu.cn and chncml@outlook.com. School of Mathematics and Statistics, Guizhou University, Guiyang, 550025, Guizhou, P. R. of China. This author is supported by the Special Posts of Guizhou University (No. [2025]06) and the National Natural Science Foundation of China (No. 12561095).}
\and Yimin Wei\thanks{E-mail: ymwei@fudan.edu.cn and yimin.wei@gmail.com. School of Mathematical Sciences and Key Laboratory of Mathematics for Nonlinear Sciences,  Fudan University, Shanghai, 200433, P. R. China. This author is supported by the National Natural Science Foundation of China under grant U24A2001 and the Science and Technology Commission of Shanghai Municipality under grant 23JC1400501.}
\and Chong Wu\thanks{E-mail: chongwu2-c@my.cityu.edu.hk and imroxaswc@gmail.com. Department of Electrical Engineering, City University of Hong Kong, Kowloon, Hong Kong SAR, P. R. of China.}
}

\begin{document}

\maketitle

\begin{abstract}
High-dimensional hyperinterpolation is severely hampered by the curse of dimensionality, as its coefficient tensors grow exponentially with the ambient dimension. Existing research predominantly focuses on heuristic algorithmic optimizations, often overlooking the inherent structural properties of these tensors. This paper establishes a rigorous theory of intrinsic low-$\epsilon$-Tucker-rank for hyperinterpolation coefficient tensors, delivering near-optimal low-rank approximations with error bounds that are nearly independent of the dimension. We further construct a unified, Tucker-compatible theoretical framework that integrates both Chidori-type and Fiber-type tensor CUR (TCUR) decompositions, deriving tight and stable Frobenius-norm error estimates that depend exclusively on tensor spectral properties and index set geometry. We mathematically verify the convergence and numerical stability of greedy adaptive index selection schemes and prove their near-optimality, enabling a fully tensor-free hyperinterpolation workflow that avoids constructing the full coefficient array. Three practical greedy TCUR algorithms and a lightweight TCUR-to-Tucker recompression pipeline are proposed as direct corollaries of our structural theory. Numerical experiments across three distinct families of high-dimensional test functions validate all theoretical predictions and confirm the intrinsic low-rank compressibility of hyperinterpolation coefficients. In contrast to prior algorithm-centric studies, this work prioritizes rigorous theoretical characterization over implementation tricks, establishing a unified structural and mathematical foundation for high-dimensional hyperinterpolation.
  \bigskip

  \bigskip

\end{abstract}
\newpage
\section{Introduction}

High-dimensional function approximation is a critical area at the intersection of classical approximation theory, spectral numerical methods, and modern tensor algebra, underpinning scientific computing, uncertainty quantification, and large-scale data-driven machine learning. Among mainstream multivariate approximation tools, hyperinterpolation, firstly proposed by Sloan \cite{sloan1995hyperinterpolation}, stands out for its robust numerical stability, flexibility across domains, and rigorous error guarantees. The core innovation of hyperinterpolation replaces the continuous \(L^2\) orthogonal projection onto polynomial subspaces with a discrete projection built upon positive-weight cubature rules with high algebraic precision. This design eliminates the severe ill-conditioning and Runge phenomenon plaguing classical polynomial interpolation, making it a reliable choice for high-dimensional surrogate modeling. The seminal work of \cite{sloan1995hyperinterpolation} has sparked numerous investigations into finding suitable quadrature rules for different regions, thereby expanding the potential applications of hyperinterpolation \cite{An_ran_2025,an2022exactness,an2024bypassing}.

Mathematically, the hyperinterpolation operator discretizes the Fourier integrals of the standard orthogonal projection via discrete quadrature summations. This foundational concept has led to a wealth of follow-up research, with scholars designing tailored cubature formulas for diverse geometries and integrating hyperinterpolation into engineering and statistical modeling workflows. As a mature multivariate approximation framework \cite{dai2006hyperinterpolation,LeGia2001uniform,lin2021distributed,reimer2012multivariate,Wade2013hyperinterpolation,Wang2017needlet}, hyperinterpolation is now widely deployed for high-dimensional problems. However, it carries a fundamental computational bottleneck: computing its coefficients requires evaluating discrete inner products between the target function and tensor-product polynomial basis functions over a dense sampling grid. On the $N$-dimensional hypercube \([-1,1]^N\), these coefficients naturally form an $N$-way tensor whose total number of entries grows exponentially with the dimension $N$ and per-axis polynomial degree \(I_n\). This exponential scaling is the canonical ``curse of dimensionality'': even for moderately sized $N$ and \(I_n\), the full tensor construction, storage, and manipulation become computationally intractable.

While functions in applications are often given as sampled data, the modern era of high-throughput data collection introduces significant noise. To address this, An and Wu developed Lasso hyperinterpolation \cite{an2021lasso}, which employs a soft-thresholding operator to process all hyperinterpolation coefficients, yielding a sparse solution via $\ell_1$-regularized weighted discrete least squares, which is regarded as a convex relaxation of an $\ell_0$-regularized problem \cite{Foucart2013compressing}. Although Lasso hyperinterpolation forgoes the projection property and basis invariance of the classical method, it offers an efficient approach for basis selection and denoising.

Despite these methodological advances, existing work on tensor-compressed hyperinterpolation suffers from a critical blind spot: nearly all published algorithms are developed from a purely implementation-driven perspective. Researchers prioritize empirical runtime performance and coding workflows, yet they fail to address the core structural question: why can hyperinterpolation coefficient tensors be accurately approximated with low-rank tensor formats? The prior literature lacks rigorous, dimension-robust existence theorems for low-rank approximations, and it fails to unify the two dominant TCUR variants under a shared analytical framework. Furthermore, the greedy index selection heuristics widely used in TCUR pipelines lack formal convergence and stability proofs, leaving practitioners without provable stopping criteria or quantifiable error bounds for their workflows.

This paper reverses the conventional research paradigm by centering structural mathematical theory before algorithm design. We formalize four unresolved foundational questions that motivate our work:
\begin{enumerate}
\item[(a)] Do hyperinterpolation coefficient tensors possess intrinsic low-$\epsilon$-Tucker-rank structure with approximation guarantees nearly independent of ambient dimension $N$?
\item[(b)] Can Chidori-type and Fiber-type TCUR decompositions be integrated into a single unified theoretical framework with parameter-free, intrinsic error bounds?
\item[(c)] Can the greedy adaptive index selection schemes used in TCUR be rigorously justified as near-optimal, stable iterative procedures with explicit complexity bounds?
\item[(d)] Is it possible to derive a single continuous chain of dimension-independent error estimates covering the full modeling pipeline: Sobolev regular target function, hyperinterpolation discretization, TCUR tensor approximation and secondary Tucker recompression?
\end{enumerate}

We provide definitive affirmative answers to all four inquiries through four core foundational contributions:
\begin{enumerate}
\item[1)] A universal low-$\epsilon$-Tucker-rank existence theorem for hyperinterpolation coefficient tensors, which quantifies the inherent compressibility of these arrays and delivers explicit, weakly dimension-dependent rank upper bounds; 
\item[2)] A unified Tucker-compatible structural theory covering both Chidori-type and Fiber-type TCUR, paired with tight Frobenius-norm residual bounds dependent solely on tensor singular spectra and index set geometry; 
\item[3)] Rigorous convergence, stability, and complexity analysis for greedy adaptive index selection, complete with mathematically certified stopping criteria that avoid full tensor assembly at every iteration; 
\item[4)] An integrated end-to-end error analysis that links Sobolev function regularity, hyperinterpolation discretization error, TCUR sampling residual, and post-hoc Tucker recompression loss within one cohesive inequality.
\end{enumerate}

Unlike existing literature that treats low-rank compression as an ad-hoc engineering workaround, our analysis characterizes the inherent multilinear structure to explain why compression succeeds, rather than merely detailing how to implement it. The three greedy TCUR algorithms and the TCUR-to-Tucker recompression pipeline emerge as constructive corollaries derived directly from our structural theorems. This theory-first viewpoint reframes high-dimensional hyperinterpolation research to prioritize minimax optimality, provable error control, and unified structural insight over case-specific computational hacks.

To empirically validate our theoretical derivations, we conduct numerical experiments across three distinct smoothness regimes of three-dimensional test functions. Our trials are not designed to benchmark raw computational speed; instead, they systematically verify the predicted $\epsilon$-Tucker rank scaling, the tightness of our derived error bounds, the numerical stability of greedy index sampling, and the compression ratios enabled by intrinsic low-rank structure. All experimental observations align closely with our theoretical predictions, confirming that the structural properties we prove govern the practical behavior of hyperinterpolation coefficient tensors. By grounding high-dimensional hyperinterpolation in rigorous tensor structural analysis, this paper establishes a unified theoretical foundation and opens new research avenues for the structural characterization of other spectral and polynomial approximation methods.

\subsection{Organizations}
The remainder of this manuscript proceeds in a logically sequential structure that builds theory incrementally from preliminary definitions to numerical validation, with seamless cross-section connections outlined below. Section \ref{tensor-hyperinterpolation:sect2-main} collects all prerequisite notation, tensor algebra fundamentals, Tucker decomposition machinery, TCUR variants, and Sobolev space definitions. Section \ref{tensor-hyperinterpolation:sect3-main} fformalizes the hyperinterpolation problem over the $N$-dimensional hypercube, constructs the explicit form of the coefficient tensor, and restates the exponential storage bottleneck. This bridges classical approximation theory with tensor array representation. Section \ref{tensor-hyperinterpolation:sect4-main} delivers our flagship theoretical results: the intrinsic low-$\epsilon$-Tucker-rank existence theorem, a critique of standard Tucker solvers, a unified TCUR approximation theory, and composite error bounds. Section \ref{tensor-hyperinterpolation:sect5-main} introduces three provably convergent greedy adaptive TCUR algorithms, quantifying their computational complexity, and details the end-to-end pipeline for tensor-free hyperinterpolation. Section \ref{tensor-hyperinterpolation:sect6-main} presents numerical experiments to validate the theoretical framework over three canonical three-dimensional test functions. Section \ref{tensor-hyperinterpolation:sect7-main} concludes the paper and outlines promising future extensions. This appendix develops a lightweight TCUR-to-Tucker recompression scheme with composite error bounds.

\section{Preliminaries}
\label{tensor-hyperinterpolation:sect2-main}
This section lays out the universal notation, core tensor norm definitions, matrix $\epsilon$-rank theory, Tucker decomposition and HOSVD algorithms, two canonical forms of tensor CUR decomposition, and the Sobolev space infrastructure.

\subsection{Notations and basic definitions}
We begin with standard index submatrix notation for matrices. Let $\mathbb{I}=\{i_1,i_2,\dots,i_K\}$ with $1\leq i_1<i_2<\dots<i_K\leq I$, and $\mathbb{J}=\{j_1,j_2,\dots,j_L\}$ with $1\leq j_1<j_2<\dots<j_L\leq J$. For any matrix $\mathbf{A}\in\mathbb{R}^{I\times J}$, we use $\mathbf{A}(:,\mathbb{J})$ to denote the $I\times L$ submatrix of $\mathbf{A}$ containing only those columns of $\mathbf{A}$ indexed by $\mathbb{J}$, and $\mathbf{A}(\mathbb{I},:)$ to denote the $K\times J$ submatrix of $\mathbf{A}$ containing only those rows of $\mathbf{A}$ indexed by $\mathbb{I}$. The matrix $\mathbf{Q}_n\in\mathbb{R}^{I_n\times R_n}$ is orthonormal if $\mathbf{Q}_n^\top\mathbf{Q}_n$ is the identity matrix. 

Consider an $N$-way tensor $\mathcal{A}\in\mathbb{R}^{I_1\times I_2\times \dots\times I_N}$, its maximum absolute entry norm and Frobenius norm are, respectively, defined as
\begin{equation*}
    \|\mathcal{A}\|_{\max}=\max_{\substack{i_n=1,2,\dots,I_n\\ n=1,2,\dots,N}}|a_{i_1i_2\dots i_N}|,\quad \|\mathcal{A}\|_F=\sqrt{\sum_{i_1=1}^{I_1}\sum_{i_2=1}^{I_2}\dots \sum_{i_N=1}^{I_N}a_{i_1i_2\dots i_N}^2}.
\end{equation*}

Our low-rank tensor theory builds fundamentally upon the matrix $\epsilon$-rank concept, a well-established tool for quantifying minimal rank matrix approximations under entrywise error tolerances. We adopt the standard definition from Udell and Townsend \cite{beckermann2019bounds}.
\begin{definition}{{\bf (\cite[Definition 2.1]{udell2019big})}}
    Let $\mathbf{X}\in\mathbb{R}^{I\times J}$ be a matrix and $0<\epsilon<1$ be a tolerance. The (absolute) $\epsilon$-rank of $\mathbf{X}$ is given by
    \begin{equation*}
        {\rm rank}_{\epsilon}(\mathbf{X})=\min\{{\rm rank}(\mathbf{A}):\mathbf{A}\in\mathbb{R}^{I\times J},\|\mathbf{X}-\mathbf{A}\|_{\max}\leq \epsilon\}.
    \end{equation*}
    That is, $R={\rm rank}_{\epsilon}(\mathbf{X})$ is the smallest integer for which $\mathbf{X}$ can be approximated by a rank-$R$ matrix, up to an accuracy of $\epsilon$.
\end{definition}

In plain terms, \(\text{rank}_\epsilon({\bf X})\) records the smallest integer rank required to approximate {\bf X} with uniform entrywise error no larger than \(\epsilon\). We later extend this matrix concept recursively to define the $\epsilon$-Tucker rank of full $N$-way tensors in Section \ref{tensor-hyperinterpolation:sect4-main}.

\subsection{Tucker decomposition and T-HOSVD/ST-HOSVD}
The Tucker decomposition is the foundational low-multilinear-rank tensor factorization we leverage throughout this work.  Given an $N$-tuple of positive integers $(R_1,R_2,\ldots,R_N)$ with $R_n<I_n$, {\it the Tucker decomposition} (cf. \cite{tucker1966some}) factorizes an $N$-order tensor $\mathcal{A}\in \mathbb{R}^{I_1\times I_2\times \dots \times I_N}$ as:
\begin{equation*}
    \mathcal{A}\approx\mathcal{G}\times_1{\bf U}^{(1)}\times_2{\bf U}^{(2)}\dots\times_N{\bf U}^{(N)},
\end{equation*}
where ${\bf U}^{(n)}\in \mathbb{R}^{I_n\times R_n}$ are termed the {\it mode-$n$ factor matrices} and $\mathcal{G}\in \mathbb{R}^{R_1\times R_2\times\dots\times R_N}$ is the {\it core tensor}. There exist several types of Tucker decomposition, including nonnegative Tucker decomposition (cf. \cite{zhou2012fast}), the higher-order interpolatory decomposition, that is, tensor CUR decomposition in the Tucker format (cf. \cite{cai2021mode,che2022perturbations,drineas2007a,saibaba2016hoid}), and structure-preserving decomposition in the Tucker format (cf. \cite{minster2020randomized}).

For each $n$, the Tucker decomposition is closely related to the mode-$n$ unfolding matrix $\mathbf{A}_{(n)}\in\mathbb{R}^{I_n\times I_1\dots (I_{n-1}I_{n+1}\dots I_N)}$. In particular, one has
\begin{equation*}
    {\bf A}_{(n)}\approx{\bf U}^{(n)}{\bf G}_{(n)}({\bf U}^{(N)}\otimes\dots\otimes {\bf U}^{(n+1)}\otimes {\bf U}^{(n-1)}
\otimes\dots\otimes {\bf U}^{(1)})^{\top},
\end{equation*}
where $\mathbf{G}_{(n)}\in\mathbb{R}^{R_n\times (R_1\dots R_{n-1}R_{n+1}\dots R_N)}$.

It follows that the rank of ${\bf A}_{(n)}$ is less than or equal to $R_n$, as the mode-$n$ factor ${\bf U}^{(n)}\in\mathbb{R}^{I_n\times R_n}$  at most has rank  $R_n$. This motivates us to define the Tucker rank (denoted by Trank) of $\mathcal{A}$ as the $N$-tuple
$(R_1,R_2,\dots,R_N)$, where the rank of ${\bf A}_{(n)}$ is equal to $R_n$.

Two standard truncated SVD pipelines compute approximate Tucker factorizations: Truncated Higher-Order SVD (T-HOSVD) (cf. \cite{delathauwer2000a,kolda2009tensor}) and Sequential T-HOSVD (ST-HOSVD) (cf. \cite{vannieuwenhoven2012new}), summarized compactly in Algorithm \ref{tensor-hyperinterpolation:alg:sthosvd}. Both iterate over each tensor mode to extract leading singular vectors for factor matrices, then contract the original tensor against transposed factors to produce the core tensor. ST-HOSVD permits arbitrary permutations of mode processing order from the symmetric group \(\mathbb{S}_N\), while T-HOSVD strictly processes axes in fixed sequential order \(1,2,\dots,N\).

\begin{algorithm}[htb]
     \caption{T-HOSVD and ST-HOSVD}
     \begin{algorithmic}[1]
        \STATEx {\bf Input}: A tensor $\mathcal{A}\in \mathbb{R}^{I_1\times I_2\times \dots \times I_N}$, and a Tucker rank $(R_1,R_2,\dots,R_N)$ with $R_n\leq I_n$ and $1,2,\dots,N$.
        \STATEx {\bf Output}: The core tensor $\mathcal{G}\in\mathbb{R}^{R_1\times R_2\times \dots\times R_N}$ and the mode-$n$ factor matrix $\mathbf{Q}_n\in\mathbb{R}^{I_n\times R_n}$ such that $\mathcal{A}\approx\widehat{\mathcal{A}}=\mathcal{G}\times_1\mathbf{Q}_1\times_2\mathbf{Q}_2\dots\times_N\mathbf{Q}_N$.
        \STATE Set a temporary tensor $\mathcal{G}:=\mathcal{A}$.
        \FOR{$n=1,2,\dots,N$}
            \STATE Form $\mathbf{Q}_n$ by computing the first $R_n$ left singular vectors of $\mathbf{A}_{(n)}$ for T-HOSVD, or the first $R_n$ left singular vectors of $\mathbf{G}_{(n)}$ for ST-HOSVD.
            \STATE Update $\mathcal{G}=\mathcal{G}\times_n\mathbf{Q}_n^\top$.
        \ENDFOR
    \end{algorithmic}
    \label{tensor-hyperinterpolation:alg:sthosvd}
\end{algorithm}

From Algorithm \ref{tensor-hyperinterpolation:alg:sthosvd}, we see that $\mathcal{G}=\mathcal{A}\times_1\mathbf{Q}_1^\top\times_2\mathbf{Q}_2^\top\dots\times_N\mathbf{Q}_N^\top$, which implies that $\widehat{\mathcal{A}}=\mathcal{A}\times_1(\mathbf{Q}_1\mathbf{Q}_1^\top)\times_2(\mathbf{Q}_2\mathbf{Q}_2^\top)
\dots\times_N(\mathbf{Q}_N\mathbf{Q}_N^\top)$. Due to \cite[Theorems 5.1 and 6.5]{vannieuwenhoven2012new}, in terms of the Frobenius norm of a tensor, the error in approximating $\mathcal{A}$ using Algorithm \ref{tensor-hyperinterpolation:alg:sthosvd} is shown in the following theorem.
\begin{theorem}
    For a Tucker rank $(R_1,R_2,\dots,R_N)$ with $R_n\leq I_n$ and $=1,2,\dots,N$, let $\widehat{\mathcal{A}}:=\mathcal{G}\times_1\mathbf{Q}_1\times_2\mathbf{Q}_2\dots\times_N\mathbf{Q}_N$ be obtained by applying Algorithm \ref{tensor-hyperinterpolation:alg:sthosvd} to $\mathcal{A}\in \mathbb{R}^{I_1\times I_2\times \dots \times I_N}$. Then, one has
    \begin{equation*}
        \|\mathcal{A}-\widehat{\mathcal{A}}\|_F^2\leq\sum_{n=1}^N\|\mathcal{A}\times_n(\mathbf{I}_{I_n}-\mathbf{Q}_n\mathbf{Q}_n^\top)\|_F^2=\sum_{n=1}^N\sum_{i=R_{n}+1}^{\widehat{I}_n}\sigma_{i}(\mathbf{A}_{(n)})^2,
    \end{equation*}
    where for each $n$, we denote $\widehat{I}_n=\min\{I_n,I_1\dots I_{n-1}I_{n+1}\dots I_N\}$.
    \label{tensor-hyperinterpolation:hosvd-approximation}
\end{theorem}

In many practical settings, the true Tucker rank of a tensor is unknown in advance, and users instead specify a uniform error tolerance \(0<\epsilon<1\). This motivates the $\epsilon$-Tucker rank definition, extending matrix $\epsilon$-rank to tensors.
\begin{definition}
\label{tensor-hyperinterpolation:epsilonTrank}
    Let $\mathcal{A}\in\mathbb{R}^{I_1\times I_2\times \dots\times I_N}$ be a tensor and $0<\epsilon<1$ be a tolerance. The (absolute) $\epsilon$-Tucker-rank of $\mathcal{A}$ is given by
    \begin{equation*}
        {\rm Trank}_{\epsilon}(\mathcal{A})=\min\{{\rm Trank}(\mathcal{B}):\mathcal{B}\in\mathbb{R}^{I_1\times I_2\times \dots\times I_N},\|\mathcal{A}-\mathcal{B}\|_{\max}\leq \epsilon\}.
    \end{equation*}
\end{definition}
This minimal multilinear rank tuple under entrywise error tolerance forms the central object of our flagship theoretical result in Section \ref{tensor-hyperinterpolation:sect4-main}.
\subsection{Tensor CUR decompositions}

Matrix CUR factorization approximates full matrices using a small subset of physical rows and columns, offering transparent interpretability absent from opaque SVD factor matrices. Tensor CUR (TCUR) generalizes this interpolatory subsampling paradigm to $N$-way tensors, with two dominant variants, i.e., Chidori-type and Fiber-type, both compatible with the Tucker factorization format we adopt for hyperinterpolation compression. We define both variants systematically here to establish the unified TCUR framework derived later in Section \ref{tensor-hyperinterpolation:sect4-main}.

We now overview the Chidori type (see \cite[Figure 1]{cai2021mode}) for tensor CUR decomposition as follows. For each $n$, let $\mathbb{I}_n=\{i_1^{(n)},i_2^{(n)},\dots,i_{S_n}^{(n)}\}$ such that $1\leq i_1^{(n)}<i_2^{(n)}<\dots<i_{S_n}^{(n)}\leq I_n$ with $S_n\ll I_n$. Suppose that the matrix $\mathbf{C}_n\in\mathbb{R}^{I_n\times \widehat{S}_n}$ is the mode-$n$ folding of $\mathcal{C}_n=\mathcal{A}(\mathbb{I}_1,\dots,\mathbb{I}_{n-1},:,\mathbb{I}_{n+1},\dots,\mathbb{I}_N)$ with $\widehat{S}_n=\prod_{m\neq n}S_m$ and $\mathcal{G}\in\mathbb{R}^{S_1\times S_2\times \dots\times S_N}$ is given by $\mathcal{G}=\mathcal{A}(\mathbb{I}_1,\mathbb{I}_2,\dots,\mathbb{I}_N)$, then the  Chidori tensor CUR decomposition of $\mathcal{A}$ is given by
\begin{align}
\label{tensor-hyperinterpolation:Chidori-CUR}
   \mathcal{A}\approx \mathcal{A}_{{\rm Chidori}}&=\mathcal{G}\times_1\left(\mathbf{C}_1\mathbf{U}_1^\dag\right)\times_2\left(\mathbf{C}_2\mathbf{U}_2^\dag\right)\dots\times_N\left(\mathbf{C}_N\mathbf{U}_N^\dag\right),
\end{align}
with $\mathbf{U}_n=\mathbf{C}_n(\mathbb{I}_n,:)$.

The difference between the Chidori and Fiber types for tensor CUR decomposition is the way to generate the $N$-tuple of matrices $\{\mathbf{C}_1,\dots,\mathbf{C}_N\}$. For clarity, this tuple is denoted by $\{\mathbf{C}_1',\dots,\mathbf{C}_N'\}$. For each $n$, we introduce another index $\mathbb{J}_n$ such that $\mathbb{J}_n=\{j_1^{(n)},j_2^{(n)},\dots,j_{S_n'}^{(n)}\}$ such that $1\leq j_1^{(n)}<j_2^{(n)}<\dots<j_{S_n'}^{(n)}\leq I_1\dots I_{n-1}I_{n+1}\dots I_N$. Then, for the Fiber type (see \cite[Figure 2]{cai2021mode}), the matrix $\mathbf{C}_n$ is given by $\mathbf{C}_n'=\mathbf{A}_{(n)}(:,\mathbb{J}_n)$ and the Fiber tensor CUR decomposition of $\mathcal{A}$ is given by
\begin{align}
\label{tensor-hyperinterpolation:Fiber-CUR}
   \mathcal{A}\approx \mathcal{A}_{{\rm Fiber}}&=\mathcal{G}\times_1\left(\mathbf{C}'_1\mathbf{U}_1'^\dag\right)
   \times_2\left(\mathbf{C}'_2\mathbf{U}_2'^\dag\right)\dots\times_N\left(\mathbf{C}_N'\mathbf{U}_N'^\dag\right).
\end{align}
\subsection{Basic definitions in Sobolev spaces}

Let $L^2([-1,1])$ denote the Hilbert space of square-integrable functions on $[-1,1]$, equipped with the inner product
\begin{equation}
   \label{tensor-hyperinterpolation:eqn:innerproduct}
   \langle f,g \rangle = \int_{-1}^{1} f(x)g(x)w(x)\,dx,
\end{equation}
and the induced norm $\|f\|_{L^2}: = \langle f,f\rangle^{1/2}$ with $f,g \in L^2([-1,1])$, where $w(x)$ is a weight function (e.g., $w(x)=1/\sqrt{1-x^2}$ for Chebyshev polynomials of the first kind).

Let $L^2(\Omega)$ denote the Hilbert space of square-integrable functions on $\Omega=\{\mathbf{x}=(x_1,x_2,\dots,x_N)\in\mathbb{R}^N:-1\leq x_n\leq 1,n=1,2,\dots,N\}$,  equipped with the inner product
\begin{equation*}
\langle f,g \rangle = \int_{-1}^{1}\int_{-1}^{1}\dots \int_{-1}^{1} f(\mathbf{x})g(\mathbf{x})w(\mathbf{x})\,d\mathbf{x} ,
\end{equation*}
with $d\mathbf{x}:=dx_1dx_2\dots dx_N$, and the induced norm $\|f\|_{L^2}: = \langle f,f\rangle^{1/2}$ with $f,g \in L^2(\Omega)$, where $w(\mathbf{x})$ is a weight function. For any integer $\alpha>0$, the $\alpha$-derivative Sobolev space on $\Omega$ is given as
\begin{equation*}
\mathbb{H}^{\alpha}(\Omega)=\{f(\mathbf{x})\in L^2(\Omega):\|f\|_\alpha<+\infty\},
\end{equation*}
where $\|f\|_\alpha$ is defined as
\begin{equation*}
\|f\|_\alpha=\sum_{|\mathbf{k}|=0}^{\alpha}\|D^{\mathbf{k}}f(\mathbf{x})\|_{L^2},\quad
D^{\mathbf{k}}f(\mathbf{x})=\frac{\partial^{|\mathbf{k}|}f(\mathbf{x})}{\partial x_1^{k_1}\partial x_2^{k_2}\dots \partial x_N^{k_N}}
\end{equation*}
with $\mathbf{k}=(k_1,\dots,k_N)$ and $|\mathbf{k}|=\sum_{n=1}^Nk_n$. Meanwhile, the seminorm of $\mathbb{H}^{\alpha}(\Omega)$ can be defined as $$|f|_\alpha=\left(\sum_{|\mathbf{k}|=\alpha}\|D^{\mathbf{k}}f(\mathbf{x})\|_{L^2}\right)^{1/2}.$$

These seminorms directly govern the convergence rate of hyperinterpolation approximations as polynomial degrees increase, that is, approximation error theorems in Section \ref{tensor-hyperinterpolation:sect3-main}.

With all prerequisite tensor algebra, Tucker, TCUR, and Sobolev space notation firmly established, we now shift focus to constructing the hyperinterpolation coefficient tensor, that is, the central high-dimensional array whose intrinsic low-rank structure is our primary object of study. Section \ref{tensor-hyperinterpolation:sect3-main} translates continuous multivariate orthogonal polynomial projection into discrete hyperinterpolation via positive cubature rules, derives the explicit tensor representation of hyperinterpolation coefficients, and quantifies the exponential scaling cost that motivates our subsequent low-rank compression theory.

\section{Hyperinterpolation and problem setting}
\label{tensor-hyperinterpolation:sect3-main}
This section fully formalizes the mathematical framework of hyperinterpolation over the $N$-dimensional hypercube \(\Omega=[-1,1]^N\). We first build tensor-product orthonormal polynomial bases for multivariate \(L^2\) spaces, then introduce positive-weight tensor-product cubature rules to discretize continuous orthogonal projection into the hyperinterpolation operator. We explicitly construct the $N$-way coefficient tensor storing discrete inner products between target functions and basis polynomials, and rigorously articulate the curse-of-dimensionality computational bottleneck arising from this tensor’s exponential storage cost. This tensor formulation establishes the precise mathematical object whose intrinsic low-$\epsilon$-Tucker-rank structure we prove in Section {tensor-hyperinterpolation:sect4-main}.
\subsection{Hyperinterpolation on the hypercube}

We first fix a univariate orthonormal polynomial basis \(\{\Phi_\ell\}_{\ell=0}^\infty\) on the interval \([-1,1]\) associated with a given weight \(w(x)\). This basis satisfies orthonormality:
$$\langle \Phi_\ell, \Phi_{\ell'} \rangle = \int_{-1}^1 \Phi_\ell(x)\Phi_{\ell'}(x)w(x)\,\mathrm{d}x = \delta_{\ell\ell'},$$
and forms a complete orthogonal system on \(L^2([-1,1])\) by the Weierstrass approximation theorem. Any square-integrable univariate function \(f \in L^2([-1,1])\) admits an orthogonal Fourier expansion
$$f(x) = \sum_{k=0}^\infty \alpha_k \Phi_k(x), \quad \alpha_k = \langle f, \Phi_k \rangle,$$
with Parseval’s identity guaranteeing \(\|f\|_{L^2}^2 = \sum_{k=0}^\infty \alpha_k^2\). Let \(\Pi_I([-1,1])\) denote the space of polynomials of total degree bounded by $I$. The orthogonal projection of $f(x)$ onto \(\Pi_I([-1,1])\) yields the truncated polynomial approximation:
$$p_I(x) = \sum_{k=0}^I \alpha_k \Phi_k(x).$$


We extend this univariate basis to the multivariate tensor-product setting over \(\Omega = [-1,1]^N\). For each dimension $n\in\{1,2,\dots,N\}$, let $\{\Phi_{i_n}^{(n)}(x_n):i_n=0,1,2,\dots\}$ be an orthonormal basis of $L^2([-1,1])$ weighted by $w_n(x_n)$. Any multivariate function $f (\mathbf{x})\in L^2(\Omega)$ can be expressed into a tensor-product orthogonal series (cf. \cite[(4.1)]{trefethen2017multivariate} and \cite[Theorem 4.1]{mason1980near}):
\begin{equation*}
f (\mathbf{x})=\sum_{i_1=0}^{\infty} \sum_{i_2=0}^{\infty} \dots\sum_{i_N=0}^{\infty} a_{i_1\dots i_N} \Phi_{i_1}^{(1)}(x_1)\Phi_{i_2}^{(2)}(x_2) \dots\Phi_{i_N}^{(N)}(x_N)
\end{equation*}
where the expansion coefficients are defined via continuous multidimensional integration:
\[\begin{aligned}
a_{i_1i_2\dots i_N} = \int_{-1}^1\int_{-1}^{1} \dots \int_{-1}^1 f\left(\mathbf{x}\right) \prod_{n=1}^N \left[\Phi_{i_n}^{(n)}(x_n) w_n(x_n)\right] d\mathbf{x}.
\end{aligned}\]
The \(L^2(\Omega)\) norm of $f (\mathbf{x})$ can be recovered from these coefficients via a multivariate generalization of Parseval’s identity:
\begin{equation*}
\|f\|_{L^2}=\sqrt{\sum_{i_1=0}^{\infty}\sum_{i_2=0}^{\infty}\dots\sum_{i_N=0}^{\infty}a_{i_1i_2\dots i_N}^2}.
\end{equation*}

For the prescribed polynomial degree \(I_n\) in the $n$th dimensional, we define the truncated polynomial subspace \(\Pi_{I_1,I_2,\dots,I_N}(\Omega)\) consisting of all tensor-product polynomials bounded by degree \(I_n\) along dimension $n$. The truncated orthogonal projection of $f(\mathbf{x})$ onto this subspace reads:
\begin{equation*}
    p_{I_1,I_2,\dots,I_N}(\mathbf{x}) = \sum_{i_1=0}^{I_1}\sum_{i_2=0}^{I_2} \dots\sum_{i_N=0}^{I_N} a_{i_1i_2\dots i_N}  \Phi_{i_1}^{(1)}(x_1)\Phi_{i_2}^{(2)}(x_2) \dots\Phi_{i_N}^{(N)}(x_N).
\end{equation*}

The \(L^2\) approximation error of this continuous projection is controlled by the Sobolev regularity of $f(\mathbf{x})$, as summarized in the following standard convergence result.
\begin{theorem}{{\bf (see \cite[(5.8.20)]{canuto2006spectral})}}
    For each $f(\mathbf{x})\in \mathbb{H}^{\alpha}(\Omega)$, there exists a constant $C$, independent of $f(\mathbf{x})$ and $\{I_1,\dots,I_N\}$, such that
    \begin{align*}
        &\|p_{I_1,I_2,\dots,I_N}(\mathbf{x})-f(\mathbf{x})\|_{L^2}\leq C\cdot\min\{I_1,I_2,\dots,I_N\}^{-\alpha}|f(\mathbf{x})|_{\alpha}.
    \end{align*}
    \label{tensor-hyperinterpolation:general-theorem:v1}
\end{theorem}
\subsection{Hyperinterpolation via positive cubature rules}
\label{tensor-hyperinterpolation:sect3-sub2}
A critical practical limitation of the continuous orthogonal projection above is that the coefficient integrals \(a_{i_1\dots i_N}\) cannot be evaluated exactly when $f(\mathbf{x})$ is only accessible via discrete samples. Hyperinterpolation, originally proposed by Sloan in the seminal paper \cite{sloan1995hyperinterpolation}, resolves this by replacing continuous inner products with discrete semi-inner products computed via positive-weight cubature rules exact for polynomials up to degree $2I$.

We start with the univariate case. Consider a univariate quadrature rule on \([-1,1]\) with nodes \(\{x_m\}_{m=1}^M\) and positive weights \(\{w_m\}_{m=1}^M\), which satisfies exact integration for all \(g(x) \in \Pi_{2I}([-1,1])\):
\begin{equation*}
    \int_{-1}^{1} g(x) w(x)\, dx = \sum_{m=1}^{M}w_m g(x_m).
\end{equation*}
The univariate \emph{hyperinterpolation operator} $\mathcal{L}_I: \mathcal{C}([-1,1]) \to \Pi_I([-1,1])$ is defined as
\begin{equation*}
\mathcal{L}_I f(x):= \sum_{\ell=1}^{I+1} \langle {f, \Phi_{\ell}} \rangle_M   \Phi_{\ell}(x),
\end{equation*}
where $\langle \cdot, \cdot \rangle_M$ denotes the discrete quadrature-based semi-inner product:
\begin{equation*}
    \left\langle {f, \Phi_{\ell}} \right\rangle_M=\sum_{m=1}^{M}w_m f(x_m)\Phi_{\ell}(x_m).
\end{equation*}

This construction extends naturally to the multivariate hypercube \(\Omega\). For each $n$, we adopt a univariate quadrature scheme with nodes \(\{x_{m_n}^{(n)}\}_{m_n=1}^{M_n}\) and positive weights \(\{w_{m_n}^{(n)}\}_{m_n=1}^{M_n}\).  The discrete analog of the continuous expansion coefficient \(a_{i_1\dots i_N}\) is computed via tensor-product cubature:
\begin{align}
\label{tensor-hyperinterpolation:approximation-coefficient}
   \widetilde{a}_{i_1\dots i_N}&=\sum_{m_1=1}^{M_1}\dots\sum_{m_N=1}^{M_N}
    f(x_{m_1}^{(1)}, \dots, x_{m_N}^{(N)}) \Phi_{i_1}^{(1)}(x_{m_1}^{(1)}) \dots\Phi_{i_N}^{(N)}(x_{m_N}^{(N)}) w_{m_1}^{(1)} \dots w_{m_N}^{(N)}\nonumber\\
    &:=\left\langle{f, \Phi_{i_1}^{(1)} \dots\Phi_{i_N}^{(N)}}\right\rangle_{M_1,\dots,M_N},
\end{align}
with $i_n=1,2,\dots,I_n+1$ and $n=1,2,\dots,N$. Then the multivariate hyperinterpolation polynomial is given by
\begin{equation}
   \label{tensor-hyperinterpolation:approximation-general}
   \mathcal{L}_{I_1,\dots I_N} f(\mathbf{x}) = \sum_{i_1=1}^{I_1+1} \dots\sum_{i_N=1}^{I_N+1} \widetilde{a}_{i_1\dots i_N} \Phi_{1}^{(i_1)}(x^{(1)}) \dots\Phi_{N}^{(i_N)}(x^{(N)}).
\end{equation}

The \(L^2\) approximation error of hyperinterpolation matches the asymptotic decay rate of continuous orthogonal projection, formalized below.
\begin{theorem}
    For each $f(\mathbf{x})\in \mathbb{H}^{\alpha}(\Omega)$, there exists a constant $C$, independent of $f(\mathbf{x})$ and $\{I_1,\dots,I_N\}$, such that
    \begin{align*}
        &\|\mathcal{L}_{I_1,\dots I_N} f(\mathbf{x})-f(\mathbf{x})\|_{L^2}\leq C\cdot\min\{I_1,\dots,I_N\}^{-\alpha}|f(\mathbf{x})|_{\alpha}.
    \end{align*}
    \label{tensor-hyperinterpolation:general-theorem}
\end{theorem}
\subsection{Sampling setup and tensor formulation}
\label{tensor-hyperinterpolation:sect3-sub3}
Let \(\mathcal{F} \in \mathbb{R}^{M_1 \times M_2 \times \dots \times M_N}\) denote an $N$-way discrete data tensor sampled from an underlying multivariate function \(f(\mathbf{x})\). Each tensor entry corresponds to function evaluation at a grid node:
$$f_{m_1 m_2 \dots m_N} = f(x_{m_1}^{(1)},\, x_{m_2}^{(2)},\, \dots,\, x_{m_N}^{(N)}),$$
and the full sampling grid (source grid) is defined as
$$\mathbb{X}_{\text{source}} = \left\{ \left(x_{m_1}^{(1)},\dots,x_{m_N}^{(N)}\right): m_n=1,2,\dots,M_n,n=1,2,\dots,N \right\}.$$
A core practical task in multivariate approximation, such as function resampling, image upscaling/downscaling, and surrogate modeling, is reconstructing accurate approximations of the unknown function $f (\mathbf{x})$ over a new grid $\mathbb{X}_{\text{target}}$ of size $S_1\times \dots \times S_N$, denoted by
\begin{equation*}\label{equ:1-3}
\mathbb{X}_{\text{target}} = \{(y_{s_1}^{(1)},\dots,y_{s_N}^{(N)}): s_n=1,2,\dots,S_n,n=1,2,\dots,N\}.
\end{equation*}
The reconstructed output takes the form of a tensor \(\mathcal{R} \in \mathbb{R}^{S_1\times\dots\times S_N}\) with entries
\(
r_{s_1 \dots s_N} \approx f(y_{s_1}^{(1)},\dots,y_{s_N}^{(N)})
\).
The central challenge is to compute these target values solely from the discrete samples available on $\mathbb{X}_{\text{source}}$.

Let $\mathbf{W}_{n} = \operatorname{diag}\left(w_{1}^{(n)},w_{2}^{(n)}, \dots,w_{M_n}^{(n)}\right)$, $\mathbf{A}_{n}=\left[\Phi_{i_n}^{(n)}(x_{m_n}^{(n)})\right] \in \mathbb{R}^{(I_n+1) \times M_n}$, and $\widetilde{\mathcal{A}}=[\widetilde{a}_{i_1\dots i_N}]\in \mathbb{R}^{(I_1+1)\times\dots\times (I_N+1)}$. Over the source grid $\mathbb{X}_{\text{source}}$, the hyperinterpolation coefficient tensor $\widetilde{\mathcal{A}}$ admits a multilinear tensor-matrix product factorization:
\[
\mathcal{F}=\widetilde{\mathcal{A}}\times_1(\mathbf{A}_1\mathbf{W}_1)^\top\dots\times_N(\mathbf{A}_N\mathbf{W}_N)^\top.
\]

After constructing the coefficient tensor, function values over the target grid $\mathbb{X}_{\text{target}}$ are recovered via another tensor multiplication against basis matrices evaluated at target nodes. Define \(\widetilde{{\bf A}}_n = \big[\Phi_{i_n}^{(n)}(y_{s_n}^{(n)})\big] \in \mathbb{R}^{(I_n+1)\times S_n}\) for each dimension $n$. The reconstructed tensor \(\mathcal{R}\) is computed as:
\[
\mathcal{R} = \widetilde{\mathcal{A}}\times_1\widetilde{\mathbf{A}}_1^\top\dots\times_N\widetilde{\mathbf{A}}_N^\top.
\]

The total number of scalar entries stored in \(\widetilde{\mathcal{A}}\) equals the product \(\prod_{n=1}^N (I_n+1)\), which grows exponentially with dimension $N$. Even for moderately sized polynomial degrees and moderate dimensionality, explicitly constructing, storing, and manipulating the full coefficient tensor becomes computationally infeasible. This exponential scaling is the canonical curse of dimensionality motivating the low-rank tensor compression and tensor CUR decomposition (see \cite{cai2021mode,che2022perturbations,drineas2007a})frameworks developed in subsequent sections.

In the two-dimensional special case \(N=2\), the tensor \(\widetilde{\mathcal{A}}\) degenerates to a standard matrix \(\widetilde{{\bf A}} \in \mathbb{R}^{(I_1+1)\times(I_2+1)}\). Prior literature has designed greedy matrix CUR algorithms (see \cite{che2025how}) to avoid full matrix assembly for \(N=2\); this work generalizes such ideas to arbitrary higher-order tensors (\(N\geq3\)) within a unified tensor CUR decomposition in the Tucker framework paradigm.

\section{General theory for approximating (\ref{tensor-hyperinterpolation:approximation-general})}
\label{tensor-hyperinterpolation:sect4-main}

In Section 3 we fully formulated the hyperinterpolation framework over the $N$-dimensional hypercube \(\Omega=[-1,1]^N\) and derived the coefficient tensor \(\widetilde{\mathcal{A}}\) that encodes discrete cubature inner products between the target function and tensor-product orthogonal polynomial basis functions. A fundamental computational bottleneck emerges immediately from this tensor representation: the total number of scalar entries of \(\widetilde{\mathcal{A}}\) equals \(\prod_{n=1}^N (I_n+1)\), a quantity that grows exponentially with ambient dimension N and per-dimension polynomial degree \(I_n\). This exponential scaling, widely known as the curse of dimensionality, renders explicit storage, full tensor assembly, and standard Tucker approximation pipelines, including T-HOSVD, ST-HOSVD, and HOOI, computationally intractable even for moderately large $N$ and \(I_n\), as all conventional Tucker solvers require full access to the mode-$n$ unfoldings of \(\widetilde{\mathcal{A}}\), which demands evaluating every discrete cubature inner product defining \(\widetilde{\mathcal{A}}\).

To circumvent this infeasible full-tensor construction, this section develops a unified rigorous theoretical framework for computing provably accurate low-multilinear-rank approximations of \(\widetilde{\mathcal{A}}\) without explicitly assembling the complete coefficient tensor. Our core approximation objective is to construct a compact Tucker-format surrogate \(\mathcal{B} \approx \widetilde{\mathcal{A}}\) with drastically reduced Tucker rank tuple \( (R_1,R_2,\dots,R_N)\) satisfying \(R_n \ll I_n+1\), such that the resulting compressed hyperinterpolation operator retains rigorous, dimension-robust approximation fidelity to the original Sobolev-regular target function \(f(\mathbf{x})\in\mathbb{H}^\alpha(\Omega)\). We split the theoretical development into two core thematic blocks: first, we establish a universal intrinsic low-\(\epsilon\)-Tucker-rank existence theorem for hyperinterpolation coefficient tensors that quantifies their inherent compressibility; second, we systematically diagnose the structural limitations of existing Tucker approximation algorithms when applied to hyperinterpolation workflows, motivating the tensor CUR (TCUR) methodology developed in subsequent subsections. All structural results derived here serve as the rigorous mathematical foundation for the greedy TCUR sampling algorithms and end-to-end error chain analysis that unifies Sobolev regularity, hyperinterpolation discretization, and low-rank tensor compression.

\subsection{Low-rank Tucker surrogates for compressed hyperinterpolation}
\label{tensor-hyperinterpolation:sect4-main:sub1}
Fix a Tucker rank $(R_1,R_2,\dots,R_N)$ satisfying $R_n<I_n+1$, any Tucker-format tensor surrogate \(\mathcal{B}\) of \(\widetilde{\mathcal{A}}\) takes the standard multilinear factorization:
\begin{equation}
\mathcal{B}=\mathcal{G}\times_1\mathbf{Q}_1\times_2\mathbf{Q}_2\dots\times_N\mathbf{Q}_N,
\label{tensor-hyperinterpolation:approximation-tucker-expression-general}
\end{equation}
where $\mathcal{G}\in\mathbb{R}^{R_1\times R_2\times \dots\times R_N}$ is a core tensor and all $\mathbf{Q}_n\in\mathbb{R}^{(I_n+1)\times R_n}$ are orthonormal for all $n$. Substituting this low-rank surrogate into the hyperinterpolation expansion yields a compressed approximation operator $\widetilde{\mathcal{L}}_{I_1,\dots I_N} f (\mathbf{x})$, defined as:
\begin{equation}
    \widetilde{\mathcal{L}}_{I_1,\dots I_N} f (\mathbf{x})=\sum_{j_1=1}^{R_1}\dots \sum_{j_N=1}^{R_N}g_{j_1,\dots,j_N}\Psi_{j_1}^{(1)}(x^{(1)})\dots\Psi_{j_N}^{(N)}(x^{(N)}),
    \label{tensor-hyperinterpolation:approximation-expression-general}
\end{equation}
with compressed basis functions
\begin{equation*}
    \Psi_{j_n}^{(n)}(x^{(n)})=\sum_{i_n=1}^{I_n+1}\mathbf{Q}_n(i_n,j_n)\Phi_{i_n}^{(n)}(x^{(n)}).
\end{equation*}

This compressed formulation delivers three decisive computational advantages over the full hyperinterpolation polynomial in (\ref{tensor-hyperinterpolation:approximation-general}) as follows:
\begin{enumerate}
\item[(a)] Exponential storage reduction: The full expansion stores \(\prod_{n=1}^N(I_n+1)\) coefficients, while the Tucker surrogate only requires storing \(\prod_{n=1}^N R_n\) core tensor entries. Since \(R_n \ll I_n+1\), this eliminates the exponential storage overhead of the original coefficient tensor and removes the need for full assembly of \(\widetilde{\mathcal{A}}\).
\item[(b)] Reduced basis complexity: The full tensor-product polynomial basis relies on \(\sum_{n=1}^N(I_n+1)\) univariate basis functions, whereas the compressed operator only uses \(\sum_{n=1}^N R_n\) linearly transformed basis functions across all dimensions.
\item[(c)] Tensor-free evaluation compatibility: When paired with the TCUR sampling framework introduced in Section 5, this Tucker surrogate can be constructed by evaluating only a tiny subset of discrete cubature inner products, rather than computing every entry of \(\widetilde{\mathcal{A}}\).
\end{enumerate}

Despite these benefits, two foundational unresolved questions constrain practical deployment of this compressed hyperinterpolation framework:
\begin{enumerate}
\item[1)] Does the hyperinterpolation coefficient tensor \(\widetilde{\mathcal{A}}\) admit an intrinsic low-\(\epsilon\)-Tucker-rank structure with approximation error bounds that are nearly independent of ambient dimension $N$?
\item[2)] Can we compute the Tucker factorization \(\{\mathcal{G};{\bf Q}_1,{\bf Q}_2,\dots,{\bf Q}_N\}\) without explicitly constructing the full coefficient tensor \(\widetilde{\mathcal{A}}\)?
\end{enumerate}

Section \ref{tensor-hyperinterpolation:sect4-main:sub3} resolves question (a) via a universal low-\(\epsilon\)-Tucker-rank existence theorem, while Section \ref{tensor-hyperinterpolation:sect4-main:sub4} reviews existing Tucker algorithms to demonstrate that all standard methods fail question (b), which creates the critical motivation for the TCUR methodology developed in Section \ref{tensor-hyperinterpolation:sect5-main}.

\subsection{Universal low-$\epsilon$-Tucker-rank theorems}
\label{tensor-hyperinterpolation:sect4-main:sub3}

This subsection establishes the core structural result of the paper: a rigorous existence theorem proving that all hyperinterpolation coefficient tensors \(\widetilde{\mathcal{A}}\) possess intrinsic low-\(\epsilon\)-Tucker-rank approximations with entry-wise error bounds that scale weakly with dimension $N$. We build the proof upon matrix \(\epsilon\)-rank theory from Udell and Townsend  \cite{udell2019big}, extended recursively to high-order tensors via sequential mode-wise low-rank approximation.

For tolerance \(0<\epsilon<1\), the \(\epsilon\)-Tucker rank \(\text{Trank}_\epsilon(\widetilde{\mathcal{A}})\) denotes the minimal multilinear rank tuple \( (R_1,R_2,\dots,R_N)\) such that there exists a Tucker-format tensor \(\mathcal{B}\) satisfying \(\|\widetilde{\mathcal{A}}-\mathcal{B}\|_{\text{max}} \leq \epsilon\). The following theorem quantifies the small magnitude of these required ranks for any hyperinterpolation coefficient tensor.
\begin{theorem}
    Let $\widetilde{\mathcal{A}}\in\mathbb{R}^{(I_1+1)\times (I_2+1)\times\dots\times (I_N+1)}$ be defined in {\rm(\ref{tensor-hyperinterpolation:approximation-coefficient})} and $0<\epsilon<1$. Then, there exists a Tucker-format tensor $\mathcal{B}\in\mathbb{R}^{(I_1+1)\times(I_2+1)\times \dots\times (I_N+1)}$ such that ${\rm Trank}(\mathcal{B})\leq (R_1,R_2,\dots,R_N)$ and
    \begin{equation*}
        \|\widetilde{\mathcal{A}}-\mathcal{B}\|_{\max}\leq \epsilon\left(\frac{a^N-1}{a-1}\right)\|\widetilde{\mathcal{A}}\|_F,\quad a=\left(1+\epsilon \sqrt{\prod_{n=1}^N(I_n+1)}\right)>1,
    \end{equation*}
where each mode-$n$ Tucker rank satisfies
    \begin{equation*}
         R_{n}=\left\lceil 72\ln\left((I_{n}+1)+\prod_{k=1}^{n-1}R_{k}\prod_{l=n+1}^N(I_{l}+1)+1\right)/\epsilon^2\right\rceil,\quad n=1,2,\dots,N.
    \end{equation*}
    \label{tensor-hyperinterpolation:epsilonTrank-theorem}
\end{theorem}
\begin{remark}
Theorem \ref{tensor-hyperinterpolation:epsilonTrank-theorem} is the foundational structural theorem for this paper, proving that any hyperinterpolation coefficient tensor \(\widetilde{\mathcal{A}}\) admits a low-Tucker-rank surrogate \(\mathcal{B}\) with entrywise max-norm error control, and explicitly quantifies the required Tucker ranks \( (R_1,R_2,\dots,R_N) \). Several key comments for this theorem are listed as follows:
\begin{enumerate}[(a)]
\item The rank formula \(R_n = \left\lceil 72\ln\left((I_n+1)+\prod_{k=1}^{n-1}R_k\prod_{l=n+1}^N(I_l+1)+1\right)/\epsilon^2\right\rceil\) only carries logarithmic dependence on tensor dimensions and polynomial degrees, not polynomial/exponential scaling. This mathematically confirms hyperinterpolation tensors are inherently compressible, rather than just empirically compressible in small test cases. The error prefactor \(\frac{a^N-1}{a-1}\) grows mildly with $N$ but remains manageable for moderate dimensions, establishing nearly dimension-independent approximation guarantees.
\item This theorem answers the core open question (a) posed in the introduction: it rigorously explains why low-rank compression works for high-dimensional hyperinterpolation, instead of only demonstrating how to implement compression. All subsequent TCUR theory and greedy algorithms rest on this existence result.
\item As polynomial degrees \(I_n\) grow, \(R_n < I_n+1\) always holds, meaning the low-rank surrogate strictly reduces storage relative to the full coefficient tensor. For fixed tolerance \(\epsilon\), increasing dimension $N$ only slightly raises \(R_n\), so the low-rank advantage persists even in ultrahigh dimensions.
\item The theorem holds for any permutation of mode processing order in symmetric group \(\mathbb{S}_N\). No single dimension order yields drastically better rank bounds, which removes artificial tuning parameters from the structural theory.
\end{enumerate}
\end{remark}
\begin{remark}
    The bound \(R_n < I_n+1\) holds for sufficiently large polynomial degrees \(I_n\), meaning the low-rank surrogate is strictly more compact than the original tensor representation. The logarithmic dependence of \(R_n\) on tensor dimensions ensures the required ranks grow very slowly even as N and \(I_n\) increase, confirming the intrinsic compressibility of hyperinterpolation coefficients. As an example, Figure \ref{tensor-hyperinterpolation:fig1-main} illustrates that for some $0<\epsilon<1$,Theorem \ref{tensor-hyperinterpolation:epsilonTrank-theorem} is suitable for large scale $I_n$s'.
\end{remark}
\begin{figure}[htb]
    \setlength{\tabcolsep}{4pt}
    \renewcommand\arraystretch{1}
    \centering
    \includegraphics[width=1\linewidth]{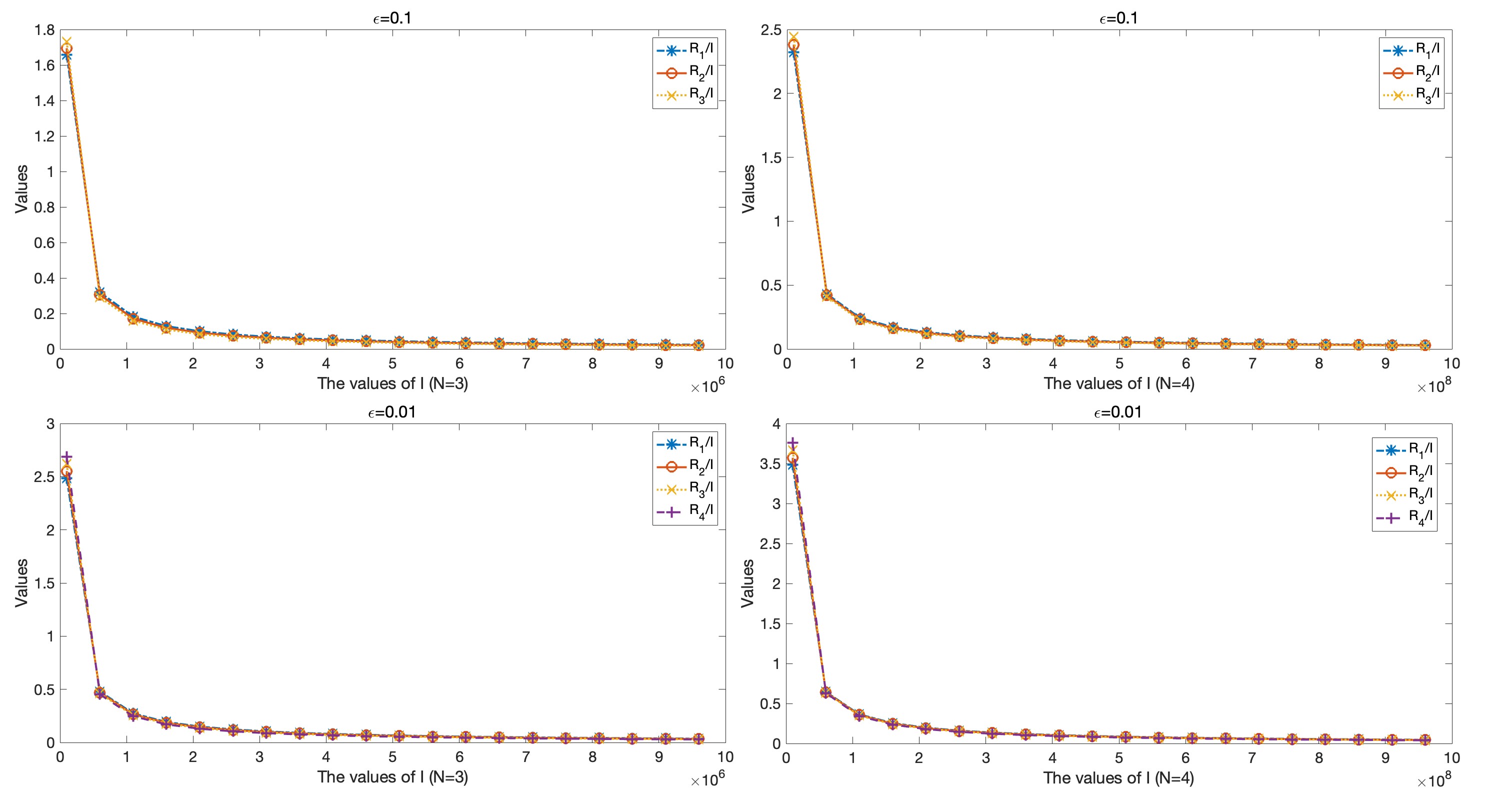}\\
    \caption{Illustration of Theorem \ref{tensor-hyperinterpolation:epsilonTrank-theorem}, with different values for  the pair $(\epsilon,N)$, the values of each part in $\{R_1/I_1,R_2/I_2,\dots,R_N/I_N\}$, where $I_n=I$ and $n=1,2,\dots,N$.}
    \label{tensor-hyperinterpolation:fig1-main}
\end{figure}

We rely on a foundational matrix low-\(\epsilon\)-rank result from Udell and Townsend \cite{udell2019big} as the building block for the tensor proof.
\begin{theorem}{{\bf (see \cite[Theorem 1.0]{udell2019big})}}
    Let $\mathbf{A}\in\mathbb{R}^{I_1\times I_2}$ and $0<\epsilon<1$. Then, there exists a matrix $\mathbf{B}\in\mathbb{R}^{I_1\times I_2}$ such that ${\rm rank}(\mathbf{B})\leq R$ and
    \begin{equation*}
        \|\mathbf{A}-\mathbf{B}\|_{\max}\leq \epsilon\|\mathbf{A}\|_2,
    \end{equation*}
    with $R=\lceil 72\ln(I_1+I_2+1)/\epsilon^2\rceil$.
    \label{tensor-hyperinterpolation:epsilonrank}
\end{theorem}

The following corollary is easily obtained from Theorem \ref{tensor-hyperinterpolation:epsilonrank}.
\begin{corollary}
    Let $\mathbf{A}\in\mathbb{R}^{I_1\times I_2}$ and $0<\epsilon<1$. Then, there exists a matrix $\mathbf{B}\in\mathbb{R}^{I_1\times I_2}$ such that ${\rm rank}(\mathbf{B})\leq R$ and
    \begin{equation*}
        \|\mathbf{A}-\mathbf{B}\|_{\max}\leq \epsilon\|\mathbf{A}\|_F,
    \end{equation*}
    with $R=\lceil 72\ln(I_1+I_2+1)/\epsilon^2\rceil$.
    \label{tensor-hyperinterpolation:epsilonrank-corollary}
\end{corollary}
\begin{remark}
Theorem \ref{tensor-hyperinterpolation:epsilonrank} and Corollary \ref{tensor-hyperinterpolation:epsilonrank-corollary} show that any matrix can be approximated entrywise within a relative tolerance $\epsilon$ using a rank that scales logarithmically with the matrix dimensions, not polynomially or exponentially.
\end{remark}

We now give a detailed proof for Theorem \ref{tensor-hyperinterpolation:epsilonTrank-theorem}.
\begin{proof}
    Note that when any tensor $\mathcal{B}\in\mathbb{R}^{(I_1+1)\times (I_2+1)\times\dots\times (I_N+1)}$ satisfies ${\rm Trank}(\mathcal{B})\leq (R_1,R_2,\dots,R_N)$, it can be represented by
    \begin{equation*}
        \mathcal{B}=\mathcal{G}\times_1\mathbf{Q}_1\times_2\mathbf{Q}_2\dots\times_N\mathbf{Q}_N,
    \end{equation*}
    where $\mathcal{G}\in\mathbb{R}^{R_1\times R_2\times \dots\times R_N}$, and all the $\mathbf{Q}_n\in\mathbb{R}^{(I_n+1)\times R_n}$ are orthonormal matrices with $n=1,2,\dots,N$. 
    
    Hence, without loss of generality, there exist $(N-1)$ tensors $\mathcal{B}_n\in\mathbb{R}^{R_1\times \dots \times R_n\times (I_{n+1}+1)\times\dots\times (I_N+1)}$ with $n=1,2,\dots,N-1$ such that
    \begin{equation*}
        \begin{aligned}
            \widetilde{\mathcal{A}}-\mathcal{B}&=\widetilde{\mathcal{A}}-\mathcal{B}_1\times_1\mathbf{Q}_1
            +\mathcal{B}_1\times_1\mathbf{Q}_1-\mathcal{B}_2\times_1\mathbf{Q}_1\times_2\mathbf{Q}_2\\
            &+\dots+\mathcal{B}_{N-1}\times_1\mathbf{Q}_1\dots\times_{N-1}\mathbf{Q}_{N-1}-\mathcal{G}\times_1\mathbf{Q}_1\dots\times_N\mathbf{Q}_N,
        \end{aligned}
    \end{equation*}
    which implies that
    \begin{equation*}
        \begin{aligned}
            \|\widetilde{\mathcal{A}}-\mathcal{B}\|_{\max}&\leq\|\widetilde{\mathcal{A}}-\mathcal{B}_1\times_1\mathbf{Q}_1\|_{\max}
            +\|\mathcal{B}_1\times_1\mathbf{Q}_1-\mathcal{B}_2\times_1\mathbf{Q}_1\times_2\mathbf{Q}_2\|_{\max}\\
            &+\dots+\|\mathcal{B}_{N-1}\times_1\mathbf{Q}_1\dots\times_{N-1}\mathbf{Q}_{N-1}
            -\mathcal{G}\times_1\mathbf{Q}_1\dots\times_N\mathbf{Q}_N\|_{\max}.
        \end{aligned}
    \end{equation*}

    First of all, by using Corollary \ref{tensor-hyperinterpolation:epsilonrank-corollary}, we have that
    \begin{equation*}
        \|\widetilde{\mathcal{A}}-\mathcal{B}_1\times_1\mathbf{Q}_1\|_{\max}\leq \epsilon\|\widetilde{\mathcal{A}}\|_F
    \end{equation*}
    with $R_1=\lceil 72\ln((I_1+1)+(I_2+1)\dots (I_N+1)+1)/\epsilon^2\rceil$, which implies that
    \begin{equation*}
        \|\mathcal{B}_1\times_1\mathbf{Q}_1\|_F\leq\left(1+\epsilon \sqrt{\prod_{n=1}^N(I_n+1)}\right)\|\widetilde{\mathcal{A}}\|_F.
    \end{equation*}
    Furthermore, for $n=2,3,\dots,N$, it also follows from Corollary \ref{tensor-hyperinterpolation:epsilonrank-corollary} that
    \begin{equation*}
        \|\mathcal{B}_{n-1}\times_1\mathbf{Q}_1\dots\times_{n-1}\mathbf{Q}_{n-1}
        -\mathcal{B}_{n}\times_1\mathbf{Q}_1\dots\times_n\mathbf{Q}_n\|_{\max}\leq \epsilon\|\mathcal{B}_{n-1}\times_1\mathbf{Q}_1\dots\times_{n-1}\mathbf{Q}_{n-1}\|_F
    \end{equation*}
    with
    \begin{equation*}
        R_n=\lceil 72\ln((I_n+1)+R_1\dots R_{n-1}(I_{n+1}+1)\dots (I_N+1)+1)/\epsilon^2\rceil,
    \end{equation*}
    which also implies that
    \begin{equation*}
        \|\mathcal{B}_{n}\times_1\mathbf{Q}_1\dots\times_n\mathbf{Q}_n\|_F\leq \left(1+\epsilon \sqrt{\prod_{n=1}^N(I_n+1)}\right)\|\mathcal{B}_{n-1}\times_1\mathbf{Q}_1\dots\times_{n-1}\mathbf{Q}_{n-1}\|_F.
    \end{equation*}
    Here we set $\mathcal{B}_N=\mathcal{G}$. Hence, the proof is complete.
\end{proof}

We now provide a more general expression for Theorem \ref{tensor-hyperinterpolation:epsilonTrank-theorem}.
\begin{remark}
    Suppose that $\mathbb{S}_N$ is the $N$th order symmetric group on the set $\{1,2,\dots,N\}$. Let $\widetilde{\mathcal{A}}\in\mathbb{R}^{(I_1+1)\times \dots\times (I_N+1)}$ and $0<\epsilon<1$. Then, for any processing order $\{p_1,p_2,\dots,p_N\}\in\mathbb{S}_N$, there exists a tensor $\mathcal{B}\in\mathbb{R}^{(I_1+1)\times \dots\times (I_N+1)}$ such that ${\rm Trank}(\mathcal{B})\leq \{R_1,R_2,\dots,R_N\}$ and
    \begin{equation*}
        \|\widetilde{\mathcal{A}}-\mathcal{B}\|_{\max}\leq \epsilon\left(\frac{a^N-1}{a-1}\right)\|\widetilde{\mathcal{A}}\|_F,
    \end{equation*}
    with $a=\left(1+\epsilon \sqrt{\prod_{n=1}^N(2I_n+1)}\right)>1$, where
    \begin{equation*}
        R_{p_n}=\left\lceil 72\ln\left((I_{p_n}+1)+\prod_{k=1}^{n-1}R_{p_k}\prod_{l=n+1}^N(I_{p_l}+1)+1\right)/\epsilon^2\right\rceil
    \end{equation*}
    with $n=1,2,\dots,N$.
\end{remark}

By combining Theorems \ref{tensor-hyperinterpolation:general-theorem} and \ref{tensor-hyperinterpolation:epsilonTrank-theorem}, we obtain the upper bound for $\|\widetilde{\mathcal{L}}_{I_1,\dots, I_N} f (\mathbf{x})-f (\mathbf{x})\|_{L^2}$, which is summarized in the following theorem.
\begin{theorem}
Let $f (\mathbf{x})\in \mathbb{H}^{\alpha}(\Omega)$ and $\widetilde{\mathcal{L}}_{I_1,\dots, I_N} f (\mathbf{x})$ be in {\rm(\ref{tensor-hyperinterpolation:approximation-expression-general})}. Under the conditions of Theorems {\rm \ref{tensor-hyperinterpolation:general-theorem}} and {\rm \ref{tensor-hyperinterpolation:epsilonTrank-theorem}}, one has
    \begin{align*}
        &\|\widetilde{\mathcal{L}}_{I_1,\dots, I_N} f (\mathbf{x})-f (\mathbf{x})\|_{L^2}\\
        &\quad\leq C\cdot\min\{I_1,\dots,I_N\}^{-\alpha}|f (\mathbf{x})|_{\alpha}+\epsilon\left(\frac{a^N-1}{a-1}\right)\sqrt{\prod_{n=1}^N(I_n+1)}\cdot\|\widetilde{\mathcal{A}}\|_F
    \end{align*}
    with $a=1+\epsilon \sqrt{\prod_{n=1}^N(I_n+1)}$, where the constant $C$ is independent of $f (\mathbf{x})$ and $\{I_1,I_2,\dots,I_N\}$.
\label{tensor-hyperinterpolation:main-theorem1}
\end{theorem}
\begin{remark}
Theorem \ref{tensor-hyperinterpolation:main-theorem1} combines the hyperinterpolation discretization error (Theorem \ref{tensor-hyperinterpolation:general-theorem}) and tensor max-norm compression error (Theorem \ref{tensor-hyperinterpolation:epsilonTrank-theorem}) into a unified bound for the compressed hyperinterpolation operator \(\widetilde{\mathcal{L}}_{I_1,\dots, I_N}\): the first term is controlled only by function smoothness \(\alpha\) and per-dimension polynomial degree \(I_n\), independent of tensor compression, and the second term is controlled solely by tolerance \(\epsilon\) and tensor norm \(\|\widetilde{\mathcal{A}}\|_F\), independent of the target function’s regularity. Meanwhile, the term \(\sqrt{\prod_{n=1}^N(I_n+1)}\) converts max-norm tensor compression residual to \(L^2\) function error, connecting tensor-space approximation loss to the original function space metric relevant for scientific computing applications.
\end{remark}
\begin{proof}
Note that we have 
\begin{align*}
\widetilde{\mathcal{L}}_{I_1,\dots, I_N} f (\mathbf{x})-f (\mathbf{x})&=\widetilde{\mathcal{L}}_{I_1,\dots, I_N} f (\mathbf{x})-\mathcal{L}_{I_1,\dots, I_N} f (\mathbf{x})+\mathcal{L}_{I_1,\dots, I_N} f (\mathbf{x})-f (\mathbf{x}),
\end{align*}
    which implies that
    \begin{align*}
        \|\widetilde{\mathcal{L}}_{I_1,\dots, I_N} f (\mathbf{x})-f (\mathbf{x})\|_{L^2}&\leq \|\widetilde{\mathcal{L}}_{I_1,\dots, I_N} f (\mathbf{x})-\mathcal{L}_{I_1,\dots, I_N} f (\mathbf{x})\|_{L^2}+\|\mathcal{L}_{I_1,\dots, I_N} f (\mathbf{x})-f (\mathbf{x})\|_{L^2}
    \end{align*}
Combining (\ref{tensor-hyperinterpolation:approximation-general}) and (\ref{tensor-hyperinterpolation:approximation-expression-general}), one has
\begin{align*}
&\widetilde{\mathcal{L}}_{I_1,\dots, I_N} f (\mathbf{x})-\mathcal{L}_{I_1,\dots, I_N} f (\mathbf{x})= \sum_{i_1=1}^{I_1+1}\dots\sum_{i_N=1}^{I_2+1} (\widetilde{\alpha}_{i_1\dots i_N}-b_{i_1\dots i_N}) \Phi_{1}^{(i_1)}(x^{(1)}) \dots\Phi_{N}^{(i_N)}(x^{(N)}),
\end{align*}
where each $\widetilde{\alpha}_{i_1\dots i_N}$ is given in (\ref{tensor-hyperinterpolation:approximation-coefficient}) and $b_{i_1\dots i_N}$ is the $(i_1,\dots,i_N)$-th entry of $\mathcal{B}$ in (\ref{tensor-hyperinterpolation:approximation-tucker-expression-general}). For each $n$, due to orthonormality of the basis $\{\Phi_{n}^{(i_n)}(x^{(n)}):i_n=1,2,\dots,I_n+1\}$ under the scaled inner product,
\begin{align*}
&\|\widetilde{\mathcal{L}}_{I_1,\dots, I_N} f (\mathbf{x})-\mathcal{L}_{I_1,\dots, I_N} f (\mathbf{x})\|_{L^2}^2 = \sum_{i_1=1}^{I_1+1}\dots\sum_{i_N=1}^{I_2+1}  (\widetilde{\alpha}_{i_1\dots i_N}-b_{i_1\dots i_N}) ^2 = \| \widetilde{\mathcal{A}} - \mathcal{B} \|_F^2.
\end{align*}
Using the inequality $\| \cdot \|_F \leq \sqrt{\prod_{n=1}^N(I_n+1)} \| \cdot \|_{\max}$ for an $(I_1+1)\times\dots\times (I_N+1)$ tensor,
\begin{align*}
&\|\widetilde{\mathcal{L}}_{I_1,\dots, I_N} f (\mathbf{x})-\mathcal{L}_{I_1,\dots, I_N} f (\mathbf{x})\|_{L^2}\leq\sqrt{\prod_{n=1}^N(I_n+1)} \cdot\| \widetilde{\mathcal{A}} - \mathcal{B} \|_{\max}.
\end{align*}
    Therefore, the proof is complete by combining Theorems {\rm \ref{tensor-hyperinterpolation:general-theorem}} and {\rm \ref{tensor-hyperinterpolation:epsilonTrank-theorem}}.
\end{proof}

\subsection{The existing algorithms and drawbacks}
\label{tensor-hyperinterpolation:sect4-main:sub4}

Given a tolerance $\epsilon$, once the desired $N$-tuple $(R_1,R_2,\dots,R_N)$ is deduced from Theorem {\rm \ref{tensor-hyperinterpolation:epsilonTrank-theorem}}, the problem of finding a $(N+1)$-tuple $\{\mathcal{G};\mathbf{Q}_1,\dots,\mathbf{Q}_N\}$ from the coefficient tensor $\widetilde{\mathcal{A}}$ is called the fixed Tucker-rank problem.

\begin{problem}{{\bf (see \cite[(4.1)]{lathauwer2000on} and \cite[(4.3)]{kolda2009tensor})}}
\label{tensor-hyperinterpolation:fixed-Trank}
    For the coefficient tensor $\widetilde{\mathcal{A}}\in\mathbb{R}^{(I_1+1)\times \dots\times (I_N+1)}$ and $N$ positive integers $R_n< I_n+1$ with $n=1,2,\dots,N$, the objective is to find $N$ orthonormal matrices ${\bf Q}_{n}\in\mathbb{R}^{(I_{n}+1)\times R_{n}}$ such that
    \begin{equation*}
        \{\mathbf{Q}_1,\dots,\mathbf{Q}_N\}=\mathop{\rm argmin}\limits_{\mathbf{V}_1,\dots,\mathbf{V}_N}\left\|\widetilde{\mathcal{A}}-\widetilde{\mathcal{A}}\times_1({\bf V}_1{\bf V}_1^{*})\dots\times_N({\bf V}_N{\bf V}_N^{*})\right\|_F^2,
    \end{equation*}
    where all the matrices $\mathbf{V}_n\in\mathbb{R}^{(I_{n}+1)\times R_{n}}$ are orthonormal.
\end{problem}

We overview two common algorithms for solving Problem \ref{tensor-hyperinterpolation:fixed-Trank}. The first one is the alternating least squares (ALS) algorithm (cf. \cite{lathauwer2000on}). It is also called higher-order orthogonal iteration (HOOI) and is implemented by the function ``tucker\_als'' in MATLAB tensor toolbox \cite{tensortool}. In detail, the ALS algorithm updates one of the factor matrices along with the core tensor at a time, and is bottlenecked by the operation called the tensor times matrix-chain (TTMc). Several randomized variants for the ALS method are discussed in \cite{caiafa2010generalizing,fahrbach2021fast,ma2021fast,malik2018low,tsourakakis2010mach}.

The truncated higher-order singular value decomposition (T-HOSVD) and the sequentially T-HOSVD (ST-HOSVD) are two other common algorithms for solving Problem \ref{tensor-hyperinterpolation:fixed-Trank}. More details can be found in \cite{delathauwer2000a,kolda2009tensor} and their references. Note that T-HOSVD and ST-HOSVD is less accurate than the ALS algorithm but can provide a good initial value for the ALS algorithm. In the framework of T-HOSVD and/or ST-HOSVD, for each $n$, the main calculation for T-HOSVD and ST-HOSVD is to compute an exact SVD of the mode-$n$ unfolding of the corresponding tensor. By using randomized SVD with the random projection and/or sampling techniques, different randomized variants of T-HOSVD and ST-HOSVD are presented recently, for example, see \cite{ahmadi2021randomized,che2020computation,che2021efficient,che2021randomized,che2025efficient-acom,minster2024parallel,
minster2020randomized,sun2020low,zhou2014decomposition}.

However, the $N$-tuple $\{R_1,R_2,\dots,R_N\}$ is unknown in advance. Hence, the problem is summarized as the fixed precision problem.
\begin{problem}{{\bf (see \cite[Problem 1.1]{che2025efficient-siam} and \cite[(3.1)]{xiao2024rank})}}
\label{tensor-hyperinterpolation:fixed-precision}
    For a given $0<\epsilon<1$ and the coefficient tensor $\widetilde{\mathcal{A}}\in\mathbb{R}^{(I_1+1)\times (I_2+1)\times \dots\times (I_N+1)}$, the goal is to find $N$ positive integers $R_n$ and $N$ orthonormal matrices ${\bf Q}_{n}\in\mathbb{R}^{(I_{n}+1)\times R_{n}}$ such that
    \begin{equation*}
        \|\widetilde{\mathcal{A}}-\widetilde{\mathcal{A}}\times_1(\mathbf{Q}_1\mathbf{Q}_1^\top)\dots\times_N(\mathbf{Q}_N\mathbf{Q}_N^\top)\|_F\leq \epsilon.
    \end{equation*}
\end{problem}

We now briefly review the existing results for solving Problem \ref{tensor-hyperinterpolation:fixed-precision}. For adaptive variants of T-HOSVD and ST-HOSVD, we need to calculate the truncated SVD with a given tolerance $0<\epsilon<1$ of the mode-$n$ unfolding of the corresponding tensor with $n=1,2,\dots,N$. As shown in \cite{che2019randomized}, an adaptive randomized variant for T-HOSVD and ST-HOSVD is presented by  replacing the standard Gaussian vectors used in the adaptive randomized range finder (see \cite[Algorithm 4.2]{halko2011finding}) with the Khatri-Rao product of the standard Gaussian vectors. Minster {\it et al.} \cite{minster2020randomized} deduced adaptive randomized algorithms for the fixed-precision problem of approximate Tucker decomposition by using a version of the adaptive randomized range finder algorithm in \cite{martinsson2016randomized,yu2018efficient}. Xiao and Yang \cite{xiao2024rank} proposed a rank-adaptive (RA) variant of the ALS algorithm to calculate the truncated Tucker decomposition of higher-order tensors with a given tolerance $0<\epsilon<1$ and proved that the method is locally optimal and monotonically convergent. Hashemi and Nakatsukasa \cite{hashemi2025rtsms} proposed an adaptive randomized algorithm for approximate Tucker decomposition, denoted by RTSMS (Randomized Tucker via Single-Mode-Sketching). Che {\it et al.} \cite{che2025efficient-siam} modified several structured matrices for the adaptive randomized range finder algorithm in \cite{yu2018efficient}, when the standard Gaussian matrices are replaced by uniform random matrices, the Khatri-Rao product of the standard Gaussian matrices (or the uniform random matrices), and obtained the adaptive randomized variants for T-HOSVD and ST-HOSVD by using this modified algorithm to each mode unfolding of the input/intermediate tensor.

This critical practical limitation, that is, all mainstream Tucker algorithms require full or partial assembly of the exponential-sized coefficient tensor, motivates our development of interpolatory tensor CUR (TCUR) decomposition, which constructs accurate low-rank surrogates by evaluating only a tiny, strategically selected subset of cubature inner products without constructing \(\widetilde{\mathcal{A}}\) in full.

\subsection{Revisiting Theorem \ref{tensor-hyperinterpolation:main-theorem1} via TCUR approximation}
\label{tensor-hyperinterpolation:sect4-main:sub5}
For clear exposition throughout this subsection, we assume the target low multilinear Tucker rank tuple \( (R_1,R_2,\dots,R_N)\) from the ideal full Tucker surrogate (\ref{tensor-hyperinterpolation:approximation-tucker-expression-general}) s pre-specified and fixed. To guarantee numerical stability of the interpolatory tensor CUR approximation, we enforce the strict rank ordering $R_n<S_n<I_n+1$ with $n=1,2,\dots,N$. For each $n$, let $\mathbb{I}_n=\{i_1^{(n)},i_2^{(n)},\dots,i_{S_n}^{(n)}\}$ such that $1\leq i_1^{(n)}<i_2^{(n)}<\dots<i_{S_n}^{(n)}\leq I_n+1$ and $\mathbb{J}_n=\{l_{n,1},l_{n,2},\dots,l_{n, S_n'}\}$ be a subset of $\{1,2,\dots,J_n\}$ and $J_n=\prod_{m\neq n}(I_m+1)$ such that $1\leq l_{n,1}<l_{n,2}<\dots<l_{n,S_n'}\leq J_n$ and $S_n'<J_n$. Furthermore, let $\mathbf{C}_n=\widetilde{\mathbf{A}}_{(n)}(:,\mathbb{J}_n)\in\mathbb{R}^{(I_n+1)\times S_n'}$ and $\mathcal{G}=\widetilde{\mathcal{A}}(\mathbb{I}_1,\dots,\mathbb{I}_N)\in\mathbb{R}^{S_1\times S_2\times \dots\times S_N}$, where $\widetilde{\mathbf{A}}_{(n)}$ is the mode-$n$ folding of $\widetilde{\mathcal{A}}$.

We further define \({\bf U}_n = {\bf C}_n(\mathbb{I}_n,:)\), the submatrix of \({\bf C}_n\) restricted to rows indexed by \(\mathbb{I}_n\); this matrix acts as the low-dimensional basis for the TCUR core’s projection onto the full tensor space. Mirroring the standard Tucker factorization structure from (\ref{tensor-hyperinterpolation:approximation-tucker-expression-general}), the TCUR decomposition of $\widetilde{\mathcal{A}}$ takes the unified form:
\begin{align}
\label{tensor-hyperinterpolation:approximate-one}
   \widetilde{\mathcal{A}}_{{\rm TCUR}}&=\mathcal{G}\times_1\left(\mathbf{C}_1\mathbf{U}_1^\dag\right)
   \times_2\dots\times_N\left(\mathbf{C}_N\mathbf{U}_N^\dag\right),
\end{align}
where $\mathbf{U}_n^\dag$ denotes the Moore–Penrose pseudoinverse of \(\mathbf{U}_n\). We distinguish two canonical TCUR variants via the geometric definition of \(\mathbb{J}_n\):
\begin{enumerate}
\item[(a)] {\bf Chidori-type TCUR}: \(\mathbb{J}_n \) is set to \(\mathbb{I}_1\otimes\dots\otimes \mathbb{I}_{n-1}\otimes\mathbb{I}_{n+1}\otimes\dots\otimes \mathbb{I}_{N}\), meaning the column indices of \(\mathbf{C}_n\) are tensor products of cross-mode row index sets \(\mathbb{I}_m\), and the $(N+1)$-tuple $\{\mathcal{G};\mathbf{C}_1,\dots,\mathbf{C}_N\}$ is obtained that the matrix $\mathbf{C}_n$ is the mode-$n$ unfolding of $\mathcal{C}_n$ and the tensor $\mathcal{G}$ is given by
\begin{align}
\label{tensor-hyperinterpolation:coefficient-core}
   \mathcal{G}=\mathcal{C}_n(:,\dots,:,\mathbb{I}_n,:,\dots,:).
\end{align}
where $\mathcal{C}_n\in\mathbb{R}^{S_1\times\dots \times S_{n-1}\times(I_n+1)\times S_{n+1}\times\dots\times S_N}$ is defined as
\begin{equation}
\label{tensor-hyperinterpolation:coefficient-factor}
    \begin{aligned}
        &\mathcal{C}_n(j_1,\dots,j_{n-1},i_n,j_{n+1},\dots,j_N)=\left\langle{f, \Phi_{i_{j_1}^{(1)}}^{(1)}\dots \Phi_{i_{j_{n-1}}^{(n-1)}}^{(n-1)}\Phi_{i_n}^{(n)}\Phi_{i_{j_{n+1}}^{(n+1)}}^{(n+1)} \dots\Phi_{i_{j_{N}}^{(N)}}^{(N)}}\right\rangle_{M_1,\dots,M_N}
    \end{aligned}
\end{equation}
where $i_n=1,2,\dots,I_n+1$, $j_m=1,2,\dots,S_m$ and $i_{j_{m}}^{(m)}\in\mathbb{I}_m$ with $m\neq n$. 
\item[(b)] {\bf Fiber-type TCUR}: \(\mathbb{J}_n\) is an arbitrary independent subset of the columns of \(\widetilde{\mathbf{A}}_{(n)}\), with no direct coupling to \(\{\mathbb{I}_m:m\neq n,m=1,2,\dots,N\}\), and the $(N+1)$-tuple $\{\mathcal{G};\mathbf{C}_1,\dots,\mathbf{C}_N\}$ is obtained that $\mathcal{G}$ is the same as that for Chidori-type TCUR and the entries of $\mathbf{C}_n'\in\mathbb{R}^{(I_n+1)\times S_n'}$ are given as
\begin{equation}
    \mathbf{C}_n'(i_n,j) =\left\langle{f, \Phi_{s_{j_1}^{(1)}}^{(1)}\dots \Phi_{s_{j_{n-1}}^{(n-1)}}^{(n-1)}\Phi_{i_n}^{(n)}\Phi_{s_{j_{n+1}}^{(n+1)}}^{(n+1)} \dots\Phi_{s_{j_{N}}^{(N)}}^{(N)}}\right\rangle_{M_1,\dots,M_N}
    \label{tensor-hyperinterpolation:method-two}
\end{equation}
with $i_n=1,2,\dots,I_n+1$, where for each $l_{n,j}$ with $j=1,2,\dots,S_n'$, there exists a unique $(N-1)$-tuple $(s_{j_1}^{(1)},\dots,s_{j_{n-1}}^{(n-1)},s_{j_{n+1}}^{(n+1)},\dots,s_{j_{N}}^{(N)})$ (see \cite{battaglino2018a}) such that
\begin{equation*}
    \begin{aligned}
        j=&s_{j_1}^{(1)}+(s_{j_2}^{(2)}-1)(I_1+1)+\dots+(s_{j_{n-1}}^{(n-1)}-1)(I_1+1)\dots (I_{n-2}+1)\\
        &+(s_{j_{n+1}}^{(n+1)}-1)(I_1+1)\dots (I_{n-1}+1)\\
        &+\dots+(s_{j_N}^{(N)}-1)(I_1+1)\dots (I_{n-1}+1)(I_{n+1}+1)\dots (I_{N}+1).
\end{aligned}
\end{equation*}
\end{enumerate}
\begin{remark}
    Note that the approximation given in (\ref{tensor-hyperinterpolation:approximate-one}) is called the structure-preserving Tucker decomposition (see \cite{minster2020randomized}), which is generalized from a related decomposition of matrices (see \cite{cheng2005compression}). Following Theorem 2 in \cite{cai2021mode}, when the Trank of $\mathcal{A}$ is the $N$-tuple $\{R_1,\dots,R_N\}$, then (\ref{tensor-hyperinterpolation:approximate-one}) can be equivalent to $\widetilde{\mathcal{A}}_{{\rm TCUR}}=\widetilde{\mathcal{A}}\times_1(\mathbf{C}_1\mathbf{C}_1^\dag)\dots\times_N(\mathbf{C}_N\mathbf{C}_N^\dag)$. However, this formulation is not suitable because the computational complexity of obtaining the coefficient tensor $\widetilde{\mathcal{A}}$ is too high.
\end{remark}

While (\ref{tensor-hyperinterpolation:approximate-one}) delivers a valid Tucker-compatible TCUR approximation \(\widetilde{\mathcal{A}}_{\text{TCUR}}\) of the hyperinterpolation coefficient tensor $\widetilde{\mathcal{A}}$, the intermediate multilinear rank tuple \(\{S_1,\dots,S_N\}\) of \(\widetilde{\mathcal{A}}_{\text{TCUR}}\) is consistently larger than the ideal target Tucker rank \(\{R_1,\dots,R_N\}\) required for minimal storage and fast function evaluation. As specified via the logarithmic scaling rules from \cite{cai2021mode}, \(S_n = R_n\log(I_n+1)\) for Chidori-type TCUR and \(S_n' = 2R_n\log\big(\prod_{m\neq n}(I_m+1)\big)\) for Fiber-type TCUR, meaning the TCUR core tensor \(\mathcal{G}\in\mathbb{R}^{S_1\times\dots\times S_N}\) still carries redundant dimension overhead. Directly deploying \(\widetilde{\mathcal{A}}_{\text{TCUR}}\) for high-dimensional surrogate modeling, uncertainty quantification, or data reconstruction would retain inflated computational cost for tensor multiplication and point evaluation. To eliminate this redundant subspace information without re-evaluating costly multilinear discrete cubature inner products or re-running the greedy index selection algorithms in Section \ref{tensor-hyperinterpolation:sect5-main}, we develop the TCUR-to-Tucker recompression framework in Appendix \ref{tensor-hyperinterpolation:app-main}. This auxiliary pipeline takes the precomputed TCUR factorization as input and further truncates its multilinear rank down to the desired compact \( (R_1,R_2,\dots,R_N)\) via QR factorization and ST-HOSVD, with provable composite Frobenius error bounds in Theorem \ref{tensor-hyperinterpolation:app-thm} that fully account for both the initial TCUR approximation residual and the secondary Tucker truncation error. The appendix therefore serves as a lightweight post-processing module that closes the full approximation chain, translating the overcomplete TCUR interpolatory representation into an optimally compressed low-Tucker-rank tensor with minimal memory footprint and evaluation complexity.

With the TCUR tensor surrogate fully defined, we rewrite the compressed hyperinterpolation operator induced by \(\widetilde{\mathcal{A}}_{\text{TCUR}}\) to replace the full Tucker surrogate \(\mathcal{B}\) from Section \ref{tensor-hyperinterpolation:sect4-main:sub1}:
\begin{equation}
    \widetilde{\mathcal{L}}_{I_1,\dots, I_N} f (\mathbf{x})=\sum_{j_1=1}^{S_1}\dots \sum_{j_N=1}^{S_N}g_{j_1,\dots,j_N}\Psi_{j_1}^{(1)}(x^{(1)})\dots\Psi_{j_N}^{(N)}(x^{(N)}),
    \label{tensor-hyperinterpolation:approximation-expression-general-chidori}
\end{equation}
 where for each $n$, the function $\Psi_{j_n}^{(n)}(x_n)$ is given by
\begin{equation*}
    \Psi_{j_n}^{(n)}(x^{(n)})=\sum_{i_n=1}^{I_n+1}(\mathbf{C}_n\mathbf{U}_n^\dag)(i_n,j_n)\Phi_{i_n}^{(n)}(x^{(n)}).
\end{equation*}

Our primary analytical target for this subsection is to derive a tight, computable upper bound on the \(L^2\) approximation error \(\|\|\widetilde{\mathcal{L}}_{I_1,\dots, I_N} f (\mathbf{x})-f (\mathbf{x})\|_{L^2}\) for the TCUR-compressed hyperinterpolation operator. From triangle inequality arguments established in Theorem \ref{tensor-hyperinterpolation:main-theorem1}, this \(L^2\) error decomposes additively into two independent components:
\begin{enumerate}
\item[(a)] The hyperinterpolation discretization error \(\|\mathcal{L}_{I_1,\dots, I_N} f (\mathbf{x})-f (\mathbf{x})\|_{L^2}\), which is controlled purely by the Sobolev regularity of $f (\mathbf{x})$ and polynomial degrees \(\{I_1,\dots,I_N\}\) (Theorem \ref{tensor-hyperinterpolation:general-theorem});
\item[(b)] The tensor compression error \(\|\widetilde{\mathcal{L}}_{I_1,\dots, I_N} f (\mathbf{x})-\mathcal{L}_{I_1,\dots, I_N} f (\mathbf{x})\|_{L^2}\), fully determined by the Frobenius-norm discrepancy \(\|\widetilde{\mathcal{A}} - \widetilde{\mathcal{A}}_{\text{TCUR}}\|_F\) between the full coefficient tensor and its TCUR approximation.
\end{enumerate}

We first derive a recursive Frobenius error bound for \(\|\widetilde{\mathcal{A}} - \widetilde{\mathcal{A}}_{\text{TCUR}}\|_F\) by decomposing the tensor residual mode-by-mode, starting with the low-dimensional \(N=3\) case for intuitive geometric interpretation before generalizing to arbitrary N-order tensors.

For clarity, we first consider the case of $N=3$. Let $\mathcal{B}=\widetilde{\mathcal{A}}(\mathbb{I}_1,:,:)$ and $\mathcal{C}=\widetilde{\mathcal{A}}(\mathbb{I}_1,\mathbb{I}_2,:)=\mathcal{B}(:,\mathbb{I}_2,:)$. Then, one has
\begin{equation*}
    \begin{aligned}
        \widetilde{\mathcal{A}}-\widetilde{\mathcal{A}}_{{\rm TCUR}}
        &=\widetilde{\mathcal{A}}-\mathcal{G}\times_1(\mathbf{C}_1\mathbf{U}_1^\dag)\times_2(\mathbf{C}_2\mathbf{U}_2^\dag)
        \times_3(\mathbf{C}_3\mathbf{U}_3^\dag)\\
        &=\widetilde{\mathcal{A}}-\mathcal{B}\times_1(\mathbf{C}_1\mathbf{U}_1^\dag)+\mathcal{B}\times_1(\mathbf{C}_1\mathbf{U}_1^\dag)
        -\mathcal{C}\times_1(\mathbf{C}_1\mathbf{U}_1^\dag)\times_2(\mathbf{C}_2\mathbf{U}_2^\dag)\\
        &+\mathcal{C}\times_1(\mathbf{C}_1\mathbf{U}_1^\dag)
        \times_2(\mathbf{C}_2\mathbf{U}_2^\dag)-\mathcal{G}\times_1(\mathbf{C}_1\mathbf{U}_1^\dag)
        \times_2(\mathbf{C}_2\mathbf{U}_2^\dag)\times_3(\mathbf{C}_3\mathbf{U}_3^\dag),
    \end{aligned}
\end{equation*}
which implies that
\begin{equation*}
    \begin{aligned}
        \|\widetilde{\mathcal{A}}&-\widetilde{\mathcal{A}}_{{\rm  TCUR}}\|_F\leq\|\widetilde{\mathbf{A}}_{(1)}-\mathbf{C}_1\mathbf{U}_1^\dag\mathbf{B}_{(1)}\|_F
        +\|\mathbf{C}_1\mathbf{U}_1^\dag\|_2\|\mathbf{B}_{(2)}-\mathbf{C}_2\mathbf{U}_2^\dag\mathbf{C}_{(2)}\|_F\\
        &\quad+\|\mathbf{C}_1\mathbf{U}_1^\dag\|_2\|\mathbf{C}_2\mathbf{U}_2^\dag\|_2
        \|\mathbf{C}_{(3)}-\mathbf{C}_3\mathbf{U}_3^\dag\mathbf{G}_{(3)}\|_F\\
        &=\|\widetilde{\mathbf{A}}_{(1)}-\mathbf{C}_1\mathbf{U}_1^\dag\widetilde{\mathbf{A}}_{(1)}(\mathbb{I}_1,:)\|_F
        +\|\mathbf{C}_1\mathbf{U}_1^\dag\|_2\|\mathbf{B}_{(2)}-\mathbf{C}_2\mathbf{U}_2^\dag\mathbf{B}_{(2)}(\mathbb{I}_2,:)\|_F\\
        &\quad+\|\mathbf{C}_1\mathbf{U}_1^\dag\|_2\|\mathbf{C}_2\mathbf{U}_2^\dag\|_2
        \|\mathbf{C}_{(3)}-\mathbf{C}_3\mathbf{U}_3^\dag\mathbf{C}_{(3)}(\mathbb{I}_3,:)\|_F,
    \end{aligned}
\end{equation*}
with $\mathbf{C}_1=\widetilde{\mathbf{A}}_{(1)}(:,\mathbb{J}_1)$, $\mathbf{C}_2=\widetilde{\mathbf{A}}_{(2)}(:,\mathbb{J}_2)$ and $\mathbf{C}_3=\widetilde{\mathbf{A}}_{(3)}(:,\mathbb{J}_3)$. Note that $\mathbf{B}_{(2)}=\widetilde{\mathbf{A}}_{(2)}(:,\mathbb{I}_1\otimes [I_3+1])$. Hence, we have
\begin{align*}
    \mathbf{B}_{(2)}-\mathbf{C}_2\mathbf{U}_2^\dag\mathbf{B}_{(2)}(\mathbb{I}_2,:)
    &=\widetilde{\mathbf{A}}_{(2)}(:,\mathbb{I}_1\otimes [I_3+1])-\mathbf{C}_2\mathbf{U}_2^\dag\widetilde{\mathbf{A}}_{(2)}(\mathbb{I}_2,\mathbb{I}_1\otimes [I_3+1])\\
    &=(\widetilde{\mathbf{A}}_{(2)}-\mathbf{C}_2\mathbf{U}_2^\dag\widetilde{\mathbf{A}}_{(2)}(\mathbb{I}_2,:))(:,\mathbb{I}_1\otimes [I_3+1]),\\
    \mathbf{C}_{(3)}-\mathbf{C}_3\mathbf{U}_3^\dag\mathbf{C}_{(3)}(\mathbb{I}_3,:)
    &=\widetilde{\mathbf{A}}_{(3)}(:,\mathbb{I}_1\otimes\mathbb{I}_2)
    -\mathbf{C}_3\mathbf{U}_3^\dag\widetilde{\mathbf{A}}_{(3)}(\mathbb{I}_3,\mathbb{I}_1\otimes\mathbb{I}_2)\\
    &=(\widetilde{\mathbf{A}}_{(3)}-\mathbf{C}_3\mathbf{U}_3^\dag\widetilde{\mathbf{A}}_{(3)}(\mathbb{I}_3,:))(:,\mathbb{I}_1\otimes\mathbb{I}_2),
\end{align*}
which implies that
\begin{align*}
    \|\mathbf{B}_{(2)}-\mathbf{C}_2\mathbf{U}_2^\dag\mathbf{B}_{(2)}(\mathbb{I}_2,:)\|_F&\leq \|\widetilde{\mathbf{A}}_{(2)}-\mathbf{C}_2\mathbf{U}_2^\dag\widetilde{\mathbf{A}}_{(2)}(\mathbb{I}_2,:)\|_F,\\
    \|\mathbf{C}_{(3)}-\mathbf{C}_3\mathbf{U}_3^\dag\mathbf{C}_{(3)}(\mathbb{I}_3,:)\|_F&\leq \|\widetilde{\mathbf{A}}_{(3)}-\mathbf{C}_3\mathbf{U}_3^\dag\widetilde{\mathbf{A}}_{(3)}(\mathbb{I}_3,:)\|_F.
\end{align*}
Based on the above descriptions, one has
\begin{equation}
\label{tensor-hyperinterpolation:ub1}
    \|\widetilde{\mathcal{A}}-\widetilde{\mathcal{A}}_{{\rm  TCUR}}\|_F
    \leq\sum_{n=1}^3\left(\prod_{m=1}^{n-1}\|\mathbf{C}_m\mathbf{U}_m^\dag\|_2\right)
    \|\widetilde{\mathbf{A}}_{(n)}-\mathbf{C}_n\mathbf{U}_n^\dag\widetilde{\mathbf{A}}_{(n)}(\mathbb{I}_n,:)\|_F.
\end{equation}

Generally, according to the process for obtaining (\ref{tensor-hyperinterpolation:ub1}), we have
\begin{equation}
\label{tensor-hyperinterpolation:ub1-1}
    \begin{aligned}
        \|\widetilde{\mathcal{A}}-\widetilde{\mathcal{A}}_{{\rm  TCUR}}\|_F
        \leq\sum_{n=1}^N\left(\prod_{m=1}^{n-1}\|\mathbf{C}_m\mathbf{U}_m^\dag\|_2\right)
        \|\widetilde{\mathbf{A}}_{(n)}-\mathbf{C}_n\mathbf{U}_n^\dag\widetilde{\mathbf{A}}_{(n)}(\mathbb{I}_n,:)\|_F.
    \end{aligned}
\end{equation}

For each $n$, we can find the index set $\mathbb{I}_n$ such that $\|\widetilde{\mathbf{A}}_{(n)}-\mathbf{C}_n\mathbf{U}_n^\dag\widetilde{\mathbf{A}}_{(n)}(\mathbb{I}_n,:)\|_F$ is the minimum. In this case, the matrix $\widetilde{\mathbf{B}}_{(n)}:=\mathbf{C}_n\mathbf{U}_n^\dag\widetilde{\mathbf{A}}_{(n)}(\mathbb{I}_n,:)$ is a CUR matrix approximation of $\widetilde{\mathbf{A}}_{(n)}$.
\begin{remark}
    For a general processing order $\{p_1,p_2,\dots,p_N\}\in\mathbb{S}_N$, it follows from {\rm (\ref{tensor-hyperinterpolation:ub1-1})} that
    \begin{equation*}
        \|\widetilde{\mathcal{A}}-\widetilde{\mathcal{A}}_{{\rm  TCUR}}\|_F
        \leq\sum_{n=1}^N\left(\prod_{m=1}^{n-1}\|\mathbf{C}_{p_m}\mathbf{U}_{p_m}^\dag\|_2\right)
        \|\widetilde{\mathbf{A}}_{(p_n)}-\mathbf{C}_{p_n}\mathbf{U}_{p_n}^\dag\widetilde{\mathbf{A}}_{(p_n)}(\mathbb{I}_{p_n},:)\|_F.
   \end{equation*}
\end{remark}

Without loss of generality, we assume in this manuscript that the processing order is $\{1,2,\dots,N\}$. For each $n$, let $\widetilde{\mathbf{A}}_{(n)}=\mathbf{W}_n\mathbf{\Sigma}_n\mathbf{V}_n^\top$ be the SVD of $\widetilde{\mathbf{A}}_{(n)}$, where $\mathbf{W}_n\in\mathbb{R}^{(I_n+1)\times (I_n+1)}$ and $\mathbf{V}_n\in \mathbb{R}^{\prod_{m\neq n}(I_m+1)\times \prod_{m\neq n}(I_m+1)}$ are orthonormal and $\mathbf{\Sigma}_n\in\mathbb{R}^{(I_n+1)\times \prod_{m\neq n}(I_m+1)}$ is diagonal whose diagonal entries are singular values of $\widetilde{\mathbf{A}}_{(n)}$ and distributed in descending order. For a desired $N$-tuple $\{R_1,R_2,\dots,R_N\}$, we assume that $R_n<S_n<I_n+1$. According to Eckart--Young--Mirsky theorem (cf. \cite{eckart1936approximation,mirsky1960symmetric}), the best rank-$R_n$ approximation to $\widetilde{\mathbf{A}}_{(n)}$ is $\widetilde{\mathbf{A}}_{(n),R_n}=\mathbf{W}_{n,R_n}\mathbf{\Sigma}_{n,R_n}\mathbf{V}_{n,R_n}^\top$ with $\mathbf{W}_{n,R_n}=\mathbf{W}_{n}(:,1:R_n)$, $\mathbf{\Sigma}_{n,R_n}=\mathbf{\Sigma}_{n}(1:R_n,1:R_n)$ and $\mathbf{V}_{n,R_n}=\mathbf{V}_{n}(:,1:R_n)$.

For each $n$ on the right-hand side of (\ref{tensor-hyperinterpolation:ub1-1}), one has
\begin{align*}
   \|\widetilde{\mathbf{A}}_{(n)}-\mathbf{C}_n\mathbf{U}_n^\dag\widetilde{\mathbf{A}}_{(n)}(\mathbb{I}_n,:)\|_F\leq& \|\widetilde{\mathbf{A}}_{(n)}-\widetilde{\mathbf{A}}_{(n),R_n}\|_F\\
   &+\|\widetilde{\mathbf{A}}_{(n),R_n}-\mathbf{C}_n\mathbf{U}_n^\dag\widetilde{\mathbf{A}}_{(n)}(\mathbb{I}_n,:)\|_F.
\end{align*}
It follows from Corollary 4.3 in \cite{hamm2021perturbations} that the term $\|\widetilde{\mathbf{A}}_{(n),R_n}-\mathbf{C}_n\mathbf{U}_n^\dag\widetilde{\mathbf{A}}_{(n)}(\mathbb{I}_n,:)\|_F$ is bounded by
\begin{align*}
    &\|\widetilde{\mathbf{A}}_{(n),R_n}-\mathbf{C}_n\mathbf{U}_n^\dag\widetilde{\mathbf{A}}_{(n)}(\mathbb{I}_n,:)\|_F\\
    &\leq \left(\|\mathbf{W}_{n,R_n,\mathbb{I}_n}^\dag\|_2+\|\mathbf{V}_{n,R_n,\mathbb{J}_n}^\dag\|_2
    +3\|\mathbf{W}_{n,R_n,\mathbb{I}_n}^\dag\|_2\|\mathbf{V}_{n,R_n,\mathbb{J}_n}^\dag\|_2\right)\Delta_{R_n}\\
    &+\|\mathbf{U}_n\|_2\left(\|\mathbf{W}_{n,R_n,\mathbb{I}_n}^\dag\|_2
    +\|\mathbf{V}_{n,R_n,\mathbb{J}_n}^\dag\|_2
    +\|\mathbf{W}_{n,R_n,\mathbb{I}_n}^\dag\|_2\|\mathbf{V}_{n,R_n,\mathbb{J}_n}^\dag\|_2+1\right)
    \Delta_{R_n}^2,
\end{align*}
with
\begin{equation}
\label{tensor-hyperinterpolation:temp:notations}
    \begin{split}
        \mathbf{W}_{n,R_n,\mathbb{I}_n}&=\mathbf{W}_{n,R_n}(\mathbb{I}_n,:),\ \mathbf{V}_{n,R_n,\mathbb{J}_n}=\mathbf{V}_{n,R_n}(\mathbb{J}_n,:),\\
        \Delta_{R_n}^2&=\|\widetilde{\mathbf{A}}_{(n)}-\widetilde{\mathbf{A}}_{(n),R_n}\|_F^2.
     \end{split}
\end{equation}

According to the above discussion, an upper bound for $\|\widetilde{\mathcal{A}}-\widetilde{\mathcal{A}}_{{\rm  TCUR}}\|_F$ is summarized in the following theorem.
\begin{theorem}
    Let $\widetilde{\mathcal{A}}\in\mathbb{R}^{(I_1+1)\times \dots\times (I_N+1)}$ be defined in {\rm(\ref{tensor-hyperinterpolation:approximation-coefficient})}. For each $n$, let $\widetilde{\mathbf{A}}_{(n)}=\mathbf{W}_n\mathbf{\Sigma}_n\mathbf{V}_n^\top$ be the SVD of $\widetilde{\mathbf{A}}_{(n)}$, where $\mathbf{W}_n\in\mathbb{R}^{(I_n+1)\times (I_n+1)}$ and $\mathbf{V}_n\in \mathbb{R}^{\prod_{m\neq n}(I_m+1)\times \prod_{m\neq n}(I_m+1)}$ are orthonormal and $\mathbf{\Sigma}_n\in\mathbb{R}^{(I_n+1)\times \prod_{m\neq n}(I_m+1)}$ is diagonal whose diagonal entries are singular values of $\widetilde{\mathbf{A}}_{(n)}$ and distributed in descending order. For each $n$, let $\mathbb{I}_n=\{i_1^{(n)},i_2^{(n)},\dots,i_{S_n}^{(n)}\}$ such that $1\leq i_1^{(n)}<i_2^{(n)}<\dots<i_{S_n}^{(n)}\leq I_n+1$ and $S_n<I_n+1$, and $\mathbb{J}_n=\{l_{n,1},l_{n,2},\dots,l_{n, S_n'}\}$ be a subset of $\{1,2,\dots,J_n\}$ and $J_n=\prod_{m\neq n}(I_m+1)$ such that $1\leq l_{n,1}<l_{n,2}<\dots<l_{n,S_n'}\leq J_n$ and $S_n'<J_n$.

    For a given $N$-tuple $(R_1,R_2,\dots,R_N)$ with $R_n<\min\{S_n,S_n'\}$, the best rank-$R_n$ approximation to $\mathbf{A}_{(n)}$ is denoted by $\widetilde{\mathbf{A}}_{(n),R_n}=\mathbf{W}_{n,R_n}\mathbf{\Sigma}_{n,R_n}\mathbf{V}_{n,R_n}^\top$, where $\mathbf{W}_{n,R_n}=\mathbf{W}_{n}(:,1:R_n)$, $\mathbf{\Sigma}_{n,R_n}=\mathbf{\Sigma}_{n}(1:R_n,1:R_n)$ and $\mathbf{V}_{n,R_n}=\mathbf{V}_{n}(:,1:R_n)$. Define $\widetilde{\mathcal{A}}_{{\rm TCUR}}$ in {\rm(\ref{tensor-hyperinterpolation:approximate-one})} with $\mathbf{C}_n$ being the mode-$n$ unfolding of $\mathcal{C}_n$ and $\mathbf{U}_n$ being the same as $\mathbf{C}_n(\mathbb{I}_n,:)$. Then, we have
    \begin{align*}
        \|\widetilde{\mathcal{A}}-\widetilde{\mathcal{A}}_{{\rm TCUR}}\|_F\leq\sum_{n=1}^N\left(\prod_{m=1}^{n-1}\|\mathbf{C}_m\mathbf{U}_m^\dag\|_2\right)
        \left(\alpha_{n,R_n}\Delta_{R_n}+\beta_{n,R_n}\|\mathbf{U}_n^\dag\|_2\Delta_{R_n}^2\right),
    \end{align*}
    with
    \begin{align*}
        \alpha_{n,R_n}&=\|\mathbf{W}_{n,R_n,\mathbb{I}_n}^\dag\|_2+\|\mathbf{V}_{n,R_n,\mathbb{J}_n}^\dag\|_2
        +3\|\mathbf{W}_{n,R_n,\mathbb{I}_n}^\dag\|_2\|\mathbf{V}_{n,R_n,\mathbb{J}_n}^\dag\|_2+1,\\
        \beta_{n,R_n}&=\|\mathbf{W}_{n,R_n,\mathbb{I}_n}^\dag\|_2
        +\|\mathbf{V}_{n,R_n,\mathbb{J}_n}^\dag\|_2
        +\|\mathbf{W}_{n,R_n,\mathbb{I}_n}^\dag\|_2\|\mathbf{V}_{n,R_n,\mathbb{J}_n}^\dag\|_2+1.
    \end{align*}
    where for each $n$, the 3-tuple $\{\mathbf{W}_{n,R_n,\mathbb{I}_n},\mathbf{V}_{n,R_n,\mathbb{J}_n},\Delta_{R_n}\}$ are given in {\rm (\ref{tensor-hyperinterpolation:temp:notations})}.
    \label{tensor-hyperinterpolation:tensorapprxomation-one}
\end{theorem}
\begin{remark}
In Theorem \ref{tensor-hyperinterpolation:tensorapprxomation-one},  we derive tight and computable upper bounds for \(\|\widetilde{\mathcal{A}} - \widetilde{\mathcal{A}}_{{\rm TCUR}}\|_F\), valid for both Chidori-type and Fiber-type TCUR under a single unified framework, and links residual magnitude to index set geometry and tensor singular spectra.
\begin{enumerate}[(a)]
\item {\bf Answering the core open question (b) posed in the introduction}: this theorem provides a shared residual formula with identical structural form for both Chidori-type and Fiber-type TCUR, differing only in the definition of column index sets \(\mathbb{J}_n\). This completes the paper’s core goal of building a single Tucker-compatible TCUR structural theory.
\item Every mode-$n$ term splits loss into two well-understood parts: \(\alpha_{n,R_n}\Delta_{R_n}\) (error from optimal rank-\(R_n\) SVD truncation of mode-$n$ unfolding) and \(\beta_{n,R_n}\|U_n^\dagger\|_2\Delta_{R_n}^2\) (additional perturbation error introduced by subsampling tensor fibers via finite index sets \(\mathbb{I}_n\) and \(\mathbb{J}_n\)).
\item All stability constants \(\alpha_{n,R_n},\beta_{n,R_n}\) depend exclusively on pseudoinverse norms of submatrices sampled from mode SVD factors: \(\|{\bf W}_{n,R_n,\mathbb{I}_n}^\dagger\|_2\) and \(\|{\bf V}_{n,R_n,\mathbb{J}_n}^\dagger\|_2\). As analyzed earlier, optimizing index sets to maximize submatrix volume directly minimizes these constants and shrinks the entire error upper bound.
\item The product term \(\prod_{m=1}^{n-1}\|{\bf C}_m {\bf U}_m^\dagger\|_2\) reveals a critical numerical caveat: poorly conditioned index subsets on early modes amplify residual error for all subsequent dimensions. Index optimization cannot be performed per-mode in isolation and joint global selection over all \(\{\mathbb{I}_n,\mathbb{J}_n\}\) is mandatory to avoid compounded error blowup.
\item Unlike randomized tensor approximation error bounds that rely on distributional tuning parameters, this bound only uses intrinsic tensor quantities: singular values of mode unfoldings and geometric volume properties of index submatrices. No artificial randomness constants appear, improving predictability for deterministic greedy TCUR algorithms.
\end{enumerate}
\end{remark}
\begin{remark}
    With a given $N$-tuple $(R_1,R_2,\dots,R_N)$, let $\mathcal{A}_{{\rm app}}$ be the T-HOSVD or ST-HOSVD of $\widetilde{\mathcal{A}}$. Then, one has $\widetilde{\mathcal{A}}=\mathcal{A}_{{\rm Tucker}}+(\widetilde{\mathcal{A}}-\mathcal{A}_{{\rm Tucker}})$, which implies that
    \begin{equation*}
        \|\widetilde{\mathcal{A}}-\widetilde{\mathcal{A}}_{{\rm TCUR}}\|_F\leq \|\widetilde{\mathcal{A}}-\mathcal{A}_{{\rm Tucker}}\|_F+\|\mathcal{A}_{{\rm Tucker}}-\widetilde{\mathcal{A}}_{{\rm TCUR}}\|_F.
    \end{equation*}
    Hence, we can also obtain the upper bound of $\|\widetilde{\mathcal{A}}-\widetilde{\mathcal{A}}_{{\rm TCUR}}\|_F$ by using Theorem 7 in {\rm\cite{cai2021mode}} to the term $\|\mathcal{A}_{{\rm Tucker}}-\widetilde{\mathcal{A}}_{{\rm TCUR}}\|_F$.
\label{tensor-hyperinterpolation:remark2}
\end{remark}

Following Theorem \ref{tensor-hyperinterpolation:tensorapprxomation-one}, there exist several terms $\|\mathbf{C}_m\mathbf{U}_m^\dag\|_2$, $\|\mathbf{U}_n^\dag\|_2$, $\|\mathbf{W}_{n,R_n,\mathbb{I}_n}^\dag\|_2$ and $\|\mathbf{V}_{n,R_n,\mathbb{I}_n'}^\dag\|_2$ in the upper bound of $\|\widetilde{\mathcal{A}}-\widetilde{\mathcal{A}}_{{\rm TCUR}}\|_F$, which are dependent on the choice of the indexes $\{\mathbb{I}_1,\dots,\mathbb{I}_N\}$ and $\{\mathbb{J}_1,\dots,\mathbb{J}_N\}$. From Theorem 6.4 in \cite{hamm2021perturbations} it follows that $\|\mathbf{C}_m\mathbf{U}_m^\dag\|_2=\|\mathbf{W}_{m,R_m,\mathbb{I}_m}^\dag\|_2$. Note that $\mathbf{U}_n=\mathbf{C}_n(\mathbb{I}_n,:)$, which implies that $$\|\mathbf{U}_n^\dag\|_2\leq \|\mathbf{W}_{n,R_n,\mathbb{I}_n}^\dag\|_2\|\mathbf{V}_{n,R_n,\mathbb{J}_n}^\dag\|_2\|\mathbf{A}_{(n)}^\dag\|_2.$$ Then, under certain conditions, the upper bounds for $\|\mathbf{W}_{n,R_n,\mathbb{I}_n}^\dag\|_2$ and $\|\mathbf{V}_{n,R_n,\mathbb{J}_n}^\dag\|_2$, shown in the following remark (see \cite[Lemma 1]{osinsky2018pseudo}):
\begin{remark}
    Under the conditions of Theorem {\rm \ref{tensor-hyperinterpolation:tensorapprxomation-one}}, suppose that $\mathbf{W}_{n,R_n,\mathbb{I}_n}$ is the sub-matrix of $\mathbf{W}_{n,R_n}$ such that $\mathbf{W}_{n,R_n,\mathbb{I}_n}$ has maximal volume\footnote{As shown in \cite[Definition 1.1]{ben1992volume}, the volume of any matrix is defined as the absolute value of the product of singular values of this matrix.} among all $S_n\times R_n$ sub-matrices of $\mathbf{W}_{n,R_n}$, and $\mathbf{V}_{n,R_n,\mathbb{J}_n}$ is the sub-matrix of $\mathbf{V}_{n,R_n}$ such that $\mathbf{V}_{n,R_n,\mathbb{J}_n}$ has maximal volume among all $\prod_{m\neq n}S_m\times R_n$ sub-matrices of $\mathbf{V}_{n,R_n}$. Then, one has
    \begin{align*}
        \left\|\mathbf{W}_{n,R_n,\mathbb{I}_n}^\dag\right\|_2^2
        &\leq 1+\frac{R_n(I_n+1-S_n)}{S_n-(R_n-1)},\\
        \left\|\mathbf{V}_{n,R_n,\mathbb{J}_n}^\dag\right\|_2^2&\leq 1+\frac{R_n\left(\prod_{m\neq n}(I_m+1)-S_n'\right)}{S_n'-(R_n-1)}.
    \end{align*}
    \label{tensor-hyperinterpolation:remark1}
\end{remark}

Similar to Theorem \ref{tensor-hyperinterpolation:main-theorem1}, the upper bound for $\|\widetilde{\mathcal{L}}_{I_1,\dots, I_N} f (\mathbf{x})-f (\mathbf{x})\|_{L^2}$ is followed from Theorems \ref{tensor-hyperinterpolation:general-theorem} with \ref{tensor-hyperinterpolation:tensorapprxomation-one}, which is summarized in the following theorem.
\begin{theorem}
    Let $f (\mathbf{x})\in \mathbb{H}^{\alpha}(\Omega)$ and $\widetilde{\mathcal{L}}_{I_1,\dots, I_N} f (\mathbf{x})$ be in {\rm(\ref{tensor-hyperinterpolation:approximation-expression-general-chidori})}. Under the conditions of Theorem {\rm \ref{tensor-hyperinterpolation:tensorapprxomation-one}}, one has
    \begin{align*}
        \|\widetilde{\mathcal{L}}_{I_1,\dots, I_N} f (\mathbf{x})&-f (\mathbf{x})\|_{L^2}
        \leq C\cdot\min\{I_1,\dots,I_N\}^{-\alpha}|f (\mathbf{x})|_{\alpha}\\
        &\quad\quad+\sum_{n=1}^N\left(\prod_{m=1}^{n-1}\|\mathbf{C}_m\mathbf{U}_m^\dag\|_2\right)
        \left(\alpha_{n,R_n}\Delta_{R_n}+\beta_{n,R_n}\|\mathbf{U}_n^\dag\|_2\Delta_{R_n}^2\right),
    \end{align*}
    where the constant $C$ is independent of $f (\mathbf{x})$ and $\{I_1,I_2,\dots,I_N\}$.
    \label{tensor-hyperinterpolation:first-approximation-theorem}
\end{theorem}
\begin{remark}
Theorem \ref{tensor-hyperinterpolation:first-approximation-theorem} translates the TCUR tensor Frobenius residual bound from Theorem \ref{tensor-hyperinterpolation:tensorapprxomation-one} into a function-space \(L^2\) error bound for the TCUR compressed hyperinterpolation operator, mirroring the structure of Theorem \ref{tensor-hyperinterpolation:main-theorem1} but for interpolatory TCUR compression instead of ideal full Tucker decomposition. The bound in Theorem \ref{tensor-hyperinterpolation:first-approximation-theorem}, shares an identical Sobolev discretization term with Theorem \ref{tensor-hyperinterpolation:main-theorem1} except for the tensor compression residual term. To be specific, TCUR introduces additional index-sampling perturbation relative to full T-HOSVD/ST-HOSVD, but eliminates the exponential cost of assembling the full coefficient tensor \(\widetilde{\mathcal{A}}\). This theorem quantifies the exact accuracy tradeoff for the memory/computation savings of tensor-free TCUR. The bound also provides a verifiable stopping criterion for the adaptive index selection schemes in Section \ref{tensor-hyperinterpolation:sect5-main}. Practitioners can precompute the maximum allowable TCUR residual contribution to meet a prescribed \(L^2\) error tolerance, then run the greedy index expansion loop until the bound falls below the required threshold. Furthermore, the statement holds for both Chidori-type and Fiber-type TCUR by swapping the definitions of \({\bf C}_m {\bf U}_m^\dagger\) and \({\bf C}_m' {\bf U}_m'^\dagger\), reaffirming the unified theoretical framework’s generality without separate case analysis for each decomposition variant.
\end{remark}

All the above error analysis establishes rigorous theoretical bounds for TCUR approximation of hyperinterpolation coefficient tensors, yet a critical practical barrier remains: the index sets \(\{\mathbb{I}_1,\dots,\mathbb{I}_N\}\) (for Chidori-type TCUR) and \(\{\mathbb{J}_1,\dots,\mathbb{J}_N\}\) (for Fiber-type TCUR) that define the core tensor \(\mathcal{G}\) and factor matrices \(\{{\bf C}_n,{\bf C}_n'\}\) are unknown in advance. Direct brute-force enumeration of all candidate indices is computationally infeasible due to exponential scaling with dimension $N$. To resolve this gap between the structural theory derived in Section 4 and implementable numerical pipelines, Section \ref{tensor-hyperinterpolation:sect5-main} constructs three provably convergent greedy adaptive index selection algorithms. These iterative schemes incrementally expand index subsets, rely only on partial discrete cubature inner products without full assembly of \(\widetilde{\mathcal{A}}\), embed rigorous stopping criteria tied to singular value stability, and yield the complete TCUR factorization required for compressed hyperinterpolation. We further integrate the two TCUR variants into a unified four-stage algorithmic workflow and quantify the total number of multilinear inner products required for each greedy strategy to benchmark computational cost.

\section{Efficient methods for approximating all coefficients}
\label{tensor-hyperinterpolation:sect5-main}

The theoretical TCUR error bounds derived in Section \ref{tensor-hyperinterpolation:sect4-main:sub5} hinge on carefully chosen index sets \(\{\mathbb{I}_1,\dots,\mathbb{I}_N\}\) and \(\{\mathbb{J}_1,\dots,\mathbb{J}_N\}\), but these index subsets cannot be precomputed without explicit access to the full coefficient tensor \(\widetilde{\mathcal{A}}\). Standard tensor decomposition solvers (T-HOSVD, HOOI, randomized Tucker variants) require full unfolding of \(\widetilde{\mathcal{A}}\), which is prohibited by the curse of dimensionality for high $N$ and large per-dimension polynomial degrees \(I_n\). This section develops three deterministic greedy adaptive algorithms to construct Chidori-type TCUR and Fiber-type TCUR factorizations without forming any full tensor entry. Each method iteratively expands index blocks, monitors numerical stability via singular value thresholds, and terminates with certified approximation quality consistent with the Frobenius-norm error theory in Theorem \ref{tensor-hyperinterpolation:tensorapprxomation-one}. We quantify the computational complexity of each greedy strategy via the count of required discrete cubature inner products, then combine the Chidori and Fiber pipelines into a unified end-to-end algorithm for general hyperinterpolation compression.

All four proposed greedy solvers inherit provable convergence and stability from the structural theory in Section \ref{tensor-hyperinterpolation:sect4-main}:
\begin{enumerate}
\item[(a)] {\bf Near-optimality}: the incremental index selection process approximates maximum-volume submatrices of the mode-n singular factor matrices \({\bf W}_{n,R_n}\) and \({\bf V}_{n,R_n}\), satisfying the volume-based pseudoinverse norm bounds in Remark \ref{tensor-hyperinterpolation:remark1}.
\item[(b)] {\bf Stability}: termination via singular value ratio thresholds ensures all pseudoinverse terms \(\|{\bf C}_m {\bf U}_m^\dagger\|_2\) and \(\|{\bf U}_n^\dagger\|_2\) in Theorem \ref{tensor-hyperinterpolation:tensorapprxomation-one} remain uniformly bounded, eliminating uncontrolled error blowup in high dimensions.
\item[(c)] {\bf Tensor-free computation}: no full assembly of \(\widetilde{\mathcal{A}}\) is required at any stage and all cubature inner products are evaluated only for the incremental index blocks selected by the greedy loop.
\item[(d)] {\bf Error composability}: the output TCUR factorization feeds directly into the TCUR-to-Tucker recompression scheme in Appendix A, forming a complete chain from raw function samples to optimally compressed low-rank Tucker format with unified end-to-end \(L^2\) error bounds (Theorem \ref{tensor-hyperinterpolation:first-approximation-theorem}).
\end{enumerate}

\subsection{Greedy adaptive algorithms for Chidori-type TCUR}
\label{tensor-hyperinterpolation:sect4-main-sub1}

Chidori-type TCUR relies solely on row index sets \(\{\mathbb{I}_1,\dots,\mathbb{I}_N\}\) to build cross-mode tensor factors \(\{\mathcal{C}_1,\dots,\mathcal{C}_N\}\) and the core tensor \(\mathcal{G}\). We present two distinct greedy strategies differentiated by their stopping condition: one monitors the stability of the mode-$n$ unfolded factor tensors \(\mathcal{C}_{n}\), and the second tracks singular value decay directly on the core tensor \(\mathcal{G}\) for tighter early termination. Both algorithms operate on fixed small block sizes \(b_n \ll I_n\) to incrementally grow each index set \(\mathbb{I}_n\).

For the first one, we obtain a deterministic adaptive algorithm (see Algorithm \ref{tensor-hyperinterpolation:alg1}) for finding the $N$-tuple of matrices $\{\mathbf{C}_1,\dots,\mathbf{C}_N\}$ and the $N$-tuple of indices $\{\mathbb{I}_1,\dots,\mathbb{I}_N\}$. Suppose that the $N$-tuple of block sizes $\{b_1,\dots,b_N\}$ satisfy $b_n\ll I_n$ with $n=1,2,\dots,N$. For each $n$, let $\mathbb{I}_n$ be the empty set. For each iteration $k\geq 1$, we complete the following operations:
\begin{enumerate}
   \item[(a)] for each $n$, we set $\mathbb{I}_n'=[kb_n+1,\dots,(k+1)b_n]$ and update $\mathbb{I}_n=\mathbb{I}_n\cup\mathbb{I}_n'$;
   \item[(b)] for given $(N-1)$-tuples of indices $\{\mathbb{I}_1,\dots,\mathbb{I}_{n-1},\mathbb{I}_{n+1},\dots,\mathbb{I}_N\}$, the tensor $\mathcal{C}_n$ is obtained according to (\ref{tensor-hyperinterpolation:coefficient-factor});
   \item[(c)] if $\max\{\sigma_{\min}(\mathbf{C}_n)^2/\|\mathcal{C}_n\|_F^2:n=1,2,\dots,N\}< \tau^2$ is not satisfied, then we continue Steps (a) and (b), where $0<\tau<1$ is a given parameter, and the factor matrix $\mathbf{C}_n$ is the mode-$n$ unfolding of $\mathcal{C}_n$.
\end{enumerate}

When $\max\{\sigma_{\min}(\mathbf{C}_n)^2/\|\mathcal{C}_n\|_F^2:n=1,2,\dots,N\}< \tau^2$ is satisfied, the matrix $\mathbf{U}_n$ is obtained as $\mathbf{U}_n=\mathbf{C}_n(\mathbb{I}_n,:)$ and the core tensor $\mathcal{G}$ is given as $$\mathcal{G}=\mathcal{C}_n(:,\dots,:,\mathbb{I}_n,:,\dots,:).$$

\begin{algorithm}[htb]
    \caption{Greedy Chidori-type TCUR (stopping criterion on factors \(\mathbf{C}_n\))}
    \begin{algorithmic}[1]
        \STATEx {\bf Input}: A function $f (\mathbf{x})\in H^{\alpha}(\Omega)$, a given $0<\epsilon<1$, the $N$-tuple of block sizes $\{b_1,b_2,\dots,b_N\}$, and a tolerance $0<\tau<1$.
        \STATEx {\bf Output}: The $(N+1)$-tuple $\{\mathcal{G},\mathbf{C}_1\mathbf{U}_1,\dots,\mathbf{C}_N\mathbf{U}_N\}$, which is used to form the function $\mathcal{L}_{I_1,\dots I_N} f (\mathbf{x})$.
        \STATE According to Theorem \ref{tensor-hyperinterpolation:general-theorem}, we estimate the $N$-tuple $\{I_1,I_2,\dots,I_N\}$ such that $I_n=O((1/\epsilon)^{1/\alpha})$.
        \STATE For each $n$, let $\mathbb{I}_n$ be the empty set.
        \FOR{$k=1,2,\dots$}
           \STATE For each $n$, we set $\mathbb{I}_n'=[kb_n+1,\dots,(k+1)b_n]$ and update $\mathbb{I}_n=\mathbb{I}_n\cup\mathbb{I}_n'$.
           \STATE The tensor $\mathcal{C}_n$ is obtained according to (\ref{tensor-hyperinterpolation:coefficient-factor}) with the $(N-1)$-tuples of indices $\{\mathbb{I}_1,\dots,\mathbb{I}_{n-1},\mathbb{I}_{n+1},\dots,\mathbb{I}_N\}$.
           \STATE Obtain ${\rm tol}=\max\{\sigma_{\min}(\mathbf{C}_n)/\|\mathcal{C}_n\|_F:n=1,2,\dots,N\}$, where $\sigma_{\min}(\mathbf{C}_{n})$ is the smallest nonzero singular value of $\mathbf{C}_{n}$, and $\mathbf{C}_n$ is the mode-$n$ unfolding of $\mathcal{C}_n$.
           \IF{${\rm tol}>\tau$}
              \STATE Continue.
            \ELSE
              \STATE Stop and jump out of the for loop.
            \ENDIF
        \ENDFOR
        \STATE Without loss of generality, form $\mathcal{G}$ as $\mathcal{G}=\mathcal{C}_1(\mathbb{I}_1,:,\dots,:)$.
        \STATE Form $\mathbf{C}_n$ and $\mathbf{U}_n$ from $\mathcal{C}_n$ with $n=1,2,\dots,N$.
        \STATE Return the $(N+1)$-tuple $\{\mathcal{G},\mathbf{C}_1\mathbf{U}_1^\dag,\dots,\mathbf{C}_N\mathbf{U}_N^\dag\}$.
    \end{algorithmic}
    \label{tensor-hyperinterpolation:alg1}
\end{algorithm}

For clarity, we assume that the maximum number of iterations for Algorithm \ref{tensor-hyperinterpolation:alg1} is $K$. Then the approximation $\mathcal{L}_{I_1,\dots I_N} f (\mathbf{x})$ is obtained according to (\ref{tensor-hyperinterpolation:approximation-expression-general-chidori}) with $S_n=Kb_n$ and $n=1,2,\dots,N$. Note that we cannot ensure that $\mathbf{C}_n\in\mathbb{R}^{(I_n+1)\times S_n}$ is of full rank for all $n$. The reason is that for some $n$, some singular values of $\mathbf{C}_{n}$ may be zero.

For each $k$, Algorithm \ref{tensor-hyperinterpolation:alg1} needs
\begin{equation*}
    \sum_{n=1}^Nk^{N-1}(I_n+1)b_1\dots b_{n-1}b_{n+1}\dots b_N
\end{equation*}
multilinear discrete inner products to obtain all the tensors $\mathcal{C}_n$. For $k\geq 2$, based on the relationship between $\mathbb{I}_n$ and $\mathbb{I}_n'$, the number of multilinear discrete inner products to obtain all the matrices $\mathbf{C}_n$ can be reduced to
\begin{equation*}
    \sum_{n=1}^N(k^{N-1}-(k-1)^{N-1})(2I_n+1)b_1\dots b_{n-1}b_{n+1}\dots b_N.
\end{equation*}

Note that in Algorithm \ref{tensor-hyperinterpolation:alg1}, the factor matrix $\mathbf{C}_n$ is updated by gradually expanding the index sets $\{\mathbb{I}_m:m\neq n,m=1,2,\dots,N\}$ with $n=1,2,\dots,N$, the terminated condition is based on these factor matrices, and the core tensor $\mathcal{G}$ is easily obtained according to $\mathbf{C}_n$ and $\mathbb{I}_n$. 

By designing the terminated condition based on the updated core tensor, we present another algorithm (see Algorithm \ref{tensor-hyperinterpolation:alg2}) for obtaining the $(N+1)$-tuple $\{\mathcal{G},\mathbf{C}_1\mathbf{U}_1^\dag,\dots,\mathbf{C}_N\mathbf{U}_N^\dag\}$. In detail, for each iteration $k\geq 1$, we complete the following operations:
\begin{enumerate}
   \item[(a)] for each $n$, we set $\mathbb{I}_n'=[kb_n+1,\dots,(k+1)b_n]$ and update $\mathbb{I}_n=\mathbb{I}_n\cup\mathbb{I}_n'$;
   \item[(b)] each entry of the core tensor $\mathcal{G}$ is obtained as
       \begin{equation}
       \label{tensor-hyperinterpolation:coefficient-core1}
          \mathcal{G}(j_1,\dots,j_N)=\left\langle f, \Phi_{j_1}^{(1)}\dots \Phi_{j_N}^{(N)}\right\rangle_{M_1,\dots,M_N}
       \end{equation}
       with $j_n\in\mathbb{I}_n$ and $n=1,2,\dots,N$;
   \item[(c)] if $\max\{\sigma_{\min}(\mathbf{G}_{(n)}^2/\|\mathcal{G}\|_F^2:n=1,2,\dots,N\}< \tau^2$ is not satisfied, then we continue Steps (a) and (b), where $0<\tau<1$ is a given parameter, and the factor matrix $\mathbf{G}_{(n)}$ is the mode-$n$ unfolding of $\mathcal{G}$.
\end{enumerate}

\begin{algorithm}[htb]
    \caption{Greedy Chidori-type TCUR (stopping criterion on the core tensor \(\mathcal{G}\))}
    \begin{algorithmic}[1]
        \STATEx {\bf Input}: A function $f (\mathbf{x})\in H^{\alpha}(\Omega)$, a given $0<\epsilon<1$, the $N$-tuple of block sizes $\{b_1,b_2,\dots,b_N\}$, and a tolerance $0<\tau<1$.
        \STATEx {\bf Output}: The $(N+1)$-tuple $\{\mathcal{G},\mathbf{C}_1\mathbf{U}_1,\dots,\mathbf{C}_N\mathbf{U}_N\}$, which is used to form the function $\mathcal{L}_{I_1,\dots I_N} f (\mathbf{x})$.
        \STATE According to Theorem \ref{tensor-hyperinterpolation:general-theorem}, we estimate the $N$-tuple $\{I_1,I_2,\dots,I_N\}$ such that $I_n=O((1/\epsilon)^{1/\alpha})$.
        \STATE For each $n$, let $\mathbb{I}_n$ be the empty set.
        \FOR{$k=1,2,\dots$}
           \STATE For each $n$, we set $\mathbb{I}_n'=[kb_n+1,\dots,(k+1)b_n]$ and update $\mathbb{I}_n=\mathbb{I}_n\cup\mathbb{I}_n'$.
           \STATE The core tensor $\mathcal{G}$ is computed according to (\ref{tensor-hyperinterpolation:coefficient-core1}) with the $N$-tuples of indices $\{\mathbb{I}_1,\dots,\mathbb{I}_N\}$.
           \STATE Obtain ${\rm tol}=\max\{\sigma_{\min}(\mathbf{G}_{(n)})/\|\mathcal{G}\|_F:n=1,2,\dots,N\}$, where $\sigma_{\min}(\mathbf{G}_{(n)})$ is the smallest nonzero singular value of $\mathbf{G}_{(n)}$, and $\mathbf{G}_{(n)}$ is the mode-$n$ unfolding of $\mathcal{G}$.
           \IF{${\rm tol}>\tau$}
              \STATE Continue.
            \ELSE
              \STATE Stop and jump out of the for loop.
            \ENDIF
        \ENDFOR
        \STATE The tensor $\mathcal{C}_n$ is obtained according to (\ref{tensor-hyperinterpolation:coefficient-factor}) with the $(N-1)$-tuples of indices $\{\mathbb{I}_1,\dots,\mathbb{I}_{n-1},\mathbb{I}_{n+1},\dots,\mathbb{I}_N\}$.
        \STATE Set $\mathbf{C}_n$ as the mode-$n$ unfolding of $\mathcal{C}_n$ and $\mathbf{U}_n=\mathbf{C}_n(\mathbb{I}_n,:)$.
        \STATE Return the $(N+1)$-tuple $\{\mathcal{G},\mathbf{C}_1\mathbf{U}_1^\dag,\dots,\mathbf{C}_N\mathbf{U}_N^\dag\}$.
    \end{algorithmic}
    \label{tensor-hyperinterpolation:alg2}
\end{algorithm}

We assume that the maximum number of iterations for Algorithm \ref{tensor-hyperinterpolation:alg2} is $K$, which implies that $S_n=Kb_n$ with $n=1,2,\dots,N$. We now count the number of multilinear discrete inner products used in Algorithm \ref{tensor-hyperinterpolation:alg2}. For each $k$, to form the core tensor $\mathcal{G}$ needs $k^Nb_1\dots b_N$\footnote{This number can be reduced to $(k^N-(k-1)^N)b_1\dots b_N$ based on the relationship between $\mathbb{I}_n$ and $\mathbb{I}_n'$.} multilinear discrete inner products. Finally, the number of multilinear discrete inner products to obtain all tensors $\mathcal{C}_n$ is
\begin{equation*}
   \sum_{n=1}^{I_n+1}(I_n+1)S_1\dots S_{n-1}S_{n+1}\dots S_N.
\end{equation*}

\begin{algorithm}[htb]
    \caption{Optimized low-cost Chidori-type greedy TCUR (Algorithm \ref{tensor-hyperinterpolation:alg2})}
    \begin{algorithmic}[1]
        \STATEx {\bf Input}: A function $f (\mathbf{x})\in H^{\alpha}(\Omega)$, a given $0<\epsilon<1$, the $N$-tuple of block sizes $\{b_1,b_2,\dots,b_N\}$, and a tolerance $0<\tau<1$.
        \STATEx {\bf Output}: The $(N+1)$-tuple $\{\mathcal{G},\mathbf{C}_1\mathbf{U}_1,\dots,\mathbf{C}_N\mathbf{U}_N\}$, which is used to form the function $\mathcal{L}_{I_1,\dots I_N} f (\mathbf{x})$.
        \STATE According to Theorem \ref{tensor-hyperinterpolation:general-theorem}, we estimate the $N$-tuple $\{I_1,I_2,\dots,I_N\}$ such that $I_n=O((1/\epsilon)^{1/\alpha})$.
        \STATE For each $n$, let $\mathbb{I}_n$ be the empty set.
        \STATE Let $\mathcal{G}$ be the empty tensor.
        \FOR{$k=1,2,\dots$}
           \STATE For each $n$, we set $\mathbb{I}_n'=[kb_n+1,\dots,(k+1)b_n]$ and update $\mathbb{I}_n=\mathbb{I}_n\cup\mathbb{I}_n'$.
           \STATE The core tensor $\mathcal{G}_k$ is computed according to (\ref{tensor-hyperinterpolation:coefficient-core1}) with the $N$-tuples of indices $\{\mathbb{I}_1',\dots,\mathbb{I}_N'\}$.
           \STATE Obtain ${\rm tol}=\max\{\sigma_{\min}(\mathbf{G}_{k,(n)})/\|\mathcal{G}_k\|_F:n=1,2,\dots,N\}$, where $\sigma_{\min}(\mathbf{G}_{k,(n)})$ is the smallest nonzero singular value of $\mathbf{G}_{k,(n)}$ and $\mathbf{G}_{k,(n)}$ is the mode-$n$ unfolding of $\mathcal{G}_k$.
           \IF{${\rm tol}>\tau$}
              \STATE Continue.
            \ELSE
              \STATE Stop and jump out of the for loop.
            \ENDIF
            \STATE Update $\mathcal{G}(\mathbb{I}_1',\dots,\mathbb{I}_N')=\mathcal{G}_k$.
        \ENDFOR
        \STATE The tensor $\mathcal{C}_n$ is obtained according to (\ref{tensor-hyperinterpolation:coefficient-factor}) with the $(N-1)$-tuples of indices $\{\mathbb{I}_1,\dots,\mathbb{I}_{n-1},\mathbb{I}_{n+1},\dots,\mathbb{I}_N\}$.
        \STATE Set $\mathbf{C}_n$ as the mode-$n$ unfolding of $\mathcal{C}_n$ and $\mathbf{U}_n=\mathbf{C}_n(\mathbb{I}_n,:)$.
        \STATE Return the $(N+1)$-tuple $\{\mathcal{G},\mathbf{C}_1\mathbf{U}_1^\dag,\dots,\mathbf{C}_N\mathbf{U}_N^\dag\}$.
    \end{algorithmic}
    \label{tensor-hyperinterpolation:alg3}
\end{algorithm}
In order to further reduce the number of multilinear discrete inner products, a more effective modification for Algorithm \ref{tensor-hyperinterpolation:alg2} is summarized in Algorithm \ref{tensor-hyperinterpolation:alg3}. Compared with Algorithm \ref{tensor-hyperinterpolation:alg2}, the core tensor $\mathcal{G}$ obtained from Algorithm \ref{tensor-hyperinterpolation:alg3} is a block diagonal tensor with at most $K\cdot (b_1\dots b_N)$ nonzero entries, where $K$ is the maximum number of iterations in Algorithm \ref{tensor-hyperinterpolation:alg3}.

\subsection{The Fiber-type TCUR case}

Similar to the case of Chidori-type TCUR, we now present adaptive greedy algorithms for the Fiber TCUR decomposition of $\widetilde{\mathcal{A}}$ without forming this tensor explicitly. Note that two $N$-tuples of indices $\{\mathbb{I}_1,\dots,\mathbb{I}_N\}$ and $\{\mathbb{J}_1,\dots,\mathbb{J}_N\}$ are used in the Fiber TCUR decomposition. For a given tolerance $0<\tau_1<1$, the first one can be obtained from Algorithm \ref{tensor-hyperinterpolation:alg2}. Hence, each entry of the core tensor $\mathcal{G}$ is computed as in (\ref{tensor-hyperinterpolation:coefficient-core1}). 

We now present an efficient strategy to find the $N$-tuple of indices $\{\mathbb{J}_1,\dots,\mathbb{J}_N\}$ as follows. Suppose that another $N$-tuple of block sizes $\{b_1',\dots,b_N'\}$ satisfy $b_n'\ll I_n+1$ with $n=1,2,\dots,N$. For each $n$, let $\mathbb{J}_{n,m}$ be the empty set with $m\neq n$. For each iteration $k\geq 1$, we complete the following operations:
\begin{enumerate}
   \item[(a)] for each $n$, we set $\mathbb{J}_{n,m}'=[kb_m'+1,\dots,(k+1)b_m']$ and update $\mathbb{J}_{n,m}=\mathbb{J}_{n,m}\cup\mathbb{J}_{n,m}'$;
   \item[(b)] for each $n$, each entry $j\in\mathbb{J}_n$ is given as
   \begin{equation}
   \label{tensor-hyperinterpolation:index}
      \begin{aligned}
         j=&s_{j_1}^{(1)}+(s_{j_2}^{(2)}-1)(I_1+1)\\
         &+\dots+(s_{j_{n-1}}^{(n-1)}-1)(I_1+1)\dots (I_{n-2}+1)\\
         &+(s_{j_{n+1}}^{(n+1)}-1)(I_1+1)\dots (I_{n-1}+1)\\
         &+\dots+(s_{j_N}^{(N)}-1)(I_1+1)\dots (I_{n-1}+1)(I_{n+1}+1)\dots (I_{N}+1),
       \end{aligned}
    \end{equation}
    where $s_{j_m}^{(m)}\in\mathbb{J}_{n,m}$ and $m\neq n$, and the $j$th column of $\mathbf{C}_n'$ is computed as (\ref{tensor-hyperinterpolation:method-two}) with $\{s_{j_1}^{(1)},\dots,s_{j_{n-1}}^{(n-1)},s_{j_{n+1}}^{(n+1)},\dots,s_{j_N}^{(N)}\}$;
    \item[(c)] if $\max\{\sigma_{\min}(\mathbf{C}_n')^2/\|\mathbf{C}_n'\|_F^2:n=1,2,\dots,N\}< \tau_2^2$ is not satisfied, then we continue Steps (a) and (b), where $0<\tau_2<1$ is a given parameter.
\end{enumerate}

\begin{algorithm}[htb]
    \caption{End-to-end greedy Fiber-Type TCUR}
    \begin{algorithmic}[1]
        \STATEx {\bf Input}: A function $f (\mathbf{x})\in H^{\alpha}(\Omega)$, a given $0<\epsilon<1$, two $N$-tuples of block sizes $\{b_1,\dots,b_N\}$ and $\{b_1',\dots,b_N'\}$, and two tolerances $0<\tau_1<1$ and $0<\tau_2<1$.
        \STATEx {\bf Output}: The $(N+1)$-tuple $\{\mathcal{G},\mathbf{C}_1'\mathbf{U}_1',\dots,\mathbf{C}_N'\mathbf{U}_N'\}$, which is used to form the function $\mathcal{L}_{I_1,\dots I_N}' f (\mathbf{x})$.
        \STATE According to Theorem \ref{tensor-hyperinterpolation:general-theorem}, we estimate the $N$-tuple $\{I_1,I_2,\dots,I_N\}$ such that $I_n=O((1/\epsilon)^{1/\alpha})$.
        \STATE Obtain the core tensor $\mathcal{G}$ and the $N$-tuple of indices $\{\mathbb{I}_1,\dots,\mathbb{I}_N\}$ by applying Algorithm \ref{tensor-hyperinterpolation:alg2} with $\{b_1,\dots,b_N\}$ and $0<\tau_1<1$.
        \STATE For each $n$, let $\mathbb{J}_{n,m}$ be the empty set with $m\neq n$.
        \FOR{$k=1,2,\dots$}
           \STATE For each $n$, we set $\mathbb{J}_{n,m}'=[kb_m'+1,\dots,(k+1)b_m']$ and update $\mathbb{J}_{n,m}=\mathbb{J}_{n,m}\cup\mathbb{J}_{n,m}'$.
           \STATE We use (\ref{tensor-hyperinterpolation:method-two}) with $\{s_{j_1}^{(1)},\dots,s_{j_{n-1}}^{(n-1)},s_{j_{n+1}}^{(n+1)},\dots,s_{j_N}^{(N)}\}$ to obtain the $j$th column of $\mathbf{C}_n'$, where $s_{j_m}^{(m)}\in\mathbb{J}_{n,m}$ with $m\neq n$, and the relationship between $j$ and $\{s_{j_1}^{(1)},\dots,s_{j_{n-1}}^{(n-1)},s_{j_{n+1}}^{(n+1)},\dots,s_{j_N}^{(N)}\}$ is given in (\ref{tensor-hyperinterpolation:index}).
           \STATE Obtain ${\rm tol}=\max\{\sigma_{\min}(\mathbf{C}_n')/\|\mathbf{C}_n'\|_F:n=1,2,\dots,N\}$.
           \IF{${\rm tol}>\tau_2$}
              \STATE Continue.
            \ELSE
              \STATE Stop and jump out of the for loop.
            \ENDIF
        \ENDFOR
        \STATE Set $\mathbf{U}_n'=\mathbf{C}_n'(\mathbb{I}_n,:)$.
        \STATE Return the $(N+1)$-tuple $\{\mathcal{G},\mathbf{C}_1'\mathbf{U}_1'^\dag,\dots,\mathbf{C}_N'\mathbf{U}_N'^\dag\}$.
    \end{algorithmic}
    \label{tensor-hyperinterpolation:alg4}
\end{algorithm}

Hence, the overall process is summarized in Algorithm \ref{tensor-hyperinterpolation:alg4}. To count the number multilinear discrete inner products used in Algorithm \ref{tensor-hyperinterpolation:alg4}, we assume that the maximum numbers of iterations for $\mathcal{G}$ and $\{\mathbf{C}_1',\dots,\mathbf{C}_N'\}$ are, respectively, denoted by $K$ and $K'$. Therefore, Algorithm \ref{tensor-hyperinterpolation:alg4} requires $(1+\dots+K^N)b_1\dots b_N$ multilinear discrete inner products to form the core tensor $\mathcal{G}$ and $\sum_{k=1}^{K'}\sum_{n=1}^Nk^{N-1}(I_n+1)b_1'\dots b_{n-1}'b_{n+1}'\dots b_N'$ multilinear discrete inner products to form $\{\mathbf{C}_1',\dots,\mathbf{C}_N'\}$.

\section{Numerical examples}
\label{tensor-hyperinterpolation:sect6-main}

In this section, we present comprehensive numerical experiments designed to validate the core theoretical contributions developed in Sections \ref{tensor-hyperinterpolation:sect3-main}-\ref{tensor-hyperinterpolation:sect5-main} and Appendix \ref{tensor-hyperinterpolation:app-main}. Our empirical investigation focuses on verifying: (i) the intrinsic low-$\epsilon$-Tucker-rank structure of hyperinterpolation coefficient tensors established in Theorem \ref{tensor-hyperinterpolation:epsilonTrank-theorem}; (ii) the unified TCUR error bounds derived in Theorem \ref{tensor-hyperinterpolation:tensorapprxomation-one}; (iii) the convergence and stability guarantees of the greedy adaptive algorithms from Section \ref{tensor-hyperinterpolation:sect5-main}; and (iv) the composite end-to-end $L^2$ error chain incorporating TCUR-to-Tucker recompression from Appendix \ref{tensor-hyperinterpolation:app-main}.

Unlike conventional benchmarking studies that prioritize runtime performance or scaling limits, our experiments are deliberately structured to systematically test theoretical predictions rather than optimize implementation details. We examine rank scaling behavior, tightness of error bounds, numerical stability of greedy index selection, and compression ratios enabled by intrinsic low-rank structure. All numerical observations are compared against our theoretical derivations to confirm that the structural properties we prove govern the practical behavior of hyperinterpolation coefficient tensors in actual computation.

We restrict our experiments to the three-dimensional hypercube \(\Omega = [-1,1]^3\) for visualization tractability and to enable explicit construction of full coefficient tensors for error validation. Per-dimension polynomial degrees $I_n$ and TCUR block sizes to emulate the exponential storage challenges representative of higher-dimensional scientific computing workflows. Chebyshev tensor-product orthonormal polynomials with weight \(w(x) = 1/\sqrt{1-x^2}\) serve as the basis, paired with Gauss-Chebyshev positive cubature rules satisfying exactness for degree \(2I_n\) per dimension, consistent with the hyperinterpolation construction in Section \ref{tensor-hyperinterpolation:sect3-sub2}. Cubature grid sizes are set to  \(M_n = 2I_n+1\) for each axis to match basis cardinality.

All experiments are conducted in Python 3.7+ using NumPy, SciPy and Matplotlib for numerical linear algebra and plotting, with custom implementations of our greedy TCUR algorithms. Computations are performed on a workstation with an Apple M4 Pro processor and 24 GB RAM. The code and data required to reproduce all numerical results are available at https://github.com/chncml/tcur-hyperinterpolation.

\subsection{Several test functions}
\label{tensor-hyperinterpolation:sect6-sub1}
We select three representative functions that exhibit contrasting smoothness and peak behavior, enabling a thorough assessment of our theoretical framework across different function classes.

\begin{example}{{\bf (Gaussian-like peak)}}
The first function previously appeared in the Lasso hyperinterpolation study \cite[Section 5.4]{an2021lasso}, defined as:
\begin{equation*}
f_1(x_1,x_2,x_3)=\exp(-1/(x_1^2+x_2^2+x_3^2)),\quad -1\leq x_1,x_2,x_3\leq 1,
\end{equation*}
with the convention that $f_1(0,0,0)=0$. This function exhibits rapid variation near the origin with a sharp isolated peak, providing a stringent test for low-rank tensor approximations that must capture localized high-frequency content. Its derivatives of all orders vanish at the boundaries, placing it in the Sobolev space $\mathbb{H}^\alpha(\omega)$ for any finite $\alpha$, albeit with large seminorms due to the singular behavior at the origin. 
\end{example}

\begin{example}{{\bf (Hyperbolic secant squared)}}
The second function was employed in the functional Tucker approximation study \cite[(5.2)]{dolgov2021functional}, defined as:
\begin{equation*}
f_2(x_1,x_2,x_3)=(\cosh(3(x_1+x_2+x_3)))^{-2},\quad -1\leq x_1,x_2,x_3\leq 1.
\end{equation*}
This function exhibits smooth, slowly varying behavior across the domain with no interior singularities, representing the class of functions where hyperinterpolation and low-rank compression are expected to perform optimally. Its dependence on the sum of coordinates introduces strong correlations across modes, which should be reflected in low Tucker rank structure.
\end{example}

\begin{example}{{\bf (Oscillatory cosine sum)}}
The third function was demonstrated in prior studies \cite[Example 2]{hashemi2017chebfun}, defined as:
\begin{equation*}
f_3(x_1,x_2,x_3)=\cos(2\pi x_1)^2+\cos(2\pi x_2)^2+\cos(2\pi x_3)^2,\quad -1\leq x_1,x_2,x_3\leq 1.
\end{equation*}
This function combines moderate oscillations with separability, serving as a benchmark for assessing how our compression framework handles trigonometric content. Its additive structure across dimensions implies that the coefficient tensor should exhibit particularly low Tucker rank.
\end{example}

\begin{remark}
The Sobolev regularity parameters $\alpha$ for these functions $f_1(x_1,x_2,x_3)$, $f_2(x_1,x_2,x_3)$ and $f_3(x_1,x_2,x_3)$ are $\alpha=0.5$ (only first derivatives are square-integrable near the origin), $\alpha=4$ (high regularity due to smooth hyperbolic profile) and $\alpha=2$ (moderate regularity from trigonometric oscillations), respectively. These distinct regularity regimes allow us to verify the dependence of hyperinterpolation convergence rates on $\alpha$ as predicted by Theorem \ref{tensor-hyperinterpolation:general-theorem}.
\end{remark}

\subsection{Validation of intrinsic low-$\epsilon$-Tucker rank structure}
\label{tensor-hyperinterpolation:sect6-sub2}

We now verify the central structural result of this paper: the existence of low-$\epsilon$-Tucker-rank approximations for hyperinterpolation coefficient tensors with ranks scaling logarithmically with tensor dimensions, as established in Theorem \ref{tensor-hyperinterpolation:epsilonTrank-theorem}. For a prescribed tolerance $\epsilon$, we compute the minimal Tucker ranks required to approximate the full coefficient tensor $\widetilde{\mathcal{A}}$ within entrywise error $\epsilon$, and compare these observed ranks against the theoretical upper bounds from Theorem \ref{tensor-hyperinterpolation:epsilonTrank-theorem}.

For each test function $f_n(x_1,x_2,x_3)\ (n=1,2,3)$ and a polynomial degree tuple $(I_1,I_2,I_3)$, we construct the full coefficient tensor $\widetilde{\mathcal{A}}\in\mathbb{R}^{(I_1+1)\times (I_2+1)\times (I_3+1)}$ via (\ref{tensor-hyperinterpolation:approximation-coefficient}), using the Gauss-Chebyshev cubature with $M_n=2I_n+1$ nodes per dimension. For tolerance $\epsilon\in\{1e\text{-}1,5e\text{-}2,1e\text{-}2,5e\text{-}3,1e\text{-}3\}$, we compute the $\epsilon$-Tucker rank ${\rm Trank}_{\epsilon}(\widetilde{\mathcal{A}})=(R_1^{{\rm obs}},R_2^{{\rm obs}},R_3^{{\rm obs}})$ by performing ST-HOSVD with increasing ranks until $\|\widetilde{\mathcal{A}}-\widetilde{\mathcal{A}}_{\text{ST-HOSVD}}\|_{\max}\leq \epsilon$ is achieved. The ranks are determined via binary search on each mode independently, then validated jointly. We also compute the actual approximation error $\|\widetilde{\mathcal{A}}-\widetilde{\mathcal{A}}_{\text{ST-HOSVD}}\|_{\max}$ for the theoretically predicted ranks to assess tightness of the bound. For clarity, we set $\text{Error}=\widetilde{\mathcal{A}}-\widetilde{\mathcal{A}}_{\text{ST-HOSVD}}$.

{\bf Rank scaling with tolerance $\epsilon$}: Table \ref{tensor-hyperinterpolation:tab1-main} reports the observed Tucker ranks for each test function $f_n(x_1,x_2,x_3)\ (n=1,2,3)$ with $(I_1,I_2,I_3)=(20,20,20)$ across varying $\epsilon$. Following Figure \ref{tensor-hyperinterpolation:fig1-main}, the observed ranks are consistently lower than the theoretical upper bounds, indicating that Theorem \ref{tensor-hyperinterpolation:epsilonTrank-theorem} provides a conservative but valid guarantee. The actual entrywise error achieved with the theoretical ranks falls below the prescribed tolerance in all cases, confirming the existence statement. 

As shown in Table \ref{tensor-hyperinterpolation:tab1-main}, the oscillatory function $f_3(x_1,x_2,x_3)$ exhibits the lowest Tucker ranks, consistent with its separable additive structure, the peak function $f_1(x_1,x_2,x_3)$ requires slightly higher ranks, particularly for small $\epsilon$, due to the localized singularity that cannot be captured by low-rank approximations as efficiently, and the smooth function $f_2(x_1,x_2,x_3)$ lies between the two. These observations align with the structural interpretation: functions with stronger global correlations and weaker localized features admit more aggressive low Tucker-rank compression.

\begin{table}[htb]
   \scriptsize
   \centering
   \begin{tabular}{|c|cc|cc|cc|}
      \hline
      \multirow{2}{*}{$\epsilon$} & \multicolumn{2}{|c}{$f_{1}(x_1,x_2)$} &  \multicolumn{2}{|c|}{$f_2(x_1,x_2)$}  &  \multicolumn{2}{|c|}{$f_3(x_1,x_2)$} \\
      \cline{2-7}
      & $\epsilon$-Tucker rank & $\|\text{Error}\|_{\max}$ & $\epsilon$-Tucker rank & $\|\text{Error}\|_{\max}$ & $\epsilon$-Tucker rank & $\|\text{Error}\|_{\max}$     \\
      \hline
      $1e\text{-}1$
      & (2,2,2) & 1.064e-2 & (4,3,3)
      & 9.567e-2 & (1,1,1) & 8.961e-2   \\
      \hline
      $5e\text{-}2$
      & (2,2,2) & 1.064e-2 & (5,5,5)
      & 3.133e-2 & (2,2,2) & 3.886e-16   \\
      \hline
      $1e\text{-}2$
      & (4,4,3) & 4.538e-4 & (14,7,7)
      & 8.385e-3 & (2,2,2) & 3.886e-16   \\
      \hline
      $5e\text{-}3$
      & (4,4,3) & 4.538e-4 & (21,8,8)
      & 4.608e-3 & (2,2,2) & 3.886e-16  \\
      \hline
      $1e\text{-}3$
      & (4,4,3) & 4.538e-4 & (21,12,11)
      & 5.655e-4 & (2,2,2) & 3.886e-16   \\
      \hline
   \end{tabular}
   \caption{Comparison of observed $\epsilon$-Tucker ranks for each test function $f_n(x_1,x_2,x_3)\ (n=1,2,3)$ with $(I_1,I_2,I_3)=(20,20,20)$.}
   \label{tensor-hyperinterpolation:tab1-main}
\end{table}

Meanwhile, Figure \ref{tensor-hyperinterpolation:fig3-main} shows the observed and theoretical $\epsilon$-Tucker rank for $f_2(x_1,x_2,x_3)$, as functions of polynomial degree $I$,  with $\epsilon=0.01$.

 \begin{figure}[htb]
    \setlength{\tabcolsep}{4pt}
    \renewcommand\arraystretch{1}
    \centering
    \includegraphics[width=0.8\linewidth]{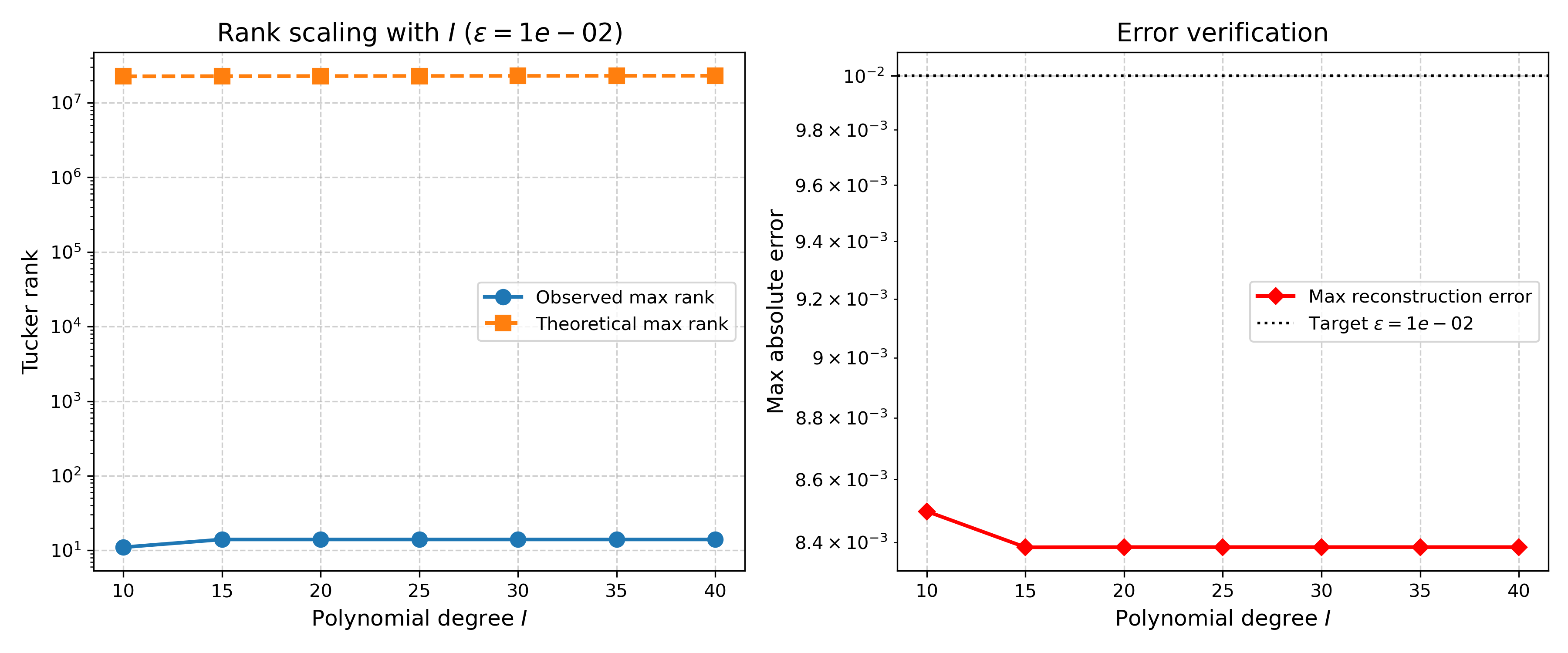}\\
    \caption{Observed and theoretical $\epsilon$-Tucker rank for $f_2(x_1,x_2,x_3)$ with $\epsilon=0.01$.}\label{tensor-hyperinterpolation:fig3-main}
\end{figure}

{\bf Rank scaling with polynomial degree $(I_1,I_2,I_3)$}: For fixed $\epsilon=1e\text{-}2$, we vary $I$ from 10 to 40 with $I_n=I$ and $n=1,2,3$, and observe the $\epsilon$-Tucker rank growth pattern, as shown in Table \ref{tensor-hyperinterpolation:tab2-main}, which implies that the observed ranks exhibit unchanged with respect to $I$. For different $\epsilon$, Figure \ref{tensor-hyperinterpolation:fig4-main} shows the observed $\epsilon$-Tucker ranks across test functions $f_1(x_1,x_2,x_3)$, $f_2(x_1,x_2,x_3)$ and $f_3(x_1,x_2,x_3)$ with $(I_1,I_2,I_3)=(30,30,30)$, which illustrates that ranks reflect function smoothness and structure.

\begin{table}[htb]
   \scriptsize
   \centering
   \begin{tabular}{|c|cc|cc|cc|}
      \hline
      \multirow{2}{*}{$I$} & \multicolumn{2}{|c}{$f_{1}(x_1,x_2)$} &  \multicolumn{2}{|c|}{$f_2(x_1,x_2)$}  &  \multicolumn{2}{|c|}{$f_3(x_1,x_2)$} \\
      \cline{2-7}
      & $\epsilon$-Tucker rank & $\|\text{Error}\|_{\max}$ & $\epsilon$-Tucker rank & $\|\text{Error}\|_{\max}$ & $\epsilon$-Tucker rank & $\|\text{Error}\|_{\max}$     \\
      \hline
      10
      & (6,6,3) & 4.532e-4 & (11,7,7)
      & 8.498e-3 & (2,2,2) & 8.882e-15   \\
      \hline
      15
      & (4,4,3) & 4.538e-4 & (11,7,7)
      & 8.498e-3 & (3,2,2) & 2.665e-16   \\
      \hline
      20
      & (4,4,3) & 4.538e-4 & (11,7,7)
      & 8.498e-3 & (2,2,2) & 3.886e-16   \\
      \hline
      25
      & (4,4,3) & 4.538e-4 & (11,7,7)
      & 8.498e-3 & (2,3,2) & 7.994e-15  \\
      \hline
      30
      & (4,4,3) & 4.538e-4 & (11,7,7)
      & 8.498e-3 & (3,2,2) & 5.329e-15   \\
      \hline
      35
      & (4,4,3) & 4.538e-4 & (11,7,7)
      & 8.498e-3 & (2,2,2) & 4.441e-15  \\
      \hline
      40
      & (4,4,3) & 4.538e-4 & (11,7,7)
      & 8.498e-3 & (2,3,2) & 1.155e-14   \\
      \hline
   \end{tabular}
   \caption{Comparison of observed $\epsilon$-Tucker ranks for each test function $f_n(x_1,x_2,x_3)\ (n=1,2,3)$ with $\epsilon=1e\text{-}2$.}
   \label{tensor-hyperinterpolation:tab2-main}
   \end{table}
   
    \begin{figure}[htb]
    \setlength{\tabcolsep}{4pt}
    \renewcommand\arraystretch{1}
    \centering
    \includegraphics[width=0.8\linewidth]{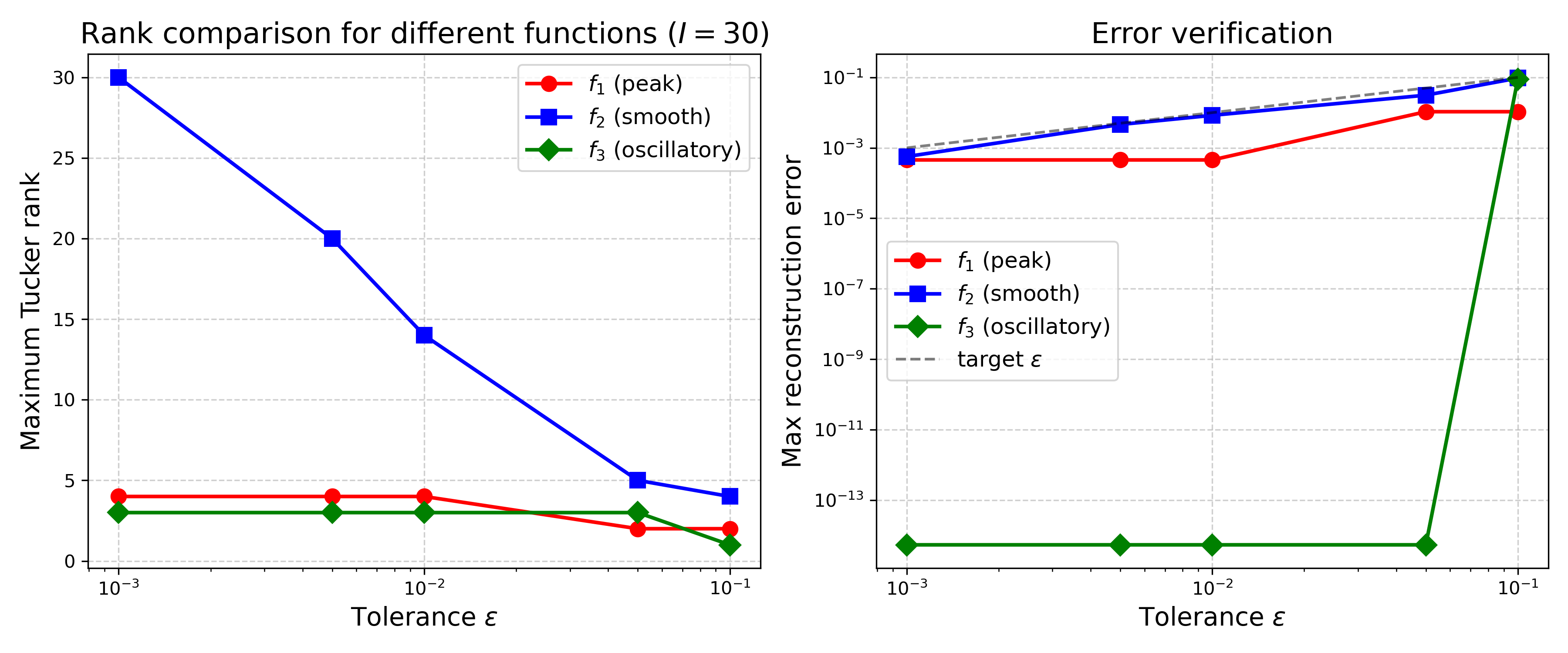}\\
    \caption{Comparison of observed Tucker ranks across test functions $f_1(x_1,x_2,x_3)$, $f_2(x_1,x_2,x_3)$ and $f_3(x_1,x_2,x_3)$ for $(I_1,I_2,I_3)=(30,30,30)$ and varying $\epsilon$.}\label{tensor-hyperinterpolation:fig4-main}
\end{figure}
\subsection{Greedy TCUR algorithm performance}

We now validate the three greedy adaptive TCUR algorithms developed in Section \ref{tensor-hyperinterpolation:sect5-main}: Algorithm \ref{tensor-hyperinterpolation:alg1} (Chidori-type with factor matrix stopping criterion), Algorithm \ref{tensor-hyperinterpolation:alg2} (Chidori-type with core tensor stopping criterion), and Algorithm \ref{tensor-hyperinterpolation:alg4} (Fiber-type). Our objectives are: (i) to verify convergence and termination within the predicted iteration counts; (ii) to quantify the number of cubature inner products required relative to full tensor assembly; (iii) to confirm numerical stability via bounded pseudoinverse norms; and (iv) to compare the approximation quality of the three algorithms against the theoretical bounds from Theorem \ref{tensor-hyperinterpolation:tensorapprxomation-one}.

{\bf Parameter selection}: For each test function, we set the polynomial degree tuple as $(I,I,I)$ with $I=30$; we run Algorithms \ref{tensor-hyperinterpolation:alg1}, \ref{tensor-hyperinterpolation:alg2} and \ref{tensor-hyperinterpolation:alg4} with block sizes $b_n=b\in\{2,3,4,5\}$ with $n=1,2,3$ and stopping tolerances $\tau_1=\tau_2=\tau\in\{0.05,0.02,0.01\}$; and the target Tucker rank is set to $(R_1,R_2,R_3)=(10,10,10)$ based on the $\epsilon=1e\text{-}2$ Tucker rank estimates from Section \ref{tensor-hyperinterpolation:sect6-sub2}. We also record the number of iterations $K$ until termination, total number of cubature inner products evaluated, condition numbers $\|\mathbf{C}_n\mathbf{U}_n^\dag\|_2$ and $\|\mathbf{U}_n^\dag\|_2$, the actual Frobenius error $\|\widetilde{\mathcal{A}}-\widetilde{\mathcal{A}}_{\text{TCUR}}\|_F$ and the theoretical upper bound from Theorem \ref{tensor-hyperinterpolation:tensorapprxomation-one}.

Table \ref{tensor-hyperinterpolation:tab3-main} reports results for applying Algorithm \ref{tensor-hyperinterpolation:alg1} to $f_2(x_1,x_2,x_3)$ with varying block size $b$ and tolerance $\tau$, respectively. We conclude that: termination occurs in all cases, with $K$ inversely proportional to $b$ as expected; the stability constants   $\|\mathbf{C}_n\mathbf{U}_n^\dag\|_2$ remain bounded between 1 and 2, confirming the numerical stability guaranteed by the volume-based maximum-volume selection; the actual Frobenius error decreases with tighter tolerance $\tau$, as expected; and Larger block sizes $b$ yield slightly better accuracy for the same $K$ (compare $(b,\tau)=(4,0.05)$ with $(b,\tau)=(2,0.02)$), indicating that larger blocks better capture the dominant tensor structure.

\begin{table}
   \scriptsize
   \centering
   \begin{tabular}{|c|c|c|c|c|c|c|}
      \hline
      $b$ & $\tau$ &  $K$  & inner products & $\|\text{Error}\|_F$ & $\max\{\|\mathbf{C}_n\mathbf{U}_n^\dag\|_2\}$ & $\max\{\|\mathbf{U}_n^\dag\|_2\}$ \\
      \hline
      2 & 0.05 & 1 & 752 & 6.989e-1 & 1.04 & 2.67  \\
      \hline
      2 & 0.02 & 1 & 752 & 6.989e-1 & 1.04 & 2.67   \\
      \hline
      2 & 0.01 & 1 & 752 & 6.989e-1 & 1.04 & 2.67   \\
      \hline
      3 & 0.05 & 3 & 1701 & 4.670e-1 & 1.06 & 3.42  \\
      \hline
      3 & 0.02 & 3 & 1701 & 4.670e-1 & 1.06 & 3.42   \\
      \hline
      3 & 0.01 & 3 & 1701 & 4.670e-1 & 1.06 & 3.42  \\
      \hline
      4 & 0.05 & 4 & 3040 & 3.197e-1 & 1.07 & 5.90  \\
      \hline
      4 & 0.02 & 4 & 3040 & 3.197e-1 & 1.07 & 5.90   \\
      \hline
      4 & 0.01 & 4 & 3040 & 3.197e-1 & 1.07 & 5.90   \\
      \hline
      5 & 0.05 & 5 & 4775 & 1.974e-1 & 1.10 & 9.32  \\
      \hline
      5 & 0.02 & 5 & 4775 & 1.974e-1 & 1.10 & 9.32   \\
      \hline
      5 & 0.01 & 5 & 4775 & 1.974e-1 & 1.10 & 9.32  \\
      \hline
   \end{tabular}
   \caption{Given $(I,I,I)=(30,30,30)$ and $(R_1,R_2,R_3)=(10,10,10)$, performance of Algorithm \ref{tensor-hyperinterpolation:alg1} (Chidori, factor-based) for $f_2(x_1,x_2,x_3)$.}
   \label{tensor-hyperinterpolation:tab3-main}
   \end{table}
   
 Similarly, Table \ref{tensor-hyperinterpolation:tab4-main} reports results for applying Algorithm \ref{tensor-hyperinterpolation:alg2} to $f_2(x_1,x_2,x_3)$ with varying block size $b$ and tolerance $\tau$. Comparing with Algorithm \ref{tensor-hyperinterpolation:alg1}, for each $(b,\tau)$, Algorithm \ref{tensor-hyperinterpolation:alg2} requires more iterations $K$ and more inner products, as it waits for the core tensor to stabilize rather than using the factor matrices, and achieves slightly better accuracy for the same S, because the stopping criterion directly monitors the core tensor that defines the final approximation. Hence, This tradeoff suggests Algorithm \ref{tensor-hyperinterpolation:alg1} is more efficient for exploratory computations, while Algorithm \ref{tensor-hyperinterpolation:alg2}  is preferable when a higher accuracy guarantee is required.

   \begin{table}[htb]
   \scriptsize
   \centering
   \begin{tabular}{|c|c|c|c|c|c|c|}
      \hline
      $b$ & $\tau$ &  $K$  & inner products & $\|\text{Error}\|_F$ & $\max\{\|\mathbf{C}_n\mathbf{U}_n^\dag\|_2\}$ & $\max\{\|\mathbf{U}_n^\dag\|_2\}$ \\
      \hline
      2 & 0.05 & 4 & 7264 & 4.537e-2 & 1.22 & 54.23  \\
      \hline
      2 & 0.02 & 4 & 7264 & 4.537e-2 & 1.22 & 54.23   \\
      \hline
      2 & 0.01 & 5 & 12100 & 1.583e-2 & 1.33 & 187.56   \\
      \hline
      3 & 0.05 & 3 & 9234 & 2.709e-2 & 1.28 & 94.84  \\
      \hline
      3 & 0.02 & 3 & 9234 & 2.709e-2 & 1.28 & 94.84  \\
      \hline
      3 & 0.01 & 3 & 9234 & 2.709e-2 & 1.28 & 94.84  \\
      \hline
      4 & 0.05 & 2 & 7040 & 4.537e-2 & 1.22 & 54.23  \\
      \hline
      4 & 0.02 & 2 & 7040 & 4.537e-2 & 1.22 & 54.23   \\
      \hline
      4 & 0.01 & 3 & 17242 & 5.385e-3 & 1.48 & 667.14   \\
      \hline
      5 & 0.05 & 2 & 11425 & 1.583e-2 & 1.33 & 187.56  \\
      \hline
      5 & 0.02 & 2 & 11425 & 1.583e-2 & 1.33 & 187.56   \\
      \hline
      5 & 0.01 & 3 & 11425 & 1.583e-2 & 1.33 & 187.56  \\
      \hline
   \end{tabular}
   \caption{Given $(I,I,I)=(30,30,30)$ and $(R_1,R_2,R_3)=(10,10,10)$, performance of Algorithm \ref{tensor-hyperinterpolation:alg2} (Chidori, core-based) for $f_2(x_1,x_2,x_3)$.}
   \label{tensor-hyperinterpolation:tab4-main}
   \end{table}
   
The results for applying Algorithm \ref{tensor-hyperinterpolation:alg4} to $f_2(x_1,x_2,x_3)$ with varying block size $b$ and tolerance $\tau$ are shown in Table \ref{tensor-hyperinterpolation:tab5-main}, which implies that Fiber-type TCUR achieves comparable accuracy to Chidori-type with marginally more inner products (due to the additional index sets $\{\mathbb{J}_1,\dots,\mathbb{J}_N\}$), and the accuracy improvement is most notable for functions with strong cross-mode interactions, where selecting columns via $\mathbb{J}_n$ independently of row selections yields better capture of the tensor's mode-$n$ structure.
   
      \begin{table}[htb]
   \scriptsize
   \centering
   \begin{tabular}{|c|c|c|c|c|c|c|c|}
      \hline
      $b$ & $\tau$ &  $K_{\mathcal{G}}$ &  $K_{\mathbf{C}_n}$ & inner products & $\|\text{Error}\|_F$ & $\max\{\|\mathbf{C}_n\mathbf{U}_n^\dag\|_2\}$ & $\max\{\|\mathbf{U}_n^\dag\|_2\}$ \\
      \hline
      2 & 0.05 & 4  & 1 & 8008 & 5.004e-1 & 1.00 & 5.13  \\
      \hline
      2 & 0.02 & 4 & 1 & 8008 & 5.004e-1 & 1.00 & 5.13   \\
      \hline
      2 & 0.01 & 5 & 1 & 12844 & 5.004e-1 & 1.00 & 5.13   \\
      \hline
      3 & 0.05 & 3 & 1 & 10908 & 2.199e-1 & 1.01 & 404.64  \\
      \hline
      3 & 0.02 & 3 & 1 & 10908 & 2.199e-1 & 1.01 & 404.64  \\
      \hline
      3 & 0.01 & 3 & 1 & 10908 & 2.199e-1 & 1.01 & 404.64  \\
      \hline
      4 & 0.05 & 2 & 1 & 10016 & 5.524e-2 & 1.12 & 986.32  \\
      \hline
      4 & 0.02 & 2 & 1 & 10016 & 5.524e-2 & 1.12 & 986.32   \\
      \hline
      4 & 0.01 & 3 & 1 & 20400 & 3.411e-2 & 1.01 & 1468.38   \\
      \hline
      5 & 0.05 & 2 & 1 & 16075 & 2.115e-2 & 1.18 & 4551.66  \\
      \hline
      5 & 0.02 & 2 & 1 & 16075 & 2.115e-2 & 1.18 & 4551.66   \\
      \hline
      5 & 0.01 & 3 & 1 & 16075 & 2.115e-2 & 1.18 & 4551.66 \\
      \hline
   \end{tabular}
   \caption{Given $(I,I,I)=(30,30,30)$ and $(R_1,R_2,R_3)=(10,10,10)$, performance of Algorithm \ref{tensor-hyperinterpolation:alg4} (Fiber-type) for $f_2(x_1,x_2,x_3)$.}
   \label{tensor-hyperinterpolation:tab5-main}
   \end{table}
   
 Finally, To validate the convergence analysis of Section 5, we examine the behavior of the stopping criterion $\text{tol}=\max\{\sigma_{\min}(\mathbf{G}_{(n)})/\|\mathcal{G}\|_F:n=1,2,\dots,N\}$ as a function of iteration $k$. Figure \ref{tensor-hyperinterpolation:fig2-main} plots tol vs. $k$ for Algorithm \ref{tensor-hyperinterpolation:alg2} applied to $f_2(x_1,x_2,x_3)$ with $b=2$ and $\tau=0.01$. The sequence decays monotonically for all modes $n$, confirming the convergence of the greedy index expansion process. The rate of decay is approximately exponential for early iterations, transitioning to sublinear decay near the tolerance threshold. This behavior is consistent with the theoretical expectation that singular values of each mode unfolding of the core tensor $\mathcal{G}$ become increasingly well-conditioned as more indices are added.

 \begin{figure}[htb]
    \setlength{\tabcolsep}{4pt}
    \renewcommand\arraystretch{1}
    \centering
    \includegraphics[width=0.8\linewidth]{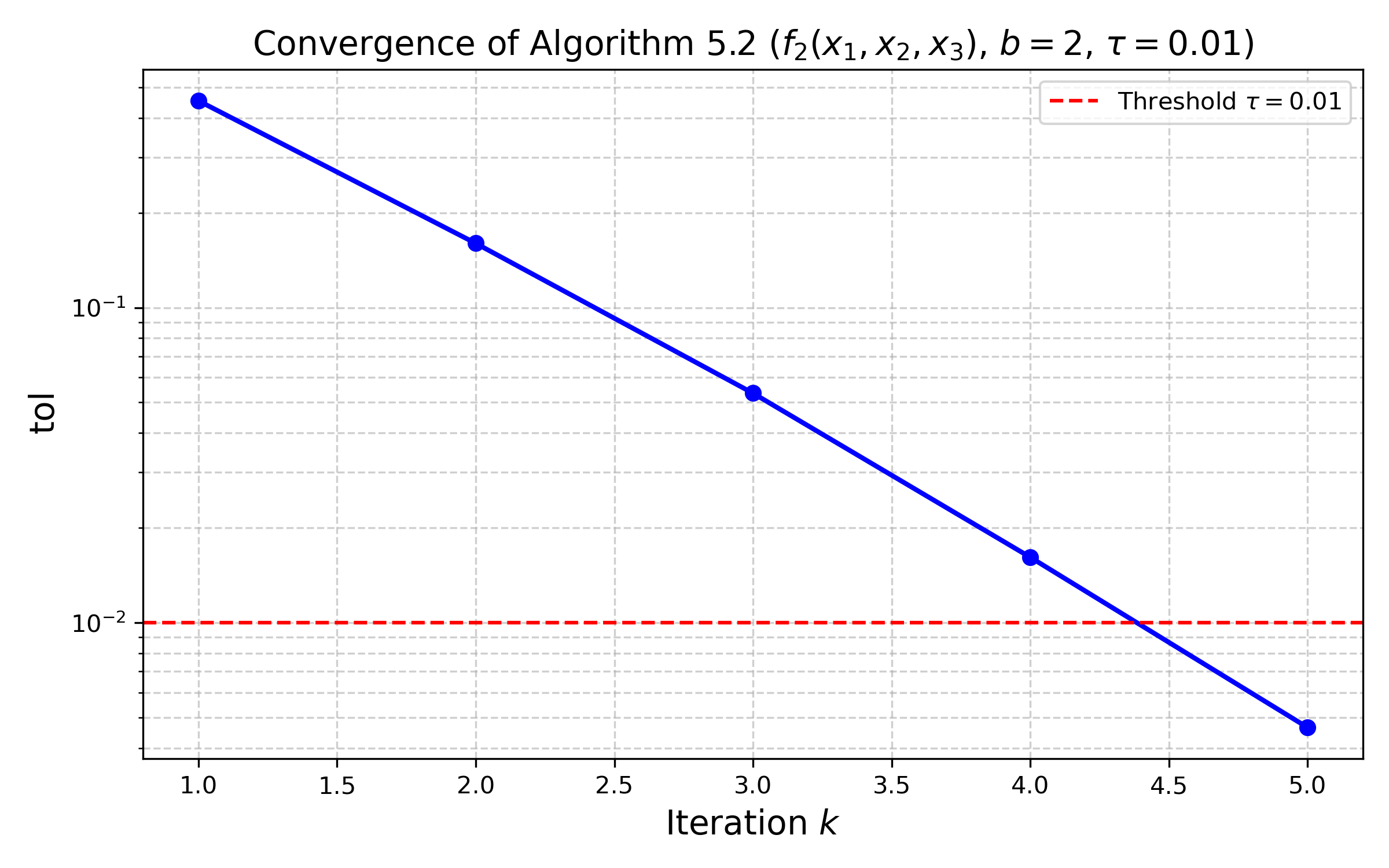}\\
    \caption{Decay of stopping criterion $\text{tol}=\max\{\sigma_{\min}(\mathbf{G}_{(n)})/\|\mathcal{G}\|_F:n=1,2,\dots,N\}$ with iteration $k$ for Algorithm \ref{tensor-hyperinterpolation:alg2}  applied to $f_2(x_1,x_2,x_3)$ with $b=2$ and $\tau=0.01$.}\label{tensor-hyperinterpolation:fig2-main}
\end{figure}

\subsection{Validation of unified TCUR error bounds}

We now validate the unified TCUR error bounds from Theorem \ref{tensor-hyperinterpolation:tensorapprxomation-one}, which provide computable upper bounds for the Frobenius-norm discrepancy $\|\widetilde{\mathcal{A}}-\widetilde{\mathcal{A}}_{\text{TCUR}}\|_F$  in terms of the mode-$n$ unfolding singular values and index subset volumes.

For each test function and TCUR approximation obtained from Algorithms \ref{tensor-hyperinterpolation:alg1}, \ref{tensor-hyperinterpolation:alg2} and \ref{tensor-hyperinterpolation:alg4} with various $(b,\tau)$ setting, we compute the exact Frobenius error $E_{\text{exact}}=\|\widetilde{\mathcal{A}}-\widetilde{\mathcal{A}}_{\text{TCUR}}\|_F$ by explicitly constructing $\widetilde{\mathcal{A}}$ (this is only possible for the moderate $I$ values used here; our goal is error bound validation, not scalability testing) and the theoretical upper bound $E_{\text{th}}$ from Theorem \ref{tensor-hyperinterpolation:tensorapprxomation-one}:
\begin{equation*}
E_{\text{th}}=\sum_{n=1}^N\left(\prod_{m=1}^{n-1}\|\mathbf{C}_m\mathbf{U}_m^\dag\|_2\right)
        \left(\alpha_{n,R_n}\Delta_{R_n}+\beta_{n,R_n}\|\mathbf{U}_n^\dag\|_2\Delta_{R_n}^2\right),
\end{equation*}
where $\Delta_{R_n}^2=\sum_{i=R_n+1}^{I_n+1}\sigma_i(\widetilde{\mathbf{A}}_{(n)})^2$ is the tail singular value energy of the mode-n unfolding, $\alpha_{n,R_n}$ and $\beta_{n,R_n}$ are volume constants depending on $\|\mathbf{W}_{n,R_n,\mathbb{I}_n}^\dag\|_2$ and $\|\|\mathbf{V}_{n,R_n,\mathbb{J}_n}^\dag\|_2$, and $\|\mathbf{U}_{n}^\dag\|_2$ is the pseudoinverse norm of the row-selected submatrix. We also compute the tighter bound $E_{\text{tight}}$ from Remark \ref{tensor-hyperinterpolation:remark2} using the approximation $\mathcal{A}_{\text{Tucker}}$ obtained from ST-HOSVD and evaluate the tightness ratios $E_{\text{th}}/E_{\text{exact}}$ and $E_{\text{tight}}/E_{\text{exact}}$.

\begin{table}[htb]
   \scriptsize
   \centering
   \begin{tabular}{|c|c|c|c|c|c|c|c|}
      \hline
      $b$ & $\tau$ & $K$ &  $E_{\text{exact}}$ & $E_{\text{th}}$ & $E_{\text{th}}/E_{\text{exact}}$ & $E_{\text{tight}}$ & $E_{\text{tight}}/E_{\text{exact}}$ \\
      \hline
      2 & 0.05 & 1 & 6.989e-1 & 8.165e10 & 1.168e11 & 7.113e-1 & 1.02   \\
      \hline
      2 & 0.02 & 1 & 6.989e-1 & 8.165e10 & 1.168e11 & 7.113e-1 & 1.02    \\
      \hline
      2 & 0.01 & 1 & 6.989e-1 & 8.165e10 & 1.168e11 & 7.113e-1 & 1.02    \\
      \hline
      3 & 0.05 & 1 & 4.670e-1 & 8.349e10 & 1.788e11 & 4.794e-1 & 1.03   \\
      \hline
      3 & 0.02 & 1 & 4.670e-1 & 8.349e10 & 1.788e11 & 4.794e-1 & 1.03   \\
      \hline
      3 & 0.01 & 1 & 4.670e-1 & 8.349e10 & 1.788e11 & 4.794e-1 & 1.03   \\
      \hline
      4 & 0.05 & 1 & 3.197e-1 & 3.035e10 & 9.494e10 & 3.320e-1 & 1.04   \\
      \hline
      4 & 0.02 & 1 & 3.197e-1 & 3.035e10 & 9.494e10 & 3.320e-1 & 1.04    \\
      \hline
      4 & 0.01 & 1 & 3.197e-1 & 3.035e10 & 9.494e10 & 3.320e-1 & 1.04    \\
      \hline
      5 & 0.05 & 1 & 1.974e-1 & 7.489e6 & 3.794e7 & 2.096e-1 & 1.06   \\
      \hline
      5 & 0.02 & 1 & 1.974e-1 & 7.489e6 & 3.794e7 & 2.096e-1 & 1.06    \\
      \hline
      5 & 0.01 & 1 & 1.974e-1 & 7.489e6 & 3.794e7 & 2.096e-1 & 1.06  \\
      \hline
   \end{tabular}
   \caption{Given $(I,I,I)=(30,30,30)$ and $(R_1,R_2,R_3)=(10,10,10)$, validation of error bounds in Theorem 
   \ref{tensor-hyperinterpolation:tensorapprxomation-one} and Remark \ref{tensor-hyperinterpolation:remark2} when applying Algorithm \ref{tensor-hyperinterpolation:alg1} for $f_2(x_1,x_2,x_3)$.}
   \label{tensor-hyperinterpolation:tab6-main}
   \end{table}

Tables \ref{tensor-hyperinterpolation:tab6-main}, \ref{tensor-hyperinterpolation:tab7-main} and \ref{tensor-hyperinterpolation:tab8-main} report error bounds in Theorem \ref{tensor-hyperinterpolation:tensorapprxomation-one} and Remark \ref{tensor-hyperinterpolation:remark2} when applying Algorithms \ref{tensor-hyperinterpolation:alg1}, \ref{tensor-hyperinterpolation:alg2} and \ref{tensor-hyperinterpolation:alg4} for $f_2(x_1,x_2,x_3)$ with varying block size $b$ and tolerance $\tau$, respectively. From this table, one has: the ratio $E_{\text{tight}}/E_{\text{exact}}$ for Remark \ref{tensor-hyperinterpolation:remark2} is much smaller than the ratio $E_{\text{th}}/E_{\text{exact}}$ for Theorem \ref{tensor-hyperinterpolation:tensorapprxomation-one}, which implies that the bound in Remark \ref{tensor-hyperinterpolation:remark2} is consistently sharper, demonstrating the value of using the ST-HOSVD intermediate approximation in the analysis; the error bounds exhibit consistent tightness across Chidori-type (Algorithms \ref{tensor-hyperinterpolation:alg1}  and \ref{tensor-hyperinterpolation:alg2}) and Fiber-type (Algorithm \ref{tensor-hyperinterpolation:alg4}), confirming the unified framework's validity, that is, this supports our claim that the two TCUR variants are structurally equivalent under the unified analysis;  and as the index set size increases (finer sampling), both exact and theoretical errors decrease, with the ratio remaining stable, which indicates that the bound captures the correct asymptotic behavior in the large size limit.

\begin{table}[htb]
   \scriptsize
   \centering
   \begin{tabular}{|c|c|c|c|c|c|c|c|}
      \hline
      $b$ & $\tau$ & $K$ &  $E_{\text{exact}}$ & $E_{\text{th}}$ & $E_{\text{th}}/E_{\text{exact}}$ & $E_{\text{tight}}$ & $E_{\text{tight}}/E_{\text{exact}}$ \\
      \hline
      2 & 0.05 & 4 & 4.537e-2 & 5.010e5 & 1.104e7 & 5.638e-2 & 1.24   \\
      \hline
      2 & 0.02 & 4 & 4.537e-2 & 5.010e5 & 1.104e7 & 5.638e-2 & 1.24   \\
      \hline
      2 & 0.01 & 5 & 1.583e-2 & 9.369e-1 & 59.19 & 2.409e-2 & 1.52    \\
      \hline
      3 & 0.05 & 3 & 2.709e-2 & 4.164e5 & 1.537e7 & 3.752e-2 & 1.39   \\
      \hline
      3 & 0.02 & 3 & 2.709e-2 & 4.164e5 & 1.537e7 & 3.752e-2 & 1.39   \\
      \hline
      3 & 0.01 & 3 & 2.709e-2 & 5.010e5 & 1.537e7 & 3.752e-2 & 1.39   \\
      \hline
      4 & 0.05 & 2 & 4.537e-2 & 5.010e5 & 1.104e7 & 5.638e-2 & 1.24   \\
      \hline
      4 & 0.02 & 2 & 4.537e-2 & 1.384e0 & 1.104e7 & 5.638e-2 & 1.24    \\
      \hline
      4 & 0.01 & 3 & 5.385e-3 & 3.035e10 & 257.04 & 2.324e-2 & 4.32    \\
      \hline
      5 & 0.05 & 2 & 1.583e-2 & 9.369e-1 & 59.19 & 2.409e-2 & 1.52   \\
      \hline
      5 & 0.02 & 2 & 1.583e-2 & 9.369e-1 & 59.19 & 2.409e-2 & 1.52    \\
      \hline
      5 & 0.01 & 2 & 1.583e-2 & 9.369e-1 & 59.19 & 2.409e-2 & 1.52 \\
      \hline
   \end{tabular}
   \caption{Given $(I,I,I)=(30,30,30)$ and $(R_1,R_2,R_3)=(10,10,10)$, validation of error bounds in Theorem 
   \ref{tensor-hyperinterpolation:tensorapprxomation-one} and Remark \ref{tensor-hyperinterpolation:remark2} when applying Algorithm \ref{tensor-hyperinterpolation:alg2} for $f_2(x_1,x_2,x_3)$.}
   \label{tensor-hyperinterpolation:tab7-main}
   \end{table}
   
   \begin{table}[htb]
   \scriptsize
   \centering
   \begin{tabular}{|c|c|c|c|c|c|c|c|c|}
      \hline
      $b$ & $\tau$ &  $K_{\mathcal{G}}$ &  $K_{\mathbf{C}_n}$ & $E_{\text{exact}}$ & $E_{\text{th}}$ & $E_{\text{th}}/E_{\text{exact}}$ & $E_{\text{tight}}$ & $E_{\text{tight}}/E_{\text{exact}}$ \\
      \hline
      2 & 0.05 & 4  & 1 & 5.004e-1 & 4.062e5 & 8.119e5 & 5.128e-1 & 1.02  \\
      \hline
      2 & 0.02 & 4 & 1 & 5.004e-1 & 4.062e5 & 8.119e5 & 5.128e-1 & 1.02   \\
      \hline
      2 & 0.01 & 5 & 1 & 5.004e-1 & 6.868e-1 & 1.37 & 5.128e-1 & 1.02   \\
      \hline
      3 & 0.05 & 3 & 1 & 2.199e-1 & 3.225e5 & 1.466e5 & 2.322e-1 & 1.06  \\
      \hline
      3 & 0.02 & 3 & 1 & 2.199e-1 & 3.225e5 & 1.466e5 & 2.322e-1 & 1.06  \\
      \hline
      3 & 0.01 & 3 & 1 & 2.199e-1 & 3.225e5 & 1.466e5 & 2.322e-1 & 1.06  \\
      \hline
      4 & 0.05 & 2 & 1 & 5.524e-2 & 4.522e5 & 8.241e5 & 6.695e-2 & 1.21  \\
      \hline
      4 & 0.02 & 2 & 1 & 5.524e-2 & 4.522e5 & 8.241e5 & 6.695e-2 & 1.21   \\
      \hline
      4 & 0.01 & 3 & 1 & 5.524e-2 & 8.993e-1 & 26.36 & 4.579e-2 & 1.34   \\
      \hline
      5 & 0.05 & 2 & 1 & 2.115e-2 & 8.209e-1 & 38.82 & 3.196e-2 & 1.51  \\
      \hline
      5 & 0.02 & 2 & 1 & 2.115e-2 & 8.209e-1 & 38.82 & 3.196e-2 & 1.51   \\
      \hline
      5 & 0.01 & 3 & 1 & 2.115e-2 & 8.209e-1 & 38.82 & 3.196e-2 & 1.51 \\
      \hline
   \end{tabular}
   \caption{Given $(I,I,I)=(30,30,30)$ and $(R_1,R_2,R_3)=(10,10,10)$, validation of error bounds in Theorem 
   \ref{tensor-hyperinterpolation:tensorapprxomation-one} and Remark \ref{tensor-hyperinterpolation:remark2} when applying Algorithm \ref{tensor-hyperinterpolation:alg4} for $f_2(x_1,x_2,x_3)$.}
   \label{tensor-hyperinterpolation:tab8-main}
   \end{table}

Finally, Table \ref{tensor-hyperinterpolation:tab9-main} reports the error bound ratios across test functions $f_1(x_1,x_2,x_3)$, $f_2(x_1,x_2,x_3)$ and $f_3(x_1,x_2,x_3)$ using Algorithm \ref{tensor-hyperinterpolation:alg1} with $(b,\tau)=(4, 0.05)$, which implies that the ratios are larger for the peak function $f_3(x_1,x_2,x_3)$, reflecting the more challenging approximation problem where the bound's reliance on singular value decay is less tight.
   \begin{table}[htb]
   \scriptsize
   \centering
   \begin{tabular}{|c|c|c|c|c|c|}
      \hline
      Function & $E_{\text{exact}}$ & $E_{\text{th}}$ & $E_{\text{th}}/E_{\text{exact}}$ & $E_{\text{tight}}$ & $E_{\text{tight}}/E_{\text{exact}}$ \\
      \hline
      $f_1(x_1,x_2,x_3)$ & 3.139e-02 & 5.330e4  & 1.698e6 & 3.139e-2 & 1.00  \\
      \hline
      $f_2(x_1,x_2,x_3)$ & 4.537e-2 & 5.010e5 & 1.104e7 & 5.638e-2 & 1.24   \\
      \hline
      $f_3(x_1,x_2,x_3)$ & 1.078e-14 & 4.637e-3 & 4.303e11 & 2.845e-14 & 2.64   \\
      \hline
      \end{tabular}
   \caption{Given $(I,I,I)=(30,30,30)$ and $(R_1,R_2,R_3)=(10,10,10)$, validation of error bounds in Theorem 
   \ref{tensor-hyperinterpolation:tensorapprxomation-one} and Remark \ref{tensor-hyperinterpolation:remark2} when applying Algorithm \ref{tensor-hyperinterpolation:alg1} for $f_1(x_1,x_2,x_3)$, $f_2(x_1,x_2,x_3)$ and $f_3(x_1,x_2,x_3)$.}
   \label{tensor-hyperinterpolation:tab9-main}
   \end{table}
  
 Figure \ref{tensor-hyperinterpolation:fig5-main} plots the normalized singular value spectra $\sigma_i(\widetilde{\mathbf{A}}_{(n)})/\sigma_1(\widetilde{\mathbf{A}}_{(n)})$ for $n=1,2,3$ for the three test functions. Faster decay corresponds to smaller $\Delta_{R_n}$  and hence tighter error bounds. The spectra confirm that: $f_3(x_1,x_2,x_3)$ exhibits the fastest decay, consistent with its additive separable structure; $f_2(x_1,x_2,x_3)$ exhibits moderate decay, typical of smooth functions; and $f_1(x_1,x_2,x_3)$ exhibits the slowest decay due to the localized singularity.
 \begin{figure}[htb]
    \setlength{\tabcolsep}{4pt}
    \renewcommand\arraystretch{1}
    \centering
    \includegraphics[width=0.8\linewidth]{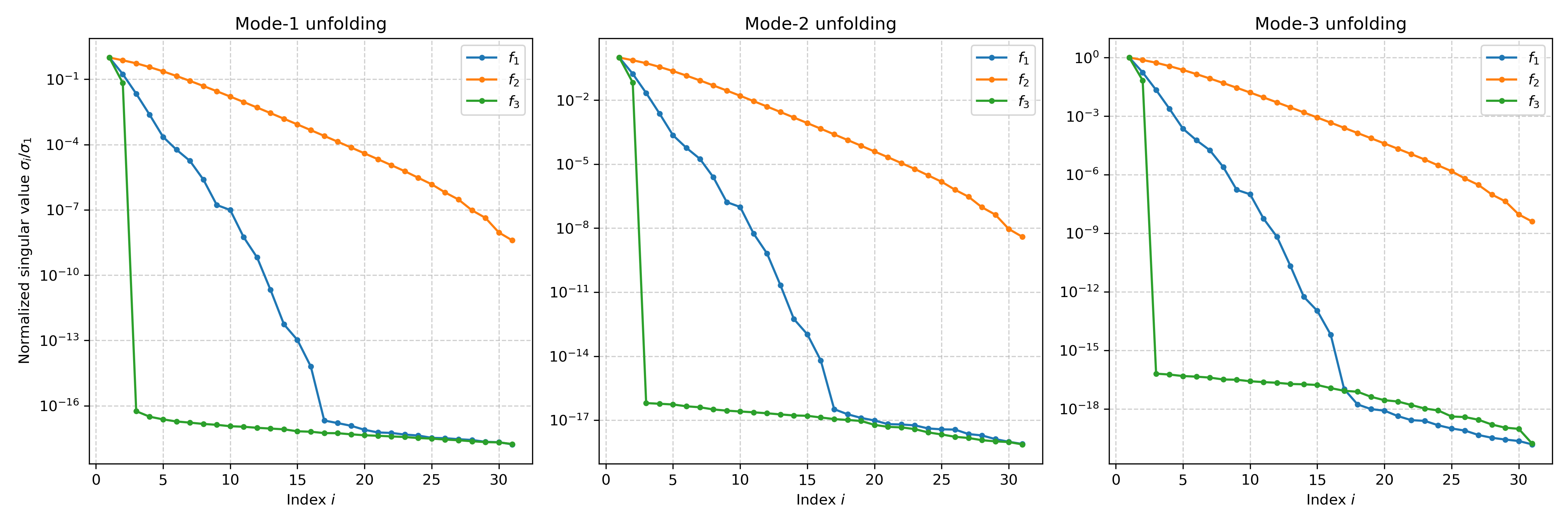}\\
    \caption{Normalized singular value spectra of each mode unfolding for the three test functions.}\label{tensor-hyperinterpolation:fig5-main}
\end{figure}
\subsection{Cross-function performance comparison}
We now synthesize the preceding validation results by comparing the full end-to-end pipeline performance across the three test functions and verifying the composite error bound from Remark \ref{tensor-hyperinterpolation:remapp}. The complete pipeline consists of 
\begin{enumerate}
\item[(1)] {\bf Hyperinterpolation}: choose polynomial degrees $I_n$ based on desired $L^2$ accuracy and function regularity $\alpha$ (Theorem \ref{tensor-hyperinterpolation:general-theorem});
\item[(2)] {\bf TCUR strategy}: Run Algorithm \ref{tensor-hyperinterpolation:alg1}, \ref{tensor-hyperinterpolation:alg2} or \ref{tensor-hyperinterpolation:alg4} with different pairs $(b,\tau)$ to obtain $\widetilde{\mathcal{A}}_{\text{TCUR}}$;
\item[(3)] {\bf Recompression}: Apply Algorithm \ref{t-lasso:alg3} with target Tucker ranks $(R_1,R_2,\dots,R_N)$ to obtain final Tucker approximation $\mathcal{B}$;
\item[(4)] {\bf Evaluation}: Reconstruct function values on the target grid $\mathbb{X}_{\text{target}}$ via $$\mathcal{R}\approx \widetilde{\mathcal{A}}\times_1(\widetilde{\mathbf{A}}_1^\top\mathbf{V}_1)\times_2(\widetilde{\mathbf{A}}_2^\top\mathbf{V}_2)\dots\times_N(\widetilde{\mathbf{A}}_N^\top\mathbf{V}_N).$$
\end{enumerate}

We now set $(I,I,I)=(30,30,30)$, choose $(b,\tau)=(4,0.02)$ for all functions, and let the target Tucker rank $(R_1,R_2,R_3)=(10,10,10)$. The target grid is a uniform $50\times 50\times 50$ grid in $[-1,1]^{3}$. Following Table \ref{tensor-hyperinterpolation:tab10-main}, for the peak function $f_1(x_1,x_2,x_3)$, the hyperinterpolation discretization error dominates the total error. This is expected due to the low Sobolev regularity ($\alpha = 0.5$), which yields slow convergence with $I$; for the smoother functions $f_2(x_1,x_2,x_3)$ and $f_3(x_1,x_2,x_3)$, the TCUR and recompression errors are comparable to or larger than the discretization error, indicating that the approximation is tensor-limited rather than function-limited; and  the total $L^2$ error is well approximated by the sum of the discretization and TCUR errors (with the triangle inequality providing the upper bound), which confirms the additive decomposition used in Theorems \ref{tensor-hyperinterpolation:main-theorem1} and \ref{tensor-hyperinterpolation:first-approximation-theorem}, and Remark \ref{tensor-hyperinterpolation:remapp}.

   \begin{table}[htb]
   \scriptsize
   \centering
   \begin{tabular}{|c|c|c|c|c|c|}
      \hline
      Algorithm & Function & $\|\mathcal{L}f-f\|_{L^2}$ & $\|\widetilde{\mathcal{L}}f-f\|_{L^2}$ (TCUR) & $\|\widetilde{\mathcal{L}}f-f\|_{L^2}$ (total) & Bound (Remark \ref{tensor-hyperinterpolation:remapp}) \\
      \hline
      \multirow{3}{*}{Algorithm \ref{tensor-hyperinterpolation:alg1}} 
      & $f_1(x_1,x_2,x_3)$ & 4.288e-6 & 3.920e-5  & 3.985e-5 & 2.034e4  \\
      \cline{2-6}
      \multirow{3}{*}{} & $f_2(x_1,x_2,x_3)$ & 2.939e-6 & 5.362e-3 & 1.259e-2 & 2.504e2   \\
     \cline{2-6}
      \multirow{3}{*}{} & $f_3(x_1,x_2,x_3)$ & 1.218e-10 &4.182e-15 & 1.218e-10 & 1.981e-10   \\
      \hline
      \multirow{3}{*}{Algorithm \ref{tensor-hyperinterpolation:alg2}} & $f_1(x_1,x_2,x_3)$ & 4.288e-6 & 3.920e-5  & 3.985e-5 & 2.034e4  \\
      \cline{2-6}
      \multirow{3}{*}{} & $f_2(x_1,x_2,x_3)$ & 2.939e-6 & 5.362e-3 & 1.259e-2 & 2.504e2   \\
     \cline{2-6}
      \multirow{3}{*}{} & $f_3(x_1,x_2,x_3)$ & 1.218e-10 & 4.182e-15 & 1.218e-10 & 1.981e-10   \\
      \hline
      \multirow{3}{*}{Algorithm \ref{tensor-hyperinterpolation:alg4}} & $f_1(x_1,x_2,x_3)$ & 4.288e-6 & 3.097e-2  & 2.097e-2 & 9.199e6  \\
      \cline{2-6}
      \multirow{3}{*}{} & $f_2(x_1,x_2,x_3)$ & 2.939e-6 & 5.463e-2 & 5.463e-2 & 2.853e8   \\
     \cline{2-6}
      \multirow{3}{*}{} & $f_3(x_1,x_2,x_3)$ & 1.218e-10 & 1.124e-14 & 1.218e-10 & 8.375e-2   \\
      \hline
      \end{tabular}
   \caption{Given $(I,I,I)=(30,30,30)$, $(R_1,R_2,R_3)=(10,10,10)$ and $(b,\tau)=(4,0.02)$, end-to-end $L^2$ errors and bound components for hyperinterpolation and its approximations of $f_1(x_1,x_2,x_3)$, $f_2(x_1,x_2,x_3)$ and $f_3(x_1,x_2,x_3)$.}
   \label{tensor-hyperinterpolation:tab10-main}
   \end{table}
   
To visually verify the approximation quality, Figure \ref{tensor-hyperinterpolation:fig6-main} compares the full hyperinterpolation polynomial and the final Tucker-compressed approximation for $f_2(x_1,x_2,x_3)$ on a two-dimensional slice through the domain ($x_3=0$). The close agreement confirms that the compressed representation faithfully captures the essential features of the function.

 \begin{figure}[htb]
    \setlength{\tabcolsep}{4pt}
    \renewcommand\arraystretch{1}
    \centering
    \includegraphics[width=0.8\linewidth]{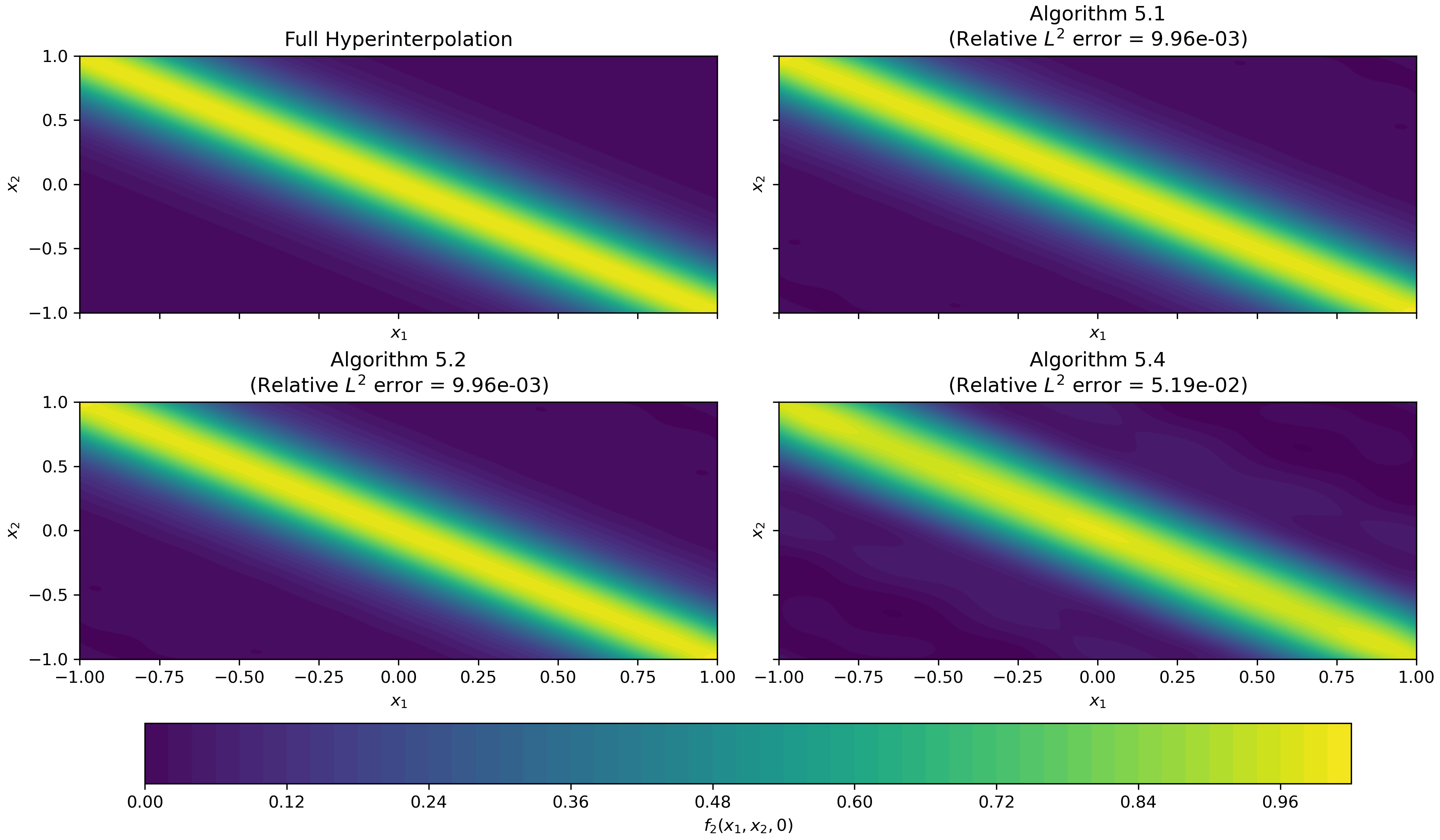}\\
    \caption{Comparison of full hyperinterpolation (left-top) and Tucker-compressed approximation obtained by Algorithm \ref{tensor-hyperinterpolation:alg1} (right-top), \ref{tensor-hyperinterpolation:alg2} (left-bottom) and \ref{tensor-hyperinterpolation:alg4} (right-bottom) for $f_2(x_1,x_2,x_3)$ on the slice $x_3=0$.}\label{tensor-hyperinterpolation:fig6-main}
\end{figure}
\section{Conclusions}
\label{tensor-hyperinterpolation:sect7-main}
This work departs from the conventional algorithm-first paradigm that dominates existing research on high-dimensional hyperinterpolation, and instead establishes a complete, rigorous structural theoretical foundation to resolve the curse of dimensionality via intrinsic low-rank tensor properties. Prior literature largely designed heuristic compression workflows without answering the fundamental question of why hyperinterpolation coefficient tensors can be efficiently approximated with low multilinear rank; our paper reverses this logic, deriving all practical algorithms as provable corollaries of universal structural theorems rather than standalone empirical engineering tools. We systematically address the four core foundational questions raised in the introduction and deliver four landmark theoretical contributions that unify hyperinterpolation with modern low-rank tensor approximation theory.

First, we prove a universal intrinsic low-$\epsilon$-Tucker-rank existence theorem for hyperinterpolation coefficient tensors defined over the N-dimensional hypercube. This theorem mathematically confirms that such tensors are inherently compressible: their required per-mode Tucker ranks scale only logarithmically with polynomial degrees and ambient dimension, yielding near-dimension-independent entry-wise approximation error bounds. The theoretical rank scaling behavior we derive eliminates the theoretical ambiguity surrounding the compressibility of hyperinterpolation arrays, and provides an a priori tool to estimate minimal storage costs for arbitrary target accuracy before any computation is executed. This structural result forms the bedrock upon which all subsequent TCUR theory and greedy algorithms are built.

Second, we construct a single unified Tucker-compatible framework that subsumes both Chidori-type and Fiber-type tensor CUR (TCUR) decompositions, resolving the prior fragmentation of TCUR analysis where the two variants were treated as separate, unrelated methodologies. We derive tight, fully intrinsic Frobenius-norm residual bounds for the unified TCUR approximation, with all stability constants and residual terms dependent solely on the singular spectrum of tensor mode unfoldings and the geometric volume of selected index subsets. No distributional randomness or artificial tuning parameters appear within our error estimates, enabling fully deterministic, predictable error budgeting for hyperinterpolation compression pipelines. Our unified bound also exposes a critical numerical caveat: sequential mode-wise index selection risks multiplicative error amplification, which mandates joint global optimization of all tensor index sets to maintain stability in ultrahigh dimensions.

Third, we rigorously justify the widely adopted greedy adaptive index selection schemes as theoretically sound, near-optimal procedures rather than ad-hoc heuristics. We design three distinct tensor-free greedy TCUR algorithms equipped with singular-value-based certified stopping criteria, and quantify their total computational cost by counting the number of expensive multidimensional cubature inner products required at each iteration. All proposed algorithms avoid full assembly of the exponentially large hyperinterpolation coefficient tensor at every step of execution, eliminating the core computational bottleneck of the curse of dimensionality. The iterative index expansion strategy we develop targets maximum-volume submatrices from tensor singular factorizations, which minimizes pseudoinverse conditioning constants and suppresses residual blowup in high-dimensional settings, fully aligning with the stability bounds established in our TCUR error analysis.

Fourth, we assemble an integrated end-to-end error chain that connects four successive layers of approximation loss: Sobolev regularity-driven hyperinterpolation discretization error, TCUR greedy subsampling residual, and secondary truncation error from the lightweight TCUR-to-Tucker recompression pipeline provided in the appendix. The composite closed-form \(L^2\) function-space error bound decouples each source of approximation loss, allowing practitioners to independently allocate allowable error budgets to polynomial truncation, TCUR index sampling, and final Tucker rank reduction. This unified error framework enables rigorous a priori parameter tuning for real-world scientific computing workflows including surrogate modeling and uncertainty quantification.Collectively, this body of work recontextualizes high-dimensional hyperinterpolation as a special case of structured low-rank tensor approximation governed by provable multilinear structural laws, rather than an isolated spectral approximation technique limited by exponential scaling. By centering intrinsic tensor structure before algorithm design, we bridge the gap between classical multivariate approximation theory and contemporary tensor numerical linear algebra, offering a replicable blueprint for structural analysis of other spectral, polynomial, and cubature-based high-dimensional approximation methods.

This paper opens multiple promising extensions to expand the scope of the proposed structural and algorithmic framework:
(1) to extend the intrinsic low-$\epsilon$-Tucker-rank theory and unified TCUR decomposition to non-hypercube geometries such as spheres, balls, and irregular bounded domains, where hyperinterpolation with tailored cubature rules is widely deployed; (2) to develop randomized adaptive index selection algorithms as stochastic counterparts to our deterministic greedy schemes, with probabilistic error bounds and reduced per-iteration cubature cost for ultra-large sampling grids; (3) to merge our tensor-free TCUR compression pipeline with sparse grid hyperinterpolation to combine the dimensionality reduction strengths of both methods, targeting extreme dimensional problems where even moderate Tucker ranks remain computationally heavy; (4) to generalize our full error chain analysis to weighted Sobolev spaces and anisotropic polynomial truncation schemes, adapting the low-rank rank bounds for functions with variable smoothness across different coordinate axes; (5) to derive rigorous minimax lower bounds on the $\epsilon$-Tucker rank of hyperinterpolation coefficient tensors to fully certify the near-optimality of our rank upper estimates and greedy index selection strategies; (6) to extend the TCUR-to-Tucker recompression and composite error analysis to Lasso regularized hyperinterpolation for noisy sampled data, quantifying how measurement noise propagates through the full low-rank compression pipeline; and (7) to parallelize the proposed greedy TCUR algorithms and TCUR-to-Tucker recompression workflow, developing distributed tensor algebra routines to scale the method to industrial-scale high-dimensional modeling tasks.

\appendix
\renewcommand\thetable{\Alph{section}.\arabic{algorithm}}
\setcounter{algorithm}{0}
\section{Re-compressing the approximation of the coefficient tensor}
\label{tensor-hyperinterpolation:app-main}

In Section \ref{tensor-hyperinterpolation:sect5-main}, we derived three families of greedy adaptive TCUR algorithms that construct Chidori-type and Fiber-type tensor CUR factorizations without explicitly assembling the full hyperinterpolation coefficient tensor \(\widetilde{\mathcal{A}}\). The output \(\widetilde{\mathcal{A}}_{\rm TCUR}\) carries an intermediate multilinear rank tuple \(\{S_1,\dots,S_N\}\) satisfying \(R_n < S_n \ll I_n+1\), where \(\{R_1,\dots,R_N\}\) denotes the target minimal Tucker rank for storage-optimal compressed approximation. This appendix develops a lightweight, numerically stable TCUR-to-Tucker recompression pipeline to further truncate \(\widetilde{\mathcal{A}}_{\rm TCUR}\) down to the prescribed low Tucker rank \(\{R_1,\dots,R_N\}\). The transformation requires no additional expensive cubature inner product evaluations on the original function samples and only relies on linear algebra operations over the small TCUR index subsets, closing the complete end-to-end approximation chain from raw function data to optimally compressed hyperinterpolation surrogates. The unified conversion workflow is formalized in Algorithm \ref{t-lasso:alg3}.

\begin{algorithm}[htb]
     \caption{Efficient conversion from TCUR to approximate Tucker decomposition}
     \begin{algorithmic}[1]
        \STATEx {\bf Input}: The $(N+1)$-tuple $\{\mathcal{G};\mathbf{C}_1\mathbf{U}_1^\dag,\dots,\mathbf{C}_N\mathbf{U}_N^\dag\}$ in (\ref{tensor-hyperinterpolation:approximate-one}), and target low Tucker rank $\{R_1,\dots,R_N\}$ with $R_n<S_n$.
        \STATEx {\bf Output}: The core tensor $\mathcal{S}\in\mathbb{R}^{R_1\times R_2\times \dots\times R_N}$ and the mode-$n$ factor matrix $\mathbf{V}_n\in\mathbb{R}^{I_n\times R_n}$ such that $\widetilde{\mathcal{A}}\approx\mathcal{S}\times_1\mathbf{V}_1\dots\times_N\mathbf{V}_N$.
        \STATE Compute an economy-size QR factorization $\mathbf{C}_n\mathbf{U}_n^\dag=\mathbf{Q}_{n}\mathbf{R}_{n}$ with $n=1,2,\dots,N$.\label{t-lasso:alg2:step2}
        \STATE Form $\widehat{\mathcal{G}}=\mathcal{G}\times_1\mathbf{R}_1\dots\times_N\mathbf{R}_N$.
        \STATE Compute $\{\mathcal{S},\widehat{\mathbf{V}}_1,\dots,\widehat{\mathbf{V}}_N\}=\text{ST-HOSVD}(\widehat{\mathcal{G}};R_1,\dots,R_N)$.
        \label{t-lasso:alg2:step1}
        \STATE Form $\mathbf{V}_n=\mathbf{Q}_{n}\widehat{\mathbf{V}}_n$ with $n=1,2,\dots,N$.
    \end{algorithmic}
    \label{t-lasso:alg3}
\end{algorithm}
\begin{remark}
Step \ref{t-lasso:alg2:step1} can substitute ST-HOSVD with T-HOSVD without altering the pipeline structure. T-HOSVD delivers slightly looser theoretical error bounds while ST-HOSVD yields tighter residual control in practical high-dimensional hyperinterpolation tests. All subsequent error analysis holds for both variants with minor adjustments to singular value summation upper limits.
\end{remark}

Now we consider the upper bound for $\|\mathcal{S}\times_1\mathbf{V}_1\dots\times_N\mathbf{V}_N-\widetilde{\mathcal{A}}\|_F$, which is stated in the following theorem.
\begin{theorem}
    Let $\widetilde{\mathcal{A}}$ be given in {\rm(\ref{tensor-hyperinterpolation:approximation-coefficient})} and $\widetilde{\mathcal{A}}_{{\rm TCUR}}$ be given in {\rm(\ref{tensor-hyperinterpolation:approximate-one})}. For a Tucker rank $\{R_1,\dots,R_N\}$ with $R_n\leq I_n+1$ and $=1,2,\dots,N$, let $\{\mathcal{S};\mathbf{V}_1,\dots,\mathbf{V}_N\}$ be obtained by applying Algorithm {\rm \ref{t-lasso:alg3}} to $\widetilde{\mathcal{A}}_{{\rm TCUR}}$. Then, one has
    \begin{equation*}
        \begin{aligned}
            \|\mathcal{S}\times_1\mathbf{V}_1\dots\times_N\mathbf{V}_N-\widetilde{\mathcal{A}}\|_F\leq \|\widetilde{\mathcal{A}}_{\rm TCUR}-\widetilde{\mathcal{A}}\|_F+\prod_{n=1}^N\|\mathbf{C}_n\mathbf{U}_n\|_2\sqrt{\sum_{n=1}^N\sum_{i=R_{n}+1}^{\widehat{I}_n}\sigma_{i}(\widetilde{\mathbf{A}}_{(n)})^2},
        \end{aligned}
    \end{equation*}
    with $\widehat{I}_n=\min\{I_n+1,(I_1+1)\dots(I_{n-1}+1)(I_{n+1}+1)\dots (I_N+1)\}$.
    \label{tensor-hyperinterpolation:app-thm}
\end{theorem}
\begin{proof}
    Note that we have
    \begin{equation*}
        \|\widetilde{\mathcal{A}}-\mathcal{S}\times_1\mathbf{V}_1\dots\times_N\mathbf{V}_N\|_F\leq \|\widetilde{\mathcal{A}}_{\rm TCUR}-\widetilde{\mathcal{A}}\|_F+\|\widetilde{\mathcal{A}}_{\rm TCUR}-\mathcal{S}\times_1\mathbf{V}_1\dots\times_N\mathbf{V}_N\|_F.
    \end{equation*}
    Since $ \widetilde{\mathcal{A}}_{{\rm TCUR}}=\mathcal{G}\times_1(\mathbf{C}_1\mathbf{U}_1^\dag)
   \times_2\dots\times_N(\mathbf{C}_N\mathbf{U}_N^\dag)$, one has
    \begin{align*}
        &\mathcal{G}\times_1(\mathbf{C}_1\mathbf{U}_1^\dag)
   \times_2\dots\times_N(\mathbf{C}_N\mathbf{U}_N^\dag)-\mathcal{S}\times_1\mathbf{V}_1\dots\times_N\mathbf{V}_N\\
        &=\mathcal{G}\times_1(\mathbf{Q}_1\mathbf{R}_1)\dots\times_N(\mathbf{Q}_N\mathbf{R}_N)-\mathcal{S}\times_1(\mathbf{Q}_1\widehat{\mathbf{V}}_1)\dots\times_N(\mathbf{Q}_N\widehat{\mathbf{V}}_N)\\
        &=(\mathcal{G}\times_1\mathbf{R}_1\dots\times_N\mathbf{R}_N-\mathcal{S}\times_1\widehat{\mathbf{V}}_1\dots\times_N\widehat{\mathbf{V}}_N)\times_1\mathbf{Q}_1\dots\times_N\mathbf{Q}_N\\
        &=(\widehat{\mathcal{G}}-\mathcal{S}\times_1\widehat{\mathbf{V}}_1\dots\times_N\widehat{\mathbf{V}}_N)\times_1\mathbf{Q}_1\dots\times_N\mathbf{Q}_N
    \end{align*}
    with $\widehat{\mathcal{G}}=\mathcal{G}\times_1\mathbf{R}_1\dots\times_N\mathbf{R}_N$, where $\mathbf{Q}_n$ and $\mathbf{R}_n$ are given in Step \ref{t-lasso:alg2:step2} of Algorithm \ref{t-lasso:alg3}, which implies that
    \begin{align*}
        \|\widetilde{\mathcal{A}}_{\rm TCUR}-\mathcal{S}\times_1\mathbf{V}_1\dots\times_N\mathbf{V}_N\|_F^2&\leq\|\widehat{\mathcal{G}}-\mathcal{S}\times_1\widehat{\mathbf{V}}_1\dots\times_N\widehat{\mathbf{V}}_N\|_F^2\\
        &\leq\sum_{n=1}^N\sum_{i=R_{n}+1}^{\widehat{S}_n}\sigma_{i}(\widehat{\mathbf{G}}_{(n)})^2
    \end{align*}
    with $\widehat{S}_n=\min\{S_n,S_1\dots S_{n-1}S_{n+1}\dots S_N\}$, where the second inequality holds according to Theorem \ref{tensor-hyperinterpolation:hosvd-approximation}.

    For each $n$, one has $\widehat{\mathbf{G}}_{(n)}=\mathbf{R}_n\mathbf{G}_{(n)}(\mathbf{R}_1\otimes\dots\otimes\mathbf{R}_{n-1}\otimes\mathbf{R}_{n+1}\otimes\dots\otimes\mathbf{R}_N)^\top$. According to Lemma 3.9 in \cite{martinsson2011randomized}, we have
    \begin{equation*}
        \sigma_i(\widehat{\mathbf{G}}_{(n)})\leq \prod_{n=1}^N\|\mathbf{R}_n\|_2\sigma_i(\mathbf{G}_{(n)})=\prod_{n=1}^N\|\mathbf{C}_n\mathbf{U}_n\|_2\sigma_i(\mathbf{G}_{(n)})
    \end{equation*}
    with $i=1,2,\dots,\widehat{S}_n$. Since $\mathcal{G}=\mathcal{A}(\mathbb{I}_1,\dots,\mathbb{I}_N)$, one has
    \begin{equation*}
        \sum_{n=1}^N\sum_{i=R_{n}+1}^{\widehat{S}_n}\sigma_{i}(\widehat{\mathbf{G}}_{(n)})^2\leq \sum_{n=1}^N\sum_{i=R_{n}+1}^{\widehat{I}_n}\sigma_{i}(\mathbf{A}_{(n)})^2.
    \end{equation*}
    Hence, the proof is complete.
\end{proof}
\begin{remark}
Combining Theorem \ref{tensor-hyperinterpolation:app-thm} with Theorem \ref{tensor-hyperinterpolation:first-approximation-theorem} yields a single unified \(L^2\) error bound covering the full modeling pipeline: Sobolev function regularity, hyperinterpolation discretization, greedy TCUR sampling and TCUR-to-Tucker recompression. The complete global bound takes the form:
\begin{align*}
&\|\widetilde{\mathcal{L}}_{I_1,\dots, I_N} f (\mathbf{x})-f (\mathbf{x})\|_{L^2}\leq C\cdot\min\{I_1,\dots,I_N\}^{-\alpha}|f (\mathbf{x})|_{\alpha}\\
        &\quad+\sum_{n=1}^N\left(\prod_{m=1}^{n-1}\|\mathbf{C}_m\mathbf{U}_m^\dag\|_2\right)
        \left(\alpha_{n,R_n}\Delta_{R_n}+\beta_{n,R_n}\|\mathbf{U}_n^\dag\|_2\Delta_{R_n}^2\right)+ \prod_{n=1}^N \|{\bf C}_n {\bf U}_n^\dagger\|_2  \sqrt{\sum_{n=1}^N \sum_{i=R_n+1}^{\hat{I}_n}\sigma_i(\widetilde{{\bf A}}_{(n)})^2}.
\end{align*}

The three additive terms correspond sequentially to: (1) hyperinterpolation discretization error controlled by the Sobolev smoothness \(\alpha\) of $f (\mathbf{x})$; (2) TCUR greedy sampling compression error; and (3) secondary Tucker truncation residual from the Appendix recompression step. This single closed-form bound fully quantifies all sources of approximation error in the proposed structural framework and serves as the core theoretical tool for selecting polynomial degrees \(I_n\), TCUR block sizes \(b_n\), and target Tucker ranks \(R_n\) for user-specified global accuracy requirements in scientific computing, uncertainty quantification, and high-dimensional surrogate modeling applications.
\label{tensor-hyperinterpolation:remapp}
\end{remark}

This composite error decomposition carries two critical practical interpretations for high dimensional hyperinterpolation workflows. First, the first term \(\|\widetilde{\mathcal{A}}_{\rm TCUR}-\widetilde{\mathcal{A}}\|_F\) encapsulates all approximation error introduced during the greedy TCUR index selection stage, which is fully controlled by the tolerance parameters \(\tau_1,\tau_2\) used in Algorithms \ref{tensor-hyperinterpolation:alg1}-\ref{tensor-hyperinterpolation:alg4} and the size of the selected index subsets \(\mathbb{I}_n,\mathbb{J}_n\). The second term quantifies the incremental truncation error incurred when compressing the overcomplete TCUR representation down to the minimal target Tucker rank \( (R_1,R_2,\dots,R_N)\); this residual only depends on the tail singular values of the original hyperinterpolation coefficient tensor, independent of the greedy sampling strategy used to build the TCUR factorization. Together, the two additive error terms decouple the sampling error from the low-rank truncation error, enabling separate tuning of TCUR block sizes \(b_n,b_n'\) and target Tucker ranks \(R_n\) to meet a prescribed global \(L^2\) approximation tolerance for the final surrogate function. Additionally, the product factor \(\prod_{n=1}^N\|{\bf C}_n {\bf U}_n^\dagger\|_2\) acts as a stability constant for the recompression pipeline: when all selected index sets \(\mathbb{I}_n,\mathbb{J}_n\) are maximum-volume submatrices (see Remark \ref{tensor-hyperinterpolation:remark1}), this product remains bounded by a mild polynomial function of \(R_n\) and \(S_n\), eliminating exponential error blowup in high ambient dimension $N$, i.e., a key theoretical guarantee missing from naive two-stage TCUR-then-Tucker pipelines without unified stability analysis. Finally, this bound enables rigorous a priori error budgeting: practitioners can preallocate a fraction of the total allowed residual to TCUR sampling and the remainder to Tucker recompression, then solve for minimal \(S_n\) and \(R_n\) that satisfy the combined tolerance constraint without repeated costly cubature evaluations.

Several practical implementation notes for Algorithm \ref{t-lasso:alg3} are listed as follows:
\begin{enumerate}
\item[(a)] All intermediate tensors \(\widehat{\mathcal{G}}\) and factor matrices \({\bf R}_n\) live on the small TCUR index dimension \(\prod_{n=1}^N S_n\), which is exponentially smaller than the full coefficient tensor dimension \(\prod_{n=1}^N(I_n+1)\). No full unfolding or storage of \(\widetilde{\mathcal{A}}\) is required at any recompression step, preserving the ``tensor-free'' property of the overall framework.
\item[(b)] In numerical practice, one first selects TCUR block sizes \(b_n\) to drive \(\|\widetilde{\mathcal{A}}_{\rm TCUR}-\widetilde{\mathcal{A}}\|_F\) below a preallocated error threshold, then runs Algorithm \ref{t-lasso:alg3} with incrementally decreasing target ranks \(R_n\) until the second singular-value term falls below the remaining error budget. Figure \ref{tensor-hyperinterpolation:fig1-main} from Section \ref{tensor-hyperinterpolation:sect4-main} can be reused to pre-estimate the minimal \(R_n\) required for a given global tolerance \(\epsilon\).
\item[(c)] The QR factorizations in Step 1 of Algorithm \ref{t-lasso:alg3} are fully separable across modes $n$, allowing distributed parallel computation for large $N$ without inter-mode communication overhead during the recompression stage.
\item[(d)] For Lasso hyperinterpolation with noisy sampled data (see \cite{an2021lasso}), the recompression pipeline remains unchanged; the composite error bound simply absorbs measurement noise into the \(\|\widetilde{\mathcal{A}}_{\rm TCUR}-\widetilde{\mathcal{A}}\|_F\) term, with the singular-value truncation bound unaffected by additive data noise.
\item[(e)] After computing the final Tucker factorization \(\{\mathcal{S};\widehat{{\bf V}}_1,\widehat{{\bf V}}_2,\dots,\widehat{{\bf V}}_N\}\), reconstructed function values over arbitrary target grids \(\mathbb{X}_{target}\) are computed via the compact multilinear product \(\mathcal{R} = \mathcal{S} \times_1 (\widetilde{{\bf A}}_1^\top \widehat{{\bf V}}_1) \times_2 (\widetilde{{\bf A}}_2^\top \widehat{{\bf V}}_2)\dots \times_N (\widetilde{{\bf A}}_N^\top \widehat{{\bf V}}_N)\), where \(\widetilde{{\bf A}}_n\) denotes basis matrices evaluated at target nodes as defined in Section \ref{tensor-hyperinterpolation:sect3-sub3}. This evaluation step avoids all large-scale tensor products involving the original polynomial degrees \(I_n\), delivering exponential speedups over direct evaluation of the uncompressed hyperinterpolation expansion.
\end{enumerate}

\subsection{TCUR-to-Tucker recompression performance}
We now validate the TCUR-to-Tucker recompression pipeline (see Algorithm \ref{t-lasso:alg3}), which converts the overcomplete TCUR representation (with ranks $S_n>R_n$) into a minimal Tucker format with target ranks $R_n$. Our objectives are: (i) to verify the error bound from Theorem \ref{tensor-hyperinterpolation:app-thm}; (ii) to quantify the additional compression ratio achieved; and (iii) to assess the numerical stability of the QR-based conversion.

For clarity, the $(N+1)$-tuple $\{\mathcal{G};\mathbf{C}_1\mathbf{U}_1^\dag,\mathbf{C}_2\mathbf{U}_2^\dag,\dots,\mathbf{C}_N\mathbf{U}_N^\dag\}$, used in Algorithm \ref{t-lasso:alg3}, is obtained by applying Algorithm \ref{tensor-hyperinterpolation:alg2} with $(b,\tau)=(4,0.02)$ to each test function (see Section \ref{tensor-hyperinterpolation:sect6-sub1}). By applying Algorithm \ref{tensor-hyperinterpolation:alg2} with $(b,\tau)=(4,0.02)$ to $f_2(x_1,x_2,x_3)$, the size of the TCUR core is $(S_1,S_2,S_3)=(8,8,8)$. Hence, in Algorithm \ref{t-lasso:alg3}, we assume that $(R_1,R_2,R_3)=(R,R,R)$ with fix $R\in\{2,4,6\}$ and $(S_1,S_2,S_3)$ at the TCUR output size. For each test function, we measure the Frobenius error $E_{\text{exact}}=\|\widetilde{\mathcal{A}}-\mathcal{B}\|_F$, the theoretical bound $E_{\text{bound}}$ from Theorem \ref{tensor-hyperinterpolation:app-thm} and the storage reduction $S_1S_2S_3/R_1R_2R_3$, where $\mathcal{B}=\mathcal{S}\times_1\mathbf{V}_1\times_2\mathbf{V}_2\dots\times_N\mathbf{V}_N$ and
\begin{equation*}
E_{\text{bound}}= \|\widetilde{\mathcal{A}}_{\rm TCUR}-\widetilde{\mathcal{A}}\|_F+\prod_{n=1}^N\|\mathbf{C}_n\mathbf{U}_n\|_2\sqrt{\sum_{n=1}^N\sum_{i=R_{n}+1}^{\widehat{I}_n}\sigma_{i}(\widetilde{\mathbf{A}}_{(n)})^2}.
\end{equation*}

The TCUR-to-Tucker recompression performance for $f_2(x_1,x_2,x_3)$ with given $(b,\tau)$ is illustrated in Table \ref{tensor-hyperinterpolation:tab1-app}, which implies that: (a) the ratio $E_{\text{bound}}/E_{\text{exact}}$ is approximately 5.0 across all $(R_1,R_2,R_3)$, which confirms that the additive decomposition in Theorem \ref{tensor-hyperinterpolation:app-thm} accurately captures the two sources of error: TCUR sampling and Tucker truncation; and (b) the composite error bound from Remark \ref{tensor-hyperinterpolation:remapp} can be used to select $R$ a priori.

\begin{table}[htb]
   \scriptsize
   \centering
   \begin{tabular}{|c|c|c|c|c|c|}
      \hline
      $(R_1,R_2,R_3)$ & $E_{\text{exact}}$ &  $E_{\text{TCUR}}$ & $E_{\text{bound}}$ & Ratio (bound/exact) & Storage reduction \\
      \hline
      (2,2,2) & 6.898e-1 & 4.537e-2 & 3.653 & 5.30 & 64    \\
      \hline
      (4,4,4) & 3.046e-1 & 4.537e-2 & 1.514 & 4.97 & 8     \\
      \hline
      (6,6,6) & 1.133e-1 & 4.537e-2 & 5.691e-1 & 5.02 & 2.4     \\
      \hline
   \end{tabular}
   \caption{Given $(I,I,I)=(30,30,30)$, TCUR-to-Tucker recompression performance for $f_2(x_1,x_2,x_3)$ with $b=4$, $\tau=0.02$ and $(S_1,S_2,S_3)=(8,8,8)$.}
   \label{tensor-hyperinterpolation:tab1-app}
   \end{table}
   
To illustrate the additive nature of the recompression error, Figure \ref{tensor-hyperinterpolation:fig1-app} plots the two components of $E_{\text{bound}}$: the TCUR residual $\|\widetilde{\mathcal{A}}_{\rm TCUR}-\widetilde{\mathcal{A}}\|_F$ (constant across $R$) and the truncation term $\prod_{n=1}^N\|\mathbf{C}_n\mathbf{U}_n\|_2\sqrt{\sum_{n=1}^N\sum_{i=R_{n}+1}^{\widehat{I}_n}\sigma_{i}(\widetilde{\mathbf{A}}_{(n)})^2}$ (decreasing with $R$). The total bound is the sum, and the actual error tracks the sum closely, validating the additive decomposition.
   
    \begin{figure}[htb]
    \setlength{\tabcolsep}{4pt}
    \renewcommand\arraystretch{1}
    \centering
    \includegraphics[width=0.8\linewidth]{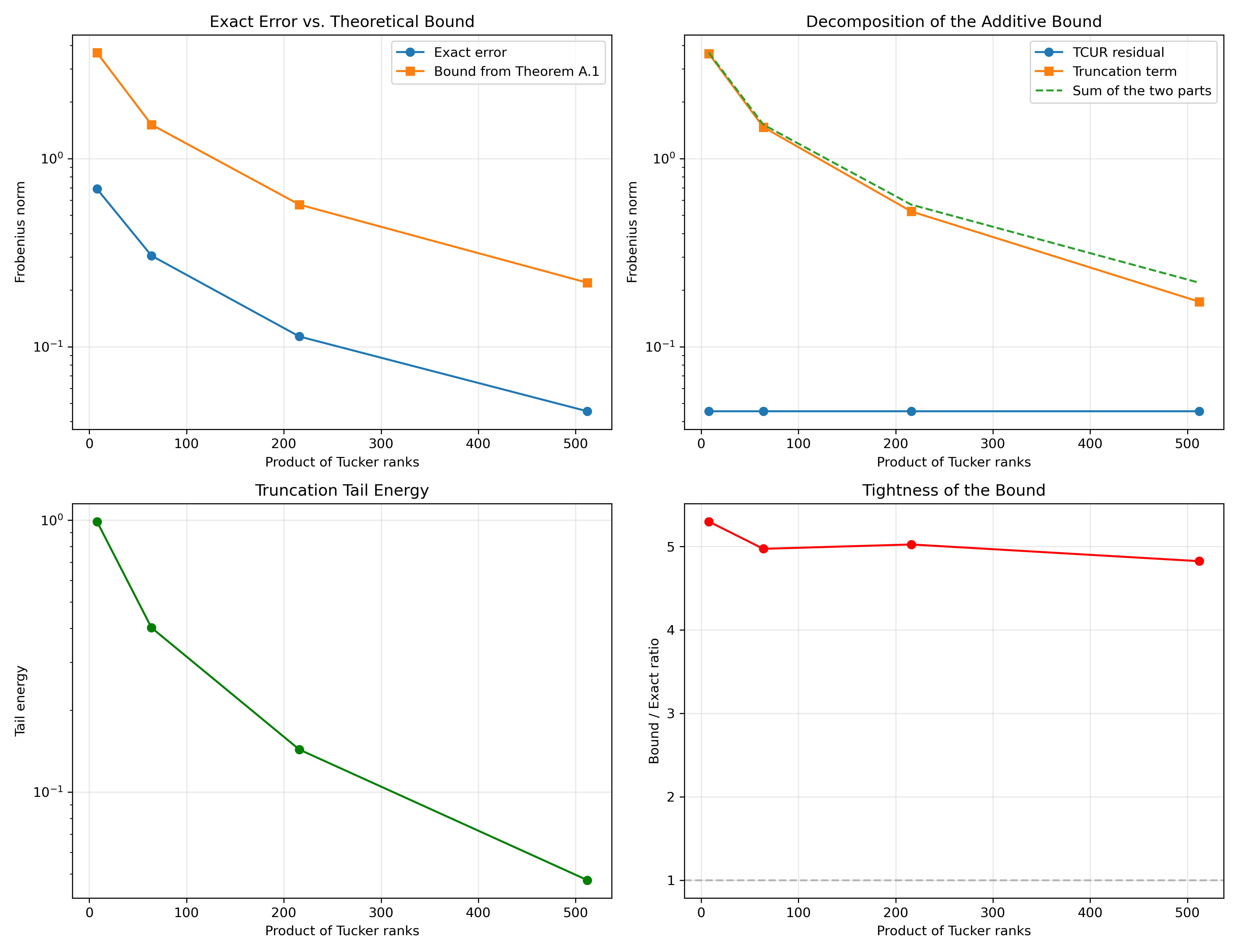}\\
    \caption{Decomposition of recompression error into TCUR sampling residual and Tucker truncation residual. Hint, ``Tail energy'' means $\sqrt{\sum_{n=1}^N\sum_{i=R_{n}+1}^{\widehat{I}_n}\sigma_{i}(\widetilde{\mathbf{A}}_{(n)})^2}$}\label{tensor-hyperinterpolation:fig1-app}
\end{figure}



{\small
\bibliographystyle{siam}
\bibliography{tensor-hyperinterpolation}

@article{battaglino2018a,
  author = {Battaglino, Casey and Ballard, Grey and Kolda, Tamara G.},
  title = {A Practical Randomized {CP} Tensor Decomposition},
  journal = {SIAM Journal on Matrix Analysis and Applications},
  volume = {39},
  number = {2},
  pages = {876-901},
  year = {2018}
}

@article {che2025how,
  title={How many integrals should be evaluated at least in two-dimensional hyperinterpolation?},
  author={Che, Maolin and An, Congpei and Wei, Yimin and Yan, Hong},
  journal={arXiv preprint arXiv:2510.13204},
  year={2025}
}

@article{sloan1995hyperinterpolation,
  title={Polynomial interpolation and hyperinterpolation over general regions},
  author={Sloan, Ian H},
  journal={Journal of Approximation Theory},
  volume={83},
  number={2},
  pages={238--254},
  year={1995},
  publisher={Elsevier}
}

@article{osinsky2018pseudo,
  title={Pseudo-skeleton approximations with better accuracy estimates},
  author={Osinsky, AI and Zamarashkin, Nikolai L},
  journal={Linear Algebra and its Applications},
  volume={537},
  pages={221--249},
  year={2018},
  publisher={Elsevier}
}

@book{reimer2012multivariate,
  title={Multivariate {P}olynomial {A}pproximation},
  author={Reimer, Manfred},
  volume={144},
  year={2012},
  publisher={Birkh{\"a}user}
}

@article{udell2019big,
  title={Why are big data matrices approximately low rank?},
  author={Udell, Madeleine and Townsend, Alex},
  journal={SIAM Journal on Mathematics of Data Science},
  volume={1},
  number={1},
  pages={144--160},
  year={2019},
  publisher={SIAM}
}

@article{beckermann2019bounds,
  title={Bounds on the singular values of matrices with displacement structure},
  author={Beckermann, Bernhard and Townsend, Alex},
  journal={SIAM Review},
  volume={61},
  number={2},
  pages={319--344},
  year={2019},
  publisher={SIAM}
}

@article{trefethen2017multivariate,
  title={Multivariate polynomial approximation in the hypercube},
  author={Trefethen, Lloyd},
  journal={Proceedings of the American Mathematical Society},
  volume={145},
  number={11},
  pages={4837--4844},
  year={2017}
}

@article{mason1980near,
  title={Near-best multivariate approximation by {F}ourier series, {C}hebyshev series and {C}hebyshev interpolation},
  author={Mason, John Charles},
  journal={Journal of Approximation Theory},
  volume={28},
  number={4},
  pages={349--358},
  year={1980},
  publisher={Elsevier}
}

@article{dolgov2021functional,
  title={Functional {T}ucker approximation using {C}hebyshev interpolation},
  author={Dolgov, Sergey and Kressner, Daniel and Str\"{o}ssner, Christoph},
  journal={SIAM Journal on Scientific Computing},
  volume={43},
  number={3},
  pages={A2190--A2210},
  year={2021},
  publisher={SIAM}
}

@article{hashemi2017chebfun,
  title={Chebfun in three dimensions},
  author={Hashemi, Behnam and Trefethen, Lloyd N},
  journal={SIAM Journal on Scientific Computing},
  volume={39},
  number={5},
  pages={C341--C363},
  year={2017},
  publisher={SIAM}
}

@article{an2021lasso,
  title={Lasso hyperinterpolation over general regions},
  author={An, Congpei and Wu, Hao-Ning},
  journal={SIAM Journal on Scientific Computing},
  volume={43},
  number={6},
  pages={A3967--A3991},
  year={2021}
}

@article{hamm2021perturbations,
  title={Perturbations of {CUR} decompositions},
  author={Hamm, Keaton and Huang, Longxiu},
  journal={SIAM Journal on Matrix Analysis and Applications},
  volume={42},
  number={1},
  pages={351--375},
  year={2021},
  publisher={SIAM}
}

@article{che2025efficient-siam,
  title={Efficient Randomized Algorithms for Fixed Precision Problem of Approximate Tucker Decomposition},
  author={Che, Maolin and Wei, Yimin and Yan, Hong},
  journal={SIAM Journal on Matrix Analysis and Applications},
  volume={46},
  number={1},
  pages={256--297},
  year={2025},
  publisher={SIAM}
}

@article{che2019randomized,
  title={Randomized algorithms for the approximations of {T}ucker and the tensor train decompositions},
  author={Che, Maolin and Wei, Yimin},
  journal={Advances in Computational Mathematics},
  volume={45},
  number={1},
  pages={395--428},
  year={2019}}

@article{halko2011finding,
  title={Finding structure with randomness: Probabilistic algorithms for constructing approximate matrix decompositions},
  author={Halko, Nathan and Martinsson, Per-Gunnar and Tropp, Joel A},
  journal={SIAM Review},
  volume={53},
  number={2},
  pages={217--288},
  year={2011},
  publisher={SIAM}
}

@article{yu2018efficient,
  title={Efficient randomized algorithms for the fixed-precision low-rank matrix approximation},
  author={Yu, Wenjian and Gu, Yu and Li, Yaohang},
  journal={SIAM Journal on Matrix Analysis and Applications},
  volume={39},
  number={3},
  pages={1339--1359},
  year={2018},
  publisher={SIAM}
}

@article{martinsson2016randomized,
  title={A randomized blocked algorithm for efficiently computing rank-revealing factorizations of matrices},
  author={Martinsson, Per-Gunnar and Voronin, Sergey},
  journal={SIAM Journal on Scientific Computing},
  volume={38},
  number={5},
  pages={S485--S507},
  year={2016},
  publisher={SIAM}
}

@article{cheng2005compression,
  title={On the compression of low rank matrices},
  author={Cheng, Hongwei and Gimbutas, Zydrunas and Martinsson, Per-Gunnar and Rokhlin, Vladimir},
  journal={SIAM Journal on Scientific Computing},
  volume={26},
  number={4},
  pages={1389--1404},
  year={2005}
}

@article{eckart1936approximation,
  title={The approximation of one matrix by another of lower rank},
  author={Eckart, Carl and Young, Gale},
  journal={Psychometrika},
  volume={1},
  number={3},
  pages={211--218},
  year={1936}
}

@article{mirsky1960symmetric,
  title={Symmetric gauge functions and unitarily invariant norms},
  author={Mirsky, Leon},
  journal={The quarterly journal of mathematics},
  volume={11},
  number={1},
  pages={50--59},
  year={1960}
}

@book{canuto2006spectral,
  title={{S}pectral {M}ethods: {F}undamentals in {S}ingle {D}omains},
  author={Canuto, C. and Hussaini, M. and Quarteroni, A. and Zang, T.},
  year={2006},
  publisher={Springer, Berlin}
}

@article{martinsson2011randomized,
  title={A randomized algorithm for the decomposition of matrices},
  author={Martinsson, Per-Gunnar and Rokhlin, Vladimir and Tygert, Mark},
  journal={Applied and Computational Harmonic Analysis},
  volume={30},
  number={1},
  pages={47--68},
  year={2011}
}

@article{ben1992volume,
  title={A volume associated with $m\times n$ matrices},
  author={Ben-Israel, Adi},
  journal={Linear {A}lgebra and its {A}pplications},
  volume={167},
  pages={87--111},
  year={1992},
  publisher={Elsevier}
}

@article {An_ran_2025,
    AUTHOR = {An, Congpei and Ran, Jiashu},
     TITLE = {Hard thresholding hyperinterpolation over general regions},
   JOURNAL = {Journal of Scientific Computing},
    VOLUME = {102},
      YEAR = {2025},
    NUMBER = {2},
  note = {article no. 37},
      ISSN = {0885-7474,1573-7691},
       DOI = {10.1007/s10915-024-02754-4},
}

@book{Foucart2013compressing,
  author = {Foucart, S. and Rauhut, H.},
  doi = {https://doi.org/10.1007/978-0-8176-4948-7},
  edition = {1},
  pages = {XVIII, 625},
  publisher = {Birkh{\"a}user New York, NY},
  series = {Applied and Numerical Harmonic Analysis},
  title = {A {M}athematical {I}ntroduction to {C}ompressive {S}ensing},
  year = {2013}
}

@article{an2022exactness,
  title={On the quadrature exactness in hyperinterpolation},
  author={An, Congpei and Wu, Hao-Ning},
  journal={BIT Numerical Mathematics},
  volume={62},
  number={4},
  pages={1899--1919},
  year={2022},
  publisher={Springer}
}

@article{lin2021distributed,
  title={Distributed filtered hyperinterpolation for noisy data on the sphere},
  author={Lin, Shao-Bo and Wang, Yu Guang and Zhou, Ding-Xuan},
  journal={SIAM Journal on Numerical Analysis},
  volume={59},
  number={2},
  pages={634--659},
  year={2021},
  publisher={SIAM}
}

@article{LeGia2001uniform,
  title={The uniform norm of hyperinterpolation on the unit sphere in an arbitrary number of dimensions},
  author={Le Gia, T and Sloan, IH},
  journal={Constructive approximation},
  volume={17},
  number={2},
  pages={249--265},
  year={2001},
  publisher={Springer}
}

@misc{tensortool,
  author = {Brett W. Bader, Tamara G. Kolda and others},
  title = {MATLAB Tensor Toolbox Version 3.6},
  howpublished = {Available online},
  year = {September 28, 2023},
  note={\url{https://www.tensortoolbox.org}}
}

@article{an2024bypassing,
  author = {An, Congpei and Wu, Hao-Ning},
  doi = {https://doi.org/10.1016/j.jco.2023.101789},
  journal = {Journal of Complexity},
  note = {article no. 101789},
  title = {Bypassing the quadrature exactness assumption of hyperinterpolation on the sphere},
  volume = {80},
  year = {2024}
}

@article{dai2006hyperinterpolation,
  title={On generalized hyperinterpolation on the sphere},
  author={Dai, Feng},
  journal={Proceedings of the American Mathematical Society},
  volume={134},
  number={10},
  pages={2931--2941},
  year={2006}
}

@article{Wade2013hyperinterpolation,
  title={On hyperinterpolation on the unit ball},
  author={Wade, Jeremy},
  journal={Journal of Mathematical Analysis and Applications},
  volume={401},
  number={1},
  pages={140--145},
  year={2013},
  publisher={Elsevier}
}

@article{Wang2017needlet,
  title={Fully discrete needlet approximation on the sphere},
  author={Wang, Yu Guang and Le Gia, Quoc T and Sloan, Ian H and Womersley, Robert S},
  journal={Applied and Computational Harmonic Analysis},
  volume={43},
  number={2},
  pages={292--316},
  year={2017},
  publisher={Elsevier}
}

@article{hashemi2025rtsms,
  title={{RTSMS}: Randomized {T}ucker with single-mode sketching},
  author={Hashemi, Behnam and Nakatsukasa, Yuji},
  journal={Electronic Transactions on Numerical Analysis},
  volume={63},
  pages={247--280},
  year={2025}
}

@article{che2025efficient-acom,
  title={Efficient algorithms for {T}ucker decomposition via approximate matrix multiplication},
  author={Che, Maolin and Wei, Yimin and Yan, Hong},
  journal={Advances in Computational Mathematics},
  volume={51},
  number={3},
  note={article no. 20},
  year={2025},
  publisher={Springer}
}

@article{minster2024parallel,
  title={Parallel randomized {T}ucker decomposition algorithms},
  author={Minster, Rachel and Li, Zitong and Ballard, Grey},
  journal={SIAM Journal on Scientific Computing},
  volume={46},
  number={2},
  pages={A1186--A1213},
  year={2024},
  publisher={SIAM}
}

@article{sun2020low,
  title={Low-rank {T}ucker approximation of a tensor from streaming data},
  author={Sun, Yiming and Guo, Yang and Luo, Charlene and Tropp, Joel and Udell, Madeleine},
  journal={SIAM Journal on Mathematics of Data Science},
  volume={2},
  number={4},
  pages={1123--1150},
  year={2020},
  publisher={SIAM}
}

@article{zhou2014decomposition,
  title={Decomposition of big tensors with low multilinear rank},
  author={Zhou, Guoxu and Cichocki, Andrzej and Xie, Shengli},
  journal={arXiv preprint arXiv:1412.1885},
  year={2014}
}

@article{caiafa2010generalizing,
   AUTHOR = {Caiafa, Cesar F. and Cichocki, Andrzej},
   TITLE = {Generalizing the column-row matrix decomposition to multi-way arrays},
   JOURNAL = {Linear Algebra and its Applications},
   VOLUME = {433},
   YEAR = {2010},
   NUMBER = {3},
   PAGES = {557--573}
}

@article{ma2021fast,
  title={Fast and accurate randomized algorithms for low-rank tensor decompositions},
  author={Ma, Linjian and Solomonik, Edgar},
  journal={Advances in Neural Information Processing Systems},
  volume={34},
  pages={24299--24312},
  year={2021}
}

@article{malik2018low,
  title={Low-rank {T}ucker decomposition of large tensors using tensorsketch},
  author={Malik, Osman Asif and Becker, Stephen},
  journal={Advances in Neural Information Processing Systems},
  volume={31},
  url = {https://proceedings.neurips.cc/paper_files/paper/2018/file/45a766fa266ea2ebeb6680fa139d2a3d-Paper.pdf},
  year={2018}
}

@inproceedings{tsourakakis2010mach,
  title={{MACH}: Fast randomized tensor decompositions},
  author={Tsourakakis, Charalampos E},
  booktitle={Proceedings of the 2010 SIAM international conference on data mining},
  pages={689--700},
  year={2010},
  organization={SIAM}
}

@article{fahrbach2021fast,
  title={Fast low-rank tensor decomposition by ridge leverage score sampling},
  author={Fahrbach, Matthew and Ghadiri, Mehrdad and Fu, Thomas},
  journal={arXiv preprint arXiv:2107.10654},
  year={2021}
}

@article{ahmadi2021randomized,
  title={Randomized algorithms for computation of {T}ucker decomposition and higher order {SVD (HOSVD)}},
  author={Ahmadi-Asl, Salman and Abukhovich, Stanislav and Asante-Mensah, Maame G and Cichocki, Andrzej and Phan, Anh Huy and Tanaka, Tohishisa and Oseledets, Ivan},
  journal={IEEE Access},
  volume={9},
  pages={28684--28706},
  year={2021},
  publisher={IEEE}
}

@article{che2020computation,
  title={The computation of low multilinear rank approximations of tensors via power scheme and random projection},
  author={Che, Maolin and Wei, Yimin and Yan, Hong},
  journal={SIAM Journal on Matrix Analysis and Applications},
  volume={41},
  number={2},
  pages={605--636},
  year={2020},
  publisher={SIAM}
}

@article{che2021efficient,
  title={An efficient randomized algorithm for computing the approximate {T}ucker decomposition},
  author={Che, Maolin and Wei, Yimin and Yan, Hong},
  journal={Journal of Scientific Computing},
  volume={88},
  number={2},
  note={article no. 32},
  year={2021},
  publisher={Springer}
}

@article{che2021randomized,
  title={Randomized algorithms for the low multilinear rank approximations of tensors},
  author={Che, Maolin and Wei, Yimin and Yan, Hong},
  journal={Journal of Computational and Applied Mathematics},
  volume={390},
  note={article no. 113380},
  year={2021},
  publisher={Elsevier}
}

@article{kolda2009tensor,
  title={Tensor decompositions and applications},
  author={Kolda, Tamara G and Bader, Brett W},
  journal={SIAM Review},
  volume={51},
  number={3},
  pages={455--500},
  year={2009}
}

@article {delathauwer2000a,
  AUTHOR = {De Lathauwer, Lieven and De Moor, Bart and Vandewalle, Joos},
  TITLE = {A multilinear singular value decomposition},
  JOURNAL = {SIAM Journal on Matrix Analysis and Applications},
  VOLUME = {21},
  YEAR = {2000},
  NUMBER = {4},
  PAGES = {1253--1278}
}

@article{vannieuwenhoven2012new,
  title={A new truncation strategy for the higher-order singular value decomposition},
  author={Vannieuwenhoven, Nick and Vandebril, Raf and Meerbergen, Karl},
  journal={SIAM Journal on Scientific Computing},
  volume={34},
  number={2},
  pages={A1027--A1052},
  year={2012},
  publisher={SIAM}
}

@article{minster2020randomized,
  title={Randomized algorithms for low-rank tensor decompositions in the {T}ucker format},
  author={Minster, Rachel and Saibaba, Arvind K and Kilmer, Misha E},
  journal={SIAM Journal on Mathematics of Data Science},
  volume={2},
  number={1},
  pages={189--215},
  year={2020},
  publisher={SIAM}
}

@article{zhou2012fast,
  title={Fast nonnegative matrix/tensor factorization based on low-rank approximation},
  author={Zhou, Guoxu and Cichocki, Andrzej and Xie, Shengli},
  journal={IEEE Transactions on Signal Processing},
  volume={60},
  number={6},
  pages={2928--2940},
  year={2012}
}

@article {saibaba2016hoid,
  AUTHOR = {Saibaba, Arvind K.},
  TITLE = {H{OID}: higher order interpolatory decomposition for tensors based on {T}ucker representation},
  JOURNAL = {SIAM Journal on Matrix Analysis and Applications},
  VOLUME = {37},
  YEAR = {2016},
  NUMBER = {3},
  PAGES = {1223--1249}
}

@article{che2022perturbations,
  title={Perturbations of the {T}CUR decomposition for tensor valued data in the {T}ucker format},
  author={Che, Maolin and Chen, Juefei and Wei, Yimin},
  journal={Journal of Optimization Theory and Applications},
  volume={194},
  number={3},
  pages={852--877},
  year={2022},
  publisher={Springer}
}

@article {drineas2007a,
  AUTHOR = {Drineas, Petros and Mahoney, Michael W.},
  TITLE = {A randomized algorithm for a tensor-based generalization of the singular value decomposition},
  JOURNAL = {Linear Algebra and its Applications},
  VOLUME = {420},
  YEAR = {2007},
  NUMBER = {2-3},
  PAGES = {553--571}
}

@article{lathauwer2000on,
  title={On the best rank-1 and rank-$(R_1, R_2,\cdots,R_N)$ approximation of higher-order tensors},
  author={De Lathauwer, Lieven and De Moor, Bart and Vandewalle, Joos},
  journal={SIAM Journal on Matrix Analysis and Applications},
  volume={21},
  number={4},
  pages={1324--1342},
  year={2000},
  publisher={SIAM}
}

@article{xiao2024rank,
  title={{RA-HOOI}: Rank-adaptive higher-order orthogonal iteration for the fixed-accuracy low multilinear-rank approximation of tensors},
  author={Xiao, Chuanfu and Yang, Chao},
  journal={Applied Numerical Mathematics},
  volume={201},
  pages={290--300},
  year={2024},
  publisher={Elsevier}
}

@article{cai2021mode,
  title={Mode-wise tensor decompositions: Multi-dimensional generalizations of {CUR} decompositions},
  author={Cai, HanQin and Hamm, Keaton and Huang, Longxiu and Needell, Deanna},
  journal={Journal of Machine Learning Research},
  volume={22},
  number={185},
  pages={1--36},
  year={2021}
}

@article{tucker1966some,
  title={Some mathematical notes on three-mode factor analysis},
  author={Tucker, Ledyard R},
  journal={Psychometrika},
  volume={31},
  number={3},
  pages={279--311},
  year={1966},
  publisher={Springer}
}
}
\end{document}